\documentclass[preprint,1p]{elsarticle}

\usepackage{amssymb,amsmath,amsthm}
\usepackage[numbers]{natbib}
\usepackage{geometry}
\usepackage{fleqn}
\usepackage{graphicx}
\usepackage{txfonts}
\usepackage{hyperref}
\usepackage{enumitem}
\usepackage{adjustbox}
\usepackage{changepage}
\usepackage[toc,page]{appendix}
\usepackage{subcaption}
\usepackage{multirow}
\usepackage{float}

\newtheorem{theorem}{Theorem}
\newtheorem{lemma}{Lemma}

\usepackage{xspace}
\usepackage{Macros/mydef}
\usepackage{Macros/remarks}

\usepackage[dvipsnames]{xcolor}
\newcommand{\cblue}[1]{{\color{black}#1}}
\newcommand{\cred}[1]{{\color{black}#1}}

\newcommand{\cpurple}[1]{{\color{black}#1}}

\begin{document}

\begin{frontmatter}
\title{Monolithic parabolic regularization of the MHD equations and entropy principles\tnoteref{t1}}
\tnotetext[t1]{This research is funded by Swedish Research Council (VR) under grant number 2021-04620.}
\author[1]{Tuan Anh Dao\corref{cor1}}
\ead{tuananh.dao@it.uu.se}
\author[1]{Murtazo Nazarov}
\ead{murtazo.nazarov@it.uu.se}
\cortext[cor1]{Corresponding author}
\address[1]{Department of Information Technology, Uppsala University, Sweden}

\begin{abstract}
We show at the PDE level that the monolithic parabolic regularization of the equations of ideal magnetohydrodynamics (MHD) is compatible with all the generalized entropies, fulfills the minimum entropy principle, and preserves the positivity of density and internal energy. We then numerically investigate this regularization for the MHD equations using continuous finite elements in space and explicit strong stability preserving Runge-Kuta methods in time. The artificial viscosity coefficient of the regularization term is constructed to be proportional to the entropy residual of MHD. It is shown that the method has a high order of accuracy for smooth problems and captures strong shocks and discontinuities accurately for non-smooth problems.
\end{abstract}

\begin{keyword}
MHD \sep artificial viscosity \sep entropy inequalities \sep viscous regularization \sep entropy viscosity
\end{keyword}

\end{frontmatter}

\section{Introduction}\label{Sec:Introduction}
Designing numerical methods which produce physically relevant solutions has been an interesting yet challenging task. In the study of non-conducting fluids, for example such described by the compressible Euler equations, entropy principles, and positivity preserving properties have been studied for many schemes. Those include Godunov scheme \citep{Godunov_1959}, Lax scheme \citep{Lax_1954}, and its variants, see e.g., \citep{Perthame_1996,Tang_2000,Tadmor_1986}. Positivity preserving properties often refer to the fact that density, internal energy, or pressure of the solution should maintain positive as it evolves in time.
\cpurple{ These properties are also required when solving the MHD equations. The ideal MHD equations describe the coupling of hydrodynamic motion of plasma fluid with very little resistivity and its electromagnetic field.
}
Existing works on the compressible Euler equations are difficult to extend to the MHD system. Several reasons are: ($i$) the additional consideration of electromagnetism leads to complex hyperbolicity of the MHD system; ($ii$) the MHD flux is non-convex; ($iii$) the solenoidal nature of the magnetic field is numerically important to maintain but it is not explicitly described by the MHD equations. A common approach to achieving positivity and entropy principles for hyperbolic systems consists of adding a vanishing viscous regularization, see e.g., \citep{Guermond_Popov_2014, Lax_1971}. The viscous regularization can be chosen such that the resulting system matches the resistive model of the MHD equations, see e.g., \citep{Bohm_2018,Dao_2021}, or simply chosen as Laplacian terms under equal weights to all the conserved quantities, here we refer to as the ``\textit{monolithic parabolic regularization}''. The resistive MHD viscous flux suffers from several drawbacks for numerical purposes. One difficulty is that there is no regularization to the mass equation, which makes it incompatible with most numerical methods due to the Gibbs phenomenon. Another problem is that unless the thermal diffusivity is zero, the resistive MHD flux violates the minimum entropy principle, see \citep{Guermond_Popov_2014}. On the other hand, the non-physically-motivated monolithic viscous flux is employed in numerical schemes for MHD, e.g., \citep{Kuzmin_2020,Mabuza_2020} as well as being a continuous analog of the well-known Lax-Friedrichs or upwind schemes. However, to the best of our knowledge, investigations regarding entropy principles and other positivity-preserving properties of viscous regularizations to the MHD equations are still missing in the literature.

\cpurple{
  The main focus and contribution of this paper is to investigate the entropy principles of the monolithic parabolic regularization to the ideal MHD equations at the PDE level. We prove that the conservative form of the MHD equations with monolithic parabolic regularization satisfies positivity of density, minimum entropy principle, positivity of internal energy, and is compatible with all the generalized entropy inequalities in the manner of \cite{Harten_1998}. Our analysis is an extension of the works on compressible Euler equations of \cite{Guermond_Popov_2014} and \cite{Harten_1998}. The theory encourages the use of the monolithic viscous flux as a stabilization tool for the numerical approximation of the ideal MHD equations. 
}

\cpurple{
  State-of-the-art finite element methods for solving convection-dominated problems, such as the compressible Euler the MHD equations, are mainly based on the least-squares argument, see \eg \citep{Hughes_et_al_2010} and references therein. A secondary contribution of this article is the use of the monolithic parabolic regularization developed in this work as a finite element stabilization term without invoking any least-squares argument.
  First, in Section~\ref{sec:ns-visc} we validate the developed theory using a first order artificial viscosity and compare the monolithic viscous flux with the traditional Navier-Stokes resistive viscous flux.
  The numerical results reveal several benefits of the monolithic flux over the resistive MHD flux, which fit well with the theoretical findings.
  
  Then, we introduce the entropy viscosity method for MHD in the spirit of \citep{Guermond_2008, Guermond_2011, Nazarov_Larcher_2017}, where the regularization coefficients are constructed to be proportional to the entropy residual of the MHD system. Although the magnetic field was divergenceless in the continuous case, this is not the case in the discrete approximation. It is necessary to apply some divergence cleaning algorithms. In this article, we use the projection method \citep{Brackbill_Barnes_1980} to fix the divergence error. The tests show that the resulting method can maintain high-order accuracy of the smooth solutions, capture the shocks accurately, and remain stable even in the turbulence phase of the numerical solution. }

The rest of the paper is organized as follows. In Section \ref{Sec:Equation}, we describe the ideal MHD equations. Section \ref{Sec:monolithic} contains the analysis of the monolithic parabolic regularization at the PDE level. In Section \ref{Sec:numerical_investigation}, we relate the theoretical results in the continuous level with numerical results using a CG discretization. Section \ref{Sec:conclusion} contains a summary of our contributions and concluding remarks.

\section{The ideal MHD equations}\label{Sec:Equation}
Consider a $d$-dimensional spatial domain $\bx\in\mR^d$, $d\in\polN^*$, and a temporal domain $t\in[0,+\infty)$. We define the space time domain $D:=\mR^d\times[0,+\infty)$. Let us denote $\bU := (\rho,\bbm,E,\bB)$ a vector containing several conserved quantities of a conducting fluid: density $\rho(\bx,t):D\to\mR$, momentum $\bbm(\bx,t):D\to\mR^d$, total energy $E(\bx,t):D\to\mR$, and magnetic field $\bB(\bx,t):D\to\mR^d$. The velocity field $\bu(\bx,t)$ is determined by the relation $\bu(\bx,t):=\bbm(\bx,t)/\rho(\bx,t)$. The governing electrodynamics and fluid mechanics of such fluid can be described in the following ideal MHD equations,
\begin{equation}\label{eq:mhd1}
\p_t \bsfU + \DIV \bsfF_{\calE}(\bsfU) + \DIV \bsfF_{\calB}(\bsfU) = 
0,
\end{equation}
where the nonlinear tensor fluxes $\bsfF_{\calE}(\bsfU)$ and $\bsfF_{\calB}(\bsfU)$ are defined as
\begin{equation*}
	\bsfF_{\calE}(\bsfU):=
	\begin{pmatrix}
	\bbm \\
	\bbm \otimes \bu + p\polI \\
	\bu (E+p) \\
	0
	\end{pmatrix},
	\quad
	\bsfF_{\calB}(\bsfU):=
	\begin{pmatrix}
	0 \\
	-\bbetaa\\
	-\bbetaa\SCAL\bu\\
	\bu \otimes \bB - \bB \otimes \bu \\
    \end{pmatrix},
\end{equation*}
$\polI$ denotes the identity matrix of size $d\times d$, and $p$ is the thermodynamic pressure. The symmetric term $\bbetaa$ is called the Maxwell stress tensor:
\[
\bbetaa := -\frac{1}{2} (\bB\SCAL\bB)\polI+\bB\otimes\bB.
\]
In the absence of the magnetic field or the conductivity of the fluid $\DIV\bsfF_{\calB}(\bsfU)\equiv 0$, the MHD equations \eqref{eq:mhd1} reduce to the compressible Euler equations. Thus, in certain circumstances, it is useful to extend the knowledge of the Euler equations to study the MHD equations. The solenoidal nature of the magnetic field $\bB$ is implied from Faraday's law,
\begin{equation}\label{eq:div0}
\DIV\bB = 0,
\end{equation}
which is also known as the ``divergence free'' constraint. In the rest of the paper, we use the fact that \eqref{eq:div0} holds in the continuous settings.
The specific internal energy $e$ is defined as
\begin{equation}\label{eq:def_internal_energy}
\rho e := E - \frac{1}{2} \rho \bu^2 - \frac{1}{2} \bB^2,
\end{equation}
where the term $\frac{1}{2} \bB^2$ is often referred to as the magnetic pressure.

The specific entropy $s(\rho, e)$ and the thermodynamic pressure $p$ are defined through the following thermodynamic identity, see \cite[Section 12.2.4]{Somov_2012},
\begin{equation}\label{eq:thermodynamic_identity}
\ud e := T \ud s + \frac{p}{\rho^2}\ud\rho.
\end{equation}
Denote $s_e:=\frac{\p s}{\p e}$ and $s_{\rho} := \frac{\p s}{\p \rho}$. The following simple chain rule will be frequently used,
\begin{equation}\label{eq:s_chain_rule}
\p_{\alpha}s=s_e\p_{\alpha}e+s_{\rho}\p_{\alpha}\rho,
\end{equation}
where $\alpha=\{\bx,t\}$.
The equation \eqref{eq:thermodynamic_identity} can be written as
\begin{equation}\label{eq:thermodynamic_identity2}
\ud s = T^{-1}\ud e - p T^{-1}\rho^{-2}\ud \rho.
\end{equation}
Combining \eqref{eq:thermodynamic_identity2} with the following direct consequence of \eqref{eq:s_chain_rule},
\begin{equation*}
\ud s = \frac{\p s}{\p e} \ud e + \frac{\p s}{\p\rho}\ud\rho,
\end{equation*}
we deduce that
\begin{equation*}
s_e := T^{-1},\quad s_{\rho} := -p T^{-1}\rho^{-2},
\end{equation*}
with the second equality being called the ``equation of state'' and being equivalent to
\begin{equation}\label{eq:equation_of_state}
p s_e + \rho^2 s_{\rho} = 0.
\end{equation}
Throughout this paper, we have the following assumptions: the temperature $T$ is positive, which is equivalent to
\[
s_e > 0,
\]
and the specific entropy $-s$ is strictly convex with respect to $\rho^{-1}$ and $e$. Note that similar to \citep{Guermond_Popov_2014}, we do not make the assumption on the sign of the pressure $p$, as it is allowed to be both positive and negative. The strict convexity of $-s$ implies that, see Appendix \ref{appendix:convexity_implication},
\begin{equation}\label{eq:convex_entropy_assumption}
\p_{\rho} (\rho^2s_{\rho}) < 0, \quad s_{ee} < 0, \quad \p_{\rho}(\rho^2s_{\rho})s_{ee} - \rho^2s_{\rho e}^2 > 0.
\end{equation}

For ideal gases, it is common to use
\begin{equation}\label{eq:ideal_pressure}
p = (\gamma-1) \rho e = \rho T
\end{equation}
to compute pressure and temperature from the conserved variables, where $\gamma > 1$ is the adiabatic gas constant and
\begin{equation}\label{eq:s}
s = c_v\ln\frac{p}{\rho^{\gamma}},
\end{equation}
where $c_v$ is the specific heat capacity at constant volume. We note that the main results of this work, which are derived in Section~\ref{Sec:monolithic}, do not rely on assuming that the gas is ideal.

\section{Monolithic parabolic regularization}\label{Sec:monolithic}
The ideal MHD equations can be regularized with monolithic parabolic terms as in the conservative form below,
\begin{equation}\label{eq:mhd_monolithic}
\p_t \bsfU + \DIV \bsfF_{\calE}(\bsfU) + \DIV \bsfF_{\calB}(\bsfU) - \DIV \bsfF_{\calV}^m(\bsfU)  = 
0,
\end{equation}
where
\begin{equation}\label{eq:monolithic_flux}
\bsfF_{\calV}^m(\bsfU):=
\begin{pmatrix}
\epsilon \nabla\rho \\
\epsilon \nabla\bbm \\
\epsilon \nabla E\\
\epsilon \nabla \bB
\end{pmatrix},
\end{equation}
and $\epsilon$ is a small positive constant. The viscous flux \eqref{eq:monolithic_flux} is called ``monolithic'' because it simply adds the Laplacian terms corresponding to the conserved variables under an equal weight $\epsilon$. 
For reference, we separately write the equations in \eqref{eq:mhd_monolithic} as
\begin{align}
\label{eq:con_rho}  \p_t\rho & + \DIV \bbm && && = \epsilon \LAP\rho \\
\label{eq:con_m}  \p_t\bbm & + \DIV (\bbm \otimes \bu + p \polI ) && - \DIV \bbetaa && = \epsilon \LAP\bbm \\
\label{eq:con_E}  \p_t E & + \DIV (\bu(E+p)) && - \DIV (\bu \SCAL \bbetaa) && = \epsilon \LAP E \\
\label{eq:con_B}  \p_t \bB & && + \DIV (\bu \otimes \bB - \bB \otimes \bu) && = \epsilon \LAP \bB
\end{align}
For numerical purposes, the constant $\epsilon$ is often dependent on the resolution scale of the underlying scheme. The monolithic parabolic viscous terms can be linked to many classical numerical schemes with attractive properties. For example, on a uniform grid in one spatial dimension $\{\dots,x_{i-1},x_i,x_{i+1},\dots\}$, the well-known Lax-Friedrichs method applied on the unregularized mass conservation equation
\begin{equation}\label{eq:con_rho_unregularized}
\p_t\rho + \DIV \bbm = 0
\end{equation}
reads
\begin{equation}\label{eq:lax_friedrichs}
\rho_i^{n+1}=\frac{1}{2}(\rho_{i-1}^{n}+\rho_{i+1}^{n})-\frac{\delta t}{2\delta x}\left(m_{i+1}^n-m_{i-1}^n\right),
\end{equation}
where $\rho_i^n$ approximates $\rho(x_{i},t_{n})$, $m_i^n$ approximates $\bbm(x_{i},t_{n})$, $x_i \in \mR$, $t_{n+1} > t_n$, $\delta x = x_{i+1}-x_i > 0$, and $\delta t >0$ is the time step. Assuming sufficient smoothness, it is known that \eqref{eq:lax_friedrichs} is stable upon choosing $\delta t = C_{LF}\delta x$, where $C_{LF}$ is a sufficiently small real number. By adding and subtracting $\rho_i^{n}$ from the right-hand-side of \eqref{eq:lax_friedrichs}, we can rewrite \eqref{eq:lax_friedrichs} as
\begin{equation}\label{eq:lax_friedrichs_rewritten}
\rho_i^{n+1}=\rho_i^{n}-\frac{\delta t}{2\delta x}\left(m_{i+1}^n-m_{i-1}^n\right)+\epsilon_{LF}\frac{\delta t}{\delta x^2}(\rho_{i+1}^{n}-2\rho_i^{n}+\rho_{i-1}^{n}),
\end{equation}
where $\epsilon_{LF}=\frac{\delta x^2}{2\delta t}$. One can see that \eqref{eq:lax_friedrichs_rewritten} is a consistent discretization of \eqref{eq:con_rho} when $\epsilon$ is set to $\epsilon_{LF}$. If $\delta t$ is chosen to be of order $\calO(\delta x)$ following the stability condition, the viscosity coefficient $\epsilon_{LF}$ is of order $\calO(\delta x)$ and vanishes as $\delta x\to\infty$.

Upwind schemes can also be viewed as directly related to \eqref{eq:con_rho}. A regular upwind scheme approximating the nonlinear equation \eqref{eq:con_rho_unregularized} is often known as,
\begin{equation}\label{eq:upwind}
\rho_i^{n+1}=\rho_i^{n}-\frac{\delta t}{2\delta x}\left(m_{i+1}^n-m_{i-1}^n\right)+\epsilon_{Up}\frac{\delta t}{\delta x^2}(\rho_{i+1}^{n}-2\rho_i^{n}+\rho_{i-1}^{n}),
\end{equation}
where the vanishing viscosity coefficient reads $\epsilon_{Up} = \frac{1}{2}\delta x\|\bu\|_{L^{\infty}(\mR^d)}$ with $\|\bu\|_{L^{\infty}(\mR^d)}$ being the maximum density wave speed. Again, it can be seen that \eqref{eq:upwind} is a consistent discretization of \eqref{eq:con_rho} when $\epsilon$ is set to $\epsilon_{Up}$.

\subsection{Positivity of density}
For positivity of density, one can apply the same proof as \cite[Section 3.1]{Guermond_Popov_2014} did for the Euler equations since no equations other than the scalar mass conservation equation \eqref{eq:con_rho} are invoked. We conclude that the density to solution to \eqref{eq:con_rho} is positive given that the assumptions in the following theorem hold and the viscous parameter $\epsilon$ in \eqref{eq:con_rho} is positive.

\begin{theorem}[Positivity of density, \cite{Guermond_Popov_2014}]\label{theorem:positivity_density}
Upon the following justifiable assumptions:
\begin{enumerate}[label=(\roman*)]
\item $\bu$ and $\DIV\bu\in L^1(D)$;
\item $\int_0^{+\infty}(\p_t\rho+\DIV\bbm)\ud t$ and $\int_0^{+\infty}\epsilon\nabla\rho\ud t\in L^1(\mR^d)$;
\item there exists $\rho_{\min} > 0$ such that outside of a closed ball in $\mR^d$ centered at the origin with radius $R_{\widehat t}<\infty$, $\rho(\bx,t) \geq \rho_{\min}$ for all $t \in (0,\widehat t\,),\,\widehat t > 0$, and $\int_0^{\widehat t}(\DIV\bu(\bx,\widehat t\,))^+\ud t\in L^\infty(\mR^d)$,
\end{enumerate}
the density solution to \eqref{eq:con_rho} satisfies the following positivity property
  \[
\operatorname{essinf}_{\bx \in \Real^{d}} \rho(\bx, t) > 0, \quad \forall t>0.
  \]
\end{theorem}
In simple words, Theorem \ref{theorem:positivity_density} implies that under sufficient smoothness of the density and the velocity, and that outside of a region of interest, then the parabolic term $\epsilon\LAP\rho$, for an arbitrary positive value of $\epsilon$, regularizes \eqref{eq:con_rho} and guarantees positivity of the density solution.

\subsection{Minimum entropy principle}
In this section, we investigate the continuous minimum entropy principle of the monolithic flux. We write \eqref{eq:con_rho}-\eqref{eq:con_B} in the following nonconservative form,
\vspace{-0.5cm}
\begin{adjustwidth}{0cm}{-0.5cm}
\begin{align}
\label{eq:noncon_rho} & \p_t\rho + \bu\SCAL\nabla\rho +\rho \DIV \bu && && = \epsilon \LAP\rho,\\
\label{eq:noncon_m} & \rho (\p_t \bu + \bu\SCAL\nabla \bu) + \bu \epsilon\LAP\rho + \nabla p && - \DIV \bbetaa && = \epsilon \LAP\bbm,\\
\label{eq:noncon_E} & \rho(\p_t \calE + \bu\SCAL\nabla\calE) + \calE\epsilon\LAP\rho + \DIV(\bu p) && - \DIV (\bu \SCAL \bbetaa) && = \epsilon \LAP E,\\
\label{eq:noncon_B} & \p_t \bB && + \bB(\DIV\bu)+(\bu\SCAL\nabla)\bB-(\bB\SCAL\nabla)\bu && = \epsilon \LAP \bB,
\end{align}
\end{adjustwidth}
where $\calE = \rho^{-1}E$.

For convenience, we use the notation $\bl := \epsilon\nabla (\rho e)$ which will appear frequently. The term $\bl$ can be written in another form described by the following lemma.
\begin{lemma}
The following equality holds
\begin{equation}\label{eq:relation_rho_e_se}
\bl=s_{e}^{-1}\left(e s_{e}-\rho s_{\rho}\right)(\epsilon\nabla\rho)+\epsilon\rho s_{e}^{-1} \nabla s.
\end{equation}
\end{lemma}
\begin{proof}
Using the chain rule \eqref{eq:s_chain_rule}, we obtain $\nabla s = s_{\rho}\nabla\rho + s_e\nabla e$. We then have
\begin{align*}
s_e\nabla(\rho e) & = e s_e \nabla\rho + \rho s_e\nabla e \\
& = e s_e \nabla\rho + (\rho\nabla s - \rho s_{\rho}\nabla\rho) \\
& = (e s_e-\rho s_{\rho})\nabla\rho+\rho\nabla s.
\end{align*}
Since $\epsilon s_e^{-1} > 0$, multiplying both sides of the above equality with $\epsilon s_e^{-1}$ gives the desired conclusion.
\end{proof}
We introduce the following matrices,
\[
\quad \polG_{\bu} := \epsilon\rho\nabla\bu = \epsilon\nabla\bbm - (\epsilon\nabla\rho)\otimes\bu, \quad \polG_{\bB} := \epsilon\nabla\bB.
\]
An equation governing the specific entropy is derived in the following lemma.
\begin{lemma}\label{lemma:entropy_equation}
The following identity holds for any specific entropy $s(\rho, e)$,
\[
\rho (\p_t s + \bu \SCAL \nabla s) - \DIV (\rho \epsilon \nabla s) - (\epsilon\nabla\rho)\SCAL\nabla(es_e-\rho s_{\rho}) + \bl\SCAL\nabla s_e-s_e\polG_{\bu}:\nabla\bu - s_e\polG_{\bB}:\nabla\bB = 0.
\]
\end{lemma}

\begin{proof}
Denote $\calE_{\bB}=\frac{1}{2}\rho^{-1}\bB^2$, and $\calE_{\bu} = \frac{1}{2}\bu^2$. From Definition \eqref{eq:def_internal_energy}, the specific internal energy can be expressed as $e = \calE-\calE_{\bu}-\calE_{\bB}$. First, we use this equality to derive an equation controlling the internal energy $e$.
Multiplying \eqref{eq:noncon_m} with $\bu$ gives
\begin{equation}\label{eq:noncon_u_energy}
\rho\left(\p_t\calE_{\bu}+\bu\SCAL\nabla\calE_{\bu}\right)+\bu^2\epsilon\LAP\rho+\bu\SCAL\nabla p - \bu\SCAL(\DIV\bbetaa)-\epsilon\bu\SCAL\LAP\bbm = 0.
\end{equation}
Multiplying \eqref{eq:noncon_B} with $\bB$ and writing $\bB^2/2$ as $\rho\calE_{\bB}$, we end up with the following equation governing the magnetic energy,
\begin{equation}\label{eq:noncon_magnetic_energy}
\rho( \p_t \calE_{\bB}+\bu\SCAL\nabla\calE_{\bB})+\calE_{\bB}\epsilon\LAP\rho+\calE_{\bB}\rho\DIV\bu-\bB\SCAL(\bB\SCAL\nabla)\bu-\epsilon\bB\SCAL\LAP\bB = 0.
\end{equation}
An equation describing internal energy balance is obtained by subtracting \eqref{eq:noncon_u_energy} and \eqref{eq:noncon_magnetic_energy} from \eqref{eq:noncon_E},
\begin{align*}
  &\rho(\p_te+\bu\SCAL\nabla e)+\left(e-\frac{1}{2}\bu^2\right)\epsilon\LAP\rho\\
  & \hspace{2.2cm}+p(\DIV\bu)-\DIV(\bu\SCAL\bbetaa)+\bu\SCAL(\DIV\bbetaa)-\calE_{\bB}\rho\DIV\bu+\bB\SCAL(\bB\SCAL\nabla)\bu\\
& \hspace{2.2cm}-\epsilon\LAP E+\epsilon\bu\SCAL\LAP\bbm+\epsilon\bB\SCAL\LAP\bB = 0.
\end{align*}
By simple manipulation, we can show that
\[
-\DIV(\bu\SCAL\bbetaa)+\bu\SCAL(\DIV\bbetaa)-\calE_{\bB}\rho\DIV\bu+\bB\SCAL(\bB\SCAL\nabla)\bu = 0.
\]
Therefore, the equation can be simplified as
\begin{equation}\label{eq:minimum_entropy_e}
\begin{aligned}
&\rho(\p_te+\bu\SCAL\nabla e) +\left(e-\frac{1}{2}\bu^2\right)\epsilon\LAP\rho+p(\DIV\bu)\\ 
&\hspace{2.2cm}-\epsilon\LAP E+\epsilon\bu\SCAL\LAP\bbm+\epsilon\bB\SCAL\LAP\bB = 0.
\end{aligned}
\end{equation}
Now, we use \eqref{eq:minimum_entropy_e} to derive the desired equality. To utilize the chain rule \eqref{eq:s_chain_rule}, we multiply \eqref{eq:minimum_entropy_e} with $s_e$, \eqref{eq:noncon_rho} with $\rho s_{\rho}$ and add them together to obtain
\begin{equation}\label{eq:minimum_entropy_s}
\begin{aligned}
\rho(\p_ts+\bu\SCAL\nabla s)+(es_e-\rho s_{\rho})\epsilon\LAP\rho+(\rho^2s_{\rho}+ps_e)(\DIV\bu) & \\
  +s_e\left(-\frac{1}{2}\bu^2\epsilon\LAP\rho-\epsilon\LAP E+\epsilon\bu\SCAL\LAP\bbm+\epsilon\bB\SCAL\LAP\bB\right) &= 0,
\end{aligned}
\end{equation}
where $(\rho^2s_{\rho}+ps_e)(\DIV\bu)$ is zero due to the equation of state \eqref{eq:equation_of_state}. The terms inside the last bracket can be written as
\[
-\frac{1}{2}\bu^2\LAP\rho-\LAP E+\bu\SCAL\LAP\bbm+\bB\SCAL\LAP\bB = -\DIV\bl-\nabla\bbm:\nabla\bu+\frac{1}{2}\nabla\rho\SCAL\nabla\bu^2-\DIV(\bB\SCAL\nabla\bB)+\bB\SCAL\LAP\bB.
\]
Because $-\DIV(\bB\SCAL\nabla\bB)+\bB\SCAL\LAP\bB = -\nabla\bB:\nabla\bB$, the equation \eqref{eq:minimum_entropy_s} leads to
\[
\rho(\p_ts+\bu\SCAL\nabla s)+(es_e-\rho s_{\rho})\epsilon\LAP\rho-s_e\DIV\bl-s_e\polG_{\bu}:\nabla\bu-s_e\polG_{\bB}:\nabla\bB = 0.
\]
Applying the product rules $(es_e-\rho s_{\rho})\epsilon\LAP\rho = -(\epsilon\nabla\rho)\SCAL\nabla(es_e-\rho s_{\rho})+\DIV\left[(es_e-\rho s_{\rho})\epsilon\nabla\rho\right]$ and $s_e\DIV\bl=\DIV(\bl s_e)-\bl\SCAL\nabla s_e$, we have
\begin{align*}
  \rho (\p_t s + \bu \SCAL \nabla s) -(\epsilon\nabla\rho)\SCAL\nabla(es_e-\rho s_{\rho})+\DIV\left[(es_e-\rho s_{\rho})\epsilon\nabla\rho-\bl s_e\right]\\
  + \bl\SCAL\nabla s_e-s_e\polG_{\bu}:\nabla\bu - s_e\polG_{\bB}:\nabla\bB = 0.
\end{align*}
Combining with \eqref{eq:relation_rho_e_se}, we have the statement proved.
\end{proof}

The following result is useful for the main theorem of the minimum entropy principle.

\begin{lemma}\label{lemma_J_negative}
The quadratic form 
\[
J_1(\nabla \rho, \nabla e):=-(\epsilon\nabla\rho) \cdot \nabla\left(e s_{e}-\rho s_{\rho}\right)+\boldsymbol{l} \cdot \nabla s_{e}+\epsilon\nabla \rho \cdot \nabla s
\]
is negative definite.
\end{lemma}
\begin{proof}
Using \eqref{eq:relation_rho_e_se} and $\nabla s = s_e\nabla e + s_{\rho}\nabla\rho$, we can express $J_1$ as
\begin{equation*}
\begin{aligned}
  J_1=&-\epsilon s_{e}\nabla\rho\SCAL\nabla e-\epsilon e\nabla\rho\SCAL\nabla s_{e} + \epsilon s_{\rho}|\nabla \rho|^{2}+\epsilon\rho\nabla\rho\SCAL\nabla s_{\rho}\\
  &+\epsilon e\nabla\rho\SCAL\nabla s_{e}-\epsilon s_e^{-1}\rho s_{\rho}\nabla\rho\SCAL\nabla s_{e} \\
&+\epsilon\rho s_{e}^{-1}\nabla s_{e}\SCAL\left(s_{\rho}\nabla\rho+s_{e}\nabla e\right)+\epsilon\nabla\rho\SCAL\left(s_{\rho}\nabla\rho+s_{e}\nabla e\right).
\end{aligned}
\end{equation*}
Grouping terms with regards to quadratic terms of $(\nabla \rho, \nabla e)$ gives
\begin{equation*}
\begin{aligned}
J_1 = (\epsilon\rho s_{\rho\rho}+2\epsilon s_{\rho}) |\nabla\rho|^2+ 2\epsilon\rho s_{\rho e} \nabla\rho\SCAL\nabla e + \epsilon\rho s_{ee} |\nabla e|^2.
\end{aligned}
\end{equation*}
That allows us to rewrite $J_1$ in matrix form as
\begin{equation}\label{eq:J_matrix_form}
J_1 = 
\begin{pmatrix}
  \nabla\rho\\
  \nabla e
\end{pmatrix}
\left(
\begin{pmatrix}
\epsilon\rho^{-1}\p_{\rho}(\rho^2 s_{\rho}) & \epsilon\rho s_{\rho e} \\
\epsilon\rho s_{\rho e} & \epsilon\rho s_{ee}  
\end{pmatrix}
\otimes
\polI_d
\right)
\begin{pmatrix}
\nabla\rho & \nabla e
\end{pmatrix}.
\end{equation}
From the convex entropy inequalities \eqref{eq:convex_entropy_assumption}, we can see that the $2\times 2$ matrix
\begin{equation}\label{eq:J1_matrix}
\begin{pmatrix}
\epsilon\rho^{-1}\p_{\rho}(\rho^2 s_{\rho}) & \epsilon\rho s_{\rho e} \\
\epsilon\rho s_{\rho e} & \epsilon\rho s_{ee}  
\end{pmatrix}
\end{equation}
has positive determinant and a negative trace. They are sufficient conditions for negative definiteness of \eqref{eq:J1_matrix}. 
The lemma is proved. Therefore, we always have $J_1 \leq 0$ for all $(\nabla \rho, \nabla e)$.
\end{proof}
We proceed as in \citep{Guermond_Popov_2014}. The following theorem completes the purpose of this section.
\begin{theorem}[Minimum entropy principle]
Upon smoothness of the solution to \eqref{eq:con_rho}-\eqref{eq:con_B}, and the assumptions of Theorem \ref{theorem:positivity_density}, and $\bB\in L^1(D)$. We further assume that the density and the internal energy uniformly converge to constant states $\rho^*, e^*$ and remain constant outside of a compact set of interest. For all $t > 0$, the minimum entropy principle holds:
\[
\inf_{\bx \in \mR^{d}} s(\bx, t) \geq \inf _{\boldsymbol{x} \in \mathbb{R}^{d}} s_{0}(\boldsymbol{x}).
\]
\end{theorem}
\begin{proof}
From Lemmas \ref{lemma:entropy_equation} and \ref{lemma_J_negative}, we have
\begin{equation}\label{eq:minimum_entropy_1}
\rho\left(\partial_{t} s+\boldsymbol{u} \cdot \nabla s\right)-\nabla \cdot(\epsilon \rho \nabla s)-a \nabla \rho \cdot \nabla s=-J_1+s_{e}\polG_{\bu}: \nabla \boldsymbol{u}+s_e\polG_{\bB}:\nabla\bB \geq 0.
\end{equation}
If the infimum of $s(\bx,t)$ is reached outside of the compact set in the assumption, then the result follows readily since $\inf_{\bx \in \mR^{d}} s(\bx, t)=s^*\geq \inf_{\bx \in \mR^{d}} s_{0}(\boldsymbol{x})$, with $s(\bx,t)\to s^*$ as $\bx\to\infty$ due to the uniform convergence and smoothness assumptions on density and internal energy. Otherwise, at the point $\bar{\bx}(t)$ where $s(\bx, t)$ reaches its infimum, by smoothness assumptions, we have $\nabla s(\bar{\bx}(t),t) = 0$ and $\LAP s(\bar{\bx}(t),t) > 0$. From \eqref{eq:minimum_entropy_1}, it is clear that
\[
\rho \partial_{t} s\left(\bar{\bx}(t), t\right)-\epsilon\rho\LAP s\left(\bar{\bx}(t), t\right) \geq 0.
\]
This says that $\rho \partial_{t} s\left(\bar{\bx}(t), t\right) \geq 0$, which leads to the conclusion since $\rho > 0$ by Theorem \ref{theorem:positivity_density}.
\end{proof}

\begin{remark}[Positivity of internal energy]
For completeness, we repeat this argument by \cite{Guermond_Popov_2014}. The minimum entropy principle and the positivity of density lead to the positivity of specific internal energy $e > 0$ and therefore the positivity of internal energy $\rho e > 0$.  From the equation of state \eqref{eq:equation_of_state}, we can write $s = s(e,\rho^{-1})$ and $e = e(s,\rho^{-1})$. By \cite[Appendix A.1]{Guermond_Popov_2014}, we have that the assumption of positive temperature $s_e > 0$ leads to $e_s > 0$ if $\rho > 0$. Combining this with the minimum entropy principle, at time $t > 0$, we have $e(s,\rho^{-1}) \geq e(\min_\Omega s,\rho^{-1}) \geq e(\min_\Omega s_0,\rho^{-1}), \;\forall \rho > 0$. This means that if the specific internal energy is positive at $t=0$, then it remains positive at any $t > 0$. Only in the case of ideal gases, due to \eqref{eq:ideal_pressure}, the positivity of internal energy implies the positivity of pressure. 
\end{remark}
\subsection{Generalized entropy inequalities}
Let $f(s)$ be a twice differentiable function. We consider a class of strictly convex generalized entropies in the form $\rho f(s)$, as derived for the Euler equations by \citep{Harten_1998}. The choice of the form $\rho f(s)$ can be motivated by the unregularized equations of \eqref{eq:con_rho} and \eqref{eq:minimum_entropy_s}, which respectively give $\p_t\rho + \DIV(\rho\bu) = 0$ and $\p_t s + \bu\SCAL\nabla s = 0$. The latter equation leads to $\p_t f(s) + \bu\SCAL\nabla f(s) = 0$ for any differentiable function $f$. Multiplying $\p_t\rho + \DIV(\rho\bu) = 0$ with $f$ and $\p_t f(s) + \bu\SCAL\nabla f(s) = 0$ with $\rho$, and adding them together leads to $\p_t(\rho f(s)) + \DIV (\bu\rho f(s))$ = 0. Therefore, $\rho f(s)$ represents a large class of entropies for the MHD equations. We investigate some important properties of this class in the following theorem. We note that the results in Theorem \ref{theorem:convex_entropy} are well established for the Euler equations, see \cite{Harten_1983, Harten_1998}.

\begin{theorem}[Inequalities of strictly convex generalized entropies]\label{theorem:convex_entropy}
The specific heat capacity at constant pressure $c_p$ is defined as $c_p = T\frac{\partial s(p,T)}{\partial T}$. The generalized entropy $-\rho f(s)$ is strictly convex if and only if
\begin{equation}\label{eq:convex_entropy_inequalities}
f^{\prime}(s)>0, \quad \frac{f^{\prime}(s)}{c_p}-f^{\prime \prime}(s)>0.
\end{equation}
\end{theorem}
\begin{proof}
See Appendix \ref{appendix:convex_entropy}.
\end{proof}

\begin{lemma}\label{lemma:M_negative_definite}The following $2\times 2$ matrix is negative definite
\[
\begin{pmatrix}
c_p^{-1}\rho s_\rho^2+\rho^{-1}\p_{\rho}(\rho^2 s_{\rho}) & c_p^{-1}s_\rho\rho s_e+\rho s_{\rho e} \\
c_p^{-1}\rho s_\rho s_e+\rho s_{\rho e} & c_p^{-1}\rho s_e^2+ \rho s_{ee}  
\end{pmatrix}.
\]
\end{lemma}
\begin{proof}
The negative definiteness of the above $2\times 2$ matrix is shown in \citep[Lemma A.3]{Guermond_Popov_2014} for the Euler equations. The proof remains valid for the MHD equations because the only invoked necessities are the convexity assumption on $-s$ and the thermodynamic identity \eqref{eq:thermodynamic_identity}.
\end{proof}

\begin{theorem}[entropy inequalities]
Any smooth solution to the regularized system \eqref{eq:con_rho}-\eqref{eq:con_B} satisfies the entropy inequality
\[
\partial_{t}(\rho f(s))+\nabla \cdot(\boldsymbol{u} \rho f(s)-\epsilon\rho \nabla f(s)-\epsilon f(s) \nabla \rho) \geq 0
\]
\end{theorem}

\begin{proof}
Multiplying both sides of \eqref{eq:minimum_entropy_1} with $f'(s)$, we have
\[
\begin{array}{c}
\rho\left(\partial_{t} f(s)+\boldsymbol{u} \cdot \nabla f(s)\right)-\nabla \cdot(\epsilon \rho \nabla f(s))+\epsilon \rho f^{\prime \prime}(s)|\nabla s|^{2}-\epsilon f^{\prime}(s) \nabla \rho \cdot \nabla s \\
+J_1 f^{\prime}(s)=f^{\prime}(s) s_{e} \polG_{\bu}: \nabla \boldsymbol{u} + f^{\prime}(s) s_{e}\polG_{\bB}:\nabla\bB.
\end{array}
\]
Multiplying the density equation \eqref{eq:con_rho} with $\rho$ and adding it to the above equation, the product rule for temporal and spatial derivatives gives
\[
\begin{aligned}
\p_t(\rho f(s))&+\nabla \cdot(\boldsymbol{u} \rho f(s)) -\nabla \cdot(\epsilon \rho \nabla f(s)+\epsilon f(s) \nabla \rho) \\
&+\epsilon\rho f^{\prime \prime}(s)|\nabla s|^{2}+J_1 f^{\prime}(s)=f^{\prime}(s) s_{e} \polG_{\bu}: \nabla \bu + f^{\prime}(s) s_{e}\polG_{\bB}:\nabla\bB,
\end{aligned}
\]
which is close to the inequality that we want to prove. Since we know that the right hand side of the above equation is nonnegative, if $\epsilon\rho f^{\prime \prime}(s)|\nabla s|^{2}+J_1 f^{\prime}(s) \leq 0$ holds, then the proof is complete. Denote $J_2:= \epsilon\rho c_p^{-1}|\nabla s|^2+J_1$. Using the second inequality of \eqref{eq:convex_entropy_inequalities}, we can derive an upper bound of the quantity of interest $\epsilon\rho f^{\prime \prime}(s)|\nabla s|^{2}+J_1 f^{\prime}(s) < f'(s)\left(\epsilon\rho c_p^{-1}|\nabla s|^2+J_1\right)=f'(s)J_2$. Using the chain rule \eqref{eq:s_chain_rule} on $\nabla s$ and the matrix form of $J_1$ in \eqref{eq:J_matrix_form}, we have
\[
J_2=  \epsilon
\begin{pmatrix}
  \nabla\rho\\
  \nabla e
\end{pmatrix}
\left(
\begin{pmatrix}
c_p^{-1}\rho s_\rho^2+\rho^{-1}\p_{\rho}(\rho^2 s_{\rho}) & c_p^{-1}s_\rho\rho s_e+\rho s_{\rho e} \\
c_p^{-1}\rho s_\rho s_e+\rho s_{\rho e} & c_p^{-1}\rho s_e^2+ \rho s_{ee}  
\end{pmatrix}
\otimes
\polI_d
\right)
\begin{pmatrix}
\nabla\rho & \nabla e
\end{pmatrix}.
\]
Due to Lemma \ref{lemma:M_negative_definite}, $J_2$ is always nonpositive. Therefore, $\epsilon\rho f^{\prime \prime}(s)|\nabla s|^{2}+J_1 f^{\prime}(s) \leq 0$, which proves the theorem.
\end{proof}

\cred{
\begin{remark}[Entropy principles in the fully discrete settings]
  In the above analysis, entropy principles such as density positivity, internal energy, and entropy minimum principles are proved at the PDE level when $\DIV \bB \equiv 0$. However, this condition is usually violated in fully discrete approximations. The authors of \cite{Janhunen_2000, Christlieb_2015} showed that a slight violation of the divergence-free condition leads to a violation of the positivity property. Recently \cite{Wu_2018b, Wu_2021} proposed to use the Godunov form of ideal MHD instead, where the so-called Powell terms \citep{Powell_et_al_1991} are added to the system. A good feature of the Godunov form is that a slight violation of the divergence-free condition does not prevent the positivity property of the MHD. Then the authors \cite{Wu_2018b, Wu_2021} proved the entropy principles both at the PDE level and at the fully discrete level for some DG schemes. Extending these works within finite elements is our ongoing work and will be reported in a separate article.
 \end{remark}
}
\section{Numerical investigation}\label{Sec:numerical_investigation}
In this section, we demonstrate some numerical properties of the monolithic parabolic flux \eqref{eq:con_rho}-\eqref{eq:con_B} using a robust shock-capturing CG method.
\subsection{Finite element (CG) discretization}
For computation, instead of considering the unbounded spatial domain $\mR^d$ in the main analysis, we consider an open bounded subset $\Omega \in \mR^d$. The domain $\Omega$ is discretized into $N_e$ disjoint triangle elements $K_i, i=1,\dots,N_e$ being open sets in $\mR^d$ such that $\cup \{\overline{K_i}\}_{i=1}^{N_e} = \overline{\Omega}$, where $\overline{K_i}$ is the closure of $K_i$, and all vertices of the polytope $\overline{\Omega} \approx \Omega$ are contained by the boundary $\p\Omega$ of $\Omega$. For a valid CG discretization, we require that no hanging nodes are present, i.e., no vertices of any element lie on an edge of any other element. The set of all vertices and all elements of this partition constitutes the computational mesh $\calT_h$.

We define a continuous Lagrange finite element function space $\calQ_h$ as
\[
\calQ_h := \{ v \in C^{0}(\overline{\Omega}) \;\big\vert\; v_{|_{K_i}} \in \polP_k(K_i), \,\forall i=1,\dots,N_e \},
\]
where $C^{0}(\overline{\Omega})$ is the space of continuous functions on $\overline{\Omega}$, and $\polP_k(K_i)$ is the space of polynomials of at most $k$-th degree on $K_i$. The corresponding vector function space $\bcalV_h$ is defined as $\bcalV_h:=[\calQ_h]^d$. A mesh-size function $h(x) \in \calQ_h$ is defined through the following projection,
\begin{equation}\label{eq:h_h}
(h,v_h) = \left(\frac{h_K}{k}, v_h\right), \,\forall v_h \in\calQ_h,
\end{equation}
where the $L^2$--inner product $(u,v)\equiv(u,v)_{\Omega}:=\int_{\Omega} uv\ud x$ is defined for real-valued functions $u,v$ in $L^2(\Omega)$, and $h_K$ is the circumradius of element $K$.

A weak formulation of the regularized ideal MHD equations \eqref{eq:mhd_monolithic} reads: find $\bsfU_h(t):=(\rho_h(t),\bbm_h(t),E_h(t),\bB_h(t))^\top \in \calC^1(\mR^+,\calQ_h\times\bcalV_h\times\calQ_h\times\bcalV_h)$ such that
\begin{equation}\label{eq:ode}
\begin{aligned}
(\p_t \bsfU_h,\bsfV_h) + (\DIV \bsfF_{\calE}(\bsfU_h),\bsfV_h) + (\DIV \bsfF_{\calB}(\bsfU_h),\bsfV_h) & \\
  + (\bsfF_{\calV}^m(\bsfU),\nabla\bsfV_h) - (\bn\SCAL\bsfF_{\calV}^m(\bsfU),\bsfV_h)_{\p\Omega} & = 0,
\end{aligned}
\end{equation}
for all test functions $\bsfV_h\in\calQ_h\times\bcalV_h\times\calQ_h\times\bcalV_h$, where the boundary inner product $(u,v)_{\p\Omega}=\int_{\p\Omega}uv\ud s$ is a surface integral, and the vector $\bn$ is the pointing-outward normal vector defined at every nodal point on the boundary $\p\Omega$. Solution of the system \eqref{eq:ode} is often said to be a viscous solution of the ideal MHD system \eqref{eq:mhd_monolithic}. The viscosity coefficient $\epsilon$ is constructed such that it vanishes with mesh refinement. Therefore, as $h\to 0$, the viscous solution of \eqref{eq:ode} converges to the weak solution of \eqref{eq:mhd_monolithic}.

Our code is implemented in FEniCS, an open source finite element library, see \cite{Logg_2012}.
\subsection{Time stepping}
To proceed in time, we solve the ODE \eqref{eq:ode} using the strong stability preserving Runge-Kutta schemes of order 3 when $\polP_1$ elements are used in space and order 4 when $\polP_3$ elements are used in space, see \cite{Ruuth_2006}.

The time step size is adaptively chosen following a CFL condition,
\[
\Delta t = \text{CFL} \left(\min_{\bx\in\Omega} h({\bx})\right) \left(\max_{i=1,8}\vert\Lambda_i\vert\right)^{-1},
\]
where $\Lambda_i$ corresponds to the $i$-th eigenvalue of the MHD system, see \cite[Equation (3.5)]{Dao_2021}. In the following tests, the CFL number is chosen to be $0.3$.
\subsection{Divergence cleaning}
Satisfying the divergence-free condition \eqref{eq:div0} has been known as a challenging task in numerically solving MHD. However, due to being not of the main focus, we use the simplest divergence cleaning method for the numerical demonstration in this paper: the projection method \citep{Brackbill_Barnes_1980}. \cblue{In each Runge-Kutta stage}, the following cleaning procedure is applied:
\begin{enumerate}
\item Solve the Poisson equation $\LAP\Psi_h-\DIV\bB_h = 0$ for $\Psi_h \in \calQ_h$.
\item Calculate the projection $\bB_h'$ of the magnetic field $\bB_h$ onto the divergence-free space: $\bB_h'=\bB_h-\nabla\Psi_h$.
\item Use $\bB_h'$ as the magnetic field solution to proceed in the next time step, and update the dependent numerical variables accordingly to ensure consistency: pressure, temperature, energy, and other entropy-related variables.
\end{enumerate}
Despite the seemingly ad-hoc nature, in \cite{Toth_2000}, conservation and accuracy preserving properties of the projection method is proved.
\subsection{Boundary conditions}
In the benchmark tests in this paper, we use two basic types of boundary conditions: one is Dirichlet, and the other is the periodic boundary. The Dirichlet boundary conditions are injected into the solution vector in each time step. The periodic boundary mapping is done via built-in functions in FEniCS, see \cite{Logg_2012}.
\subsection{Comparison with the resistive MHD flux}
\label{sec:ns-visc}
It is natural and reasonable to use the resistive MHD flux, similar to using the Navier-Stokes flux for the compressible Euler equations.
We demonstrate why using the resistive MHD flux is not suitable for artificial viscosity methods.

The resistive model of the ideal MHD equations are obtained by replacing the monolithic flux $\bsfF_{\calV}^m(\bsfU)$ in \eqref{eq:mhd_monolithic} with the viscous flux $\bsfF_{\calV}^r(\bsfU)$ which we call ``resistive MHD flux'', see e.g., \cite{Bohm_2018},
\begin{equation}\label{eq:resistive_mhd_flux}
\bsfF_{\calV}^r(\bsfU):=
\begin{pmatrix}
0 \\
\btau \\
\bu \SCAL \btau + \kappa \GRAD T + \eta\bB \SCAL \big( \GRAD \bB - \GRAD \bB^\top \big)\\
\eta \big( \GRAD \bB - \GRAD \bB^\top \big)
\end{pmatrix},
\end{equation}
where the viscous shear stress tensor $\btau$ is
\[
\btau = \mu \Big( \GRAD\bu + \GRAD\bu^\top\Big) - \lambda \DIV\bu\polI,
\]
and $\kappa, \eta, \mu, \lambda$ are different viscosity coefficients. The sign of $\lambda$ is not determined. In applications, $\lambda$ is often neglected, or is set as $\lambda=-\frac{2}{3}\mu$.
\subsubsection{An example: contact waves}\label{sec:example_contact_wave}
We consider a contact line problem as an example: there is a discontinuity in the density, but the velocity, pressure, and magnetic field are constant functions. Let $\rho(\bx,0)=\rho_0(\bx)$ be an initial density field containing the contact discontinuity. We can verify that given some uniform fields $\bu(\bx,t) = \bu_0, \bB(\bx,t) = \bB_0,p(\bx,t) = p_0$, where the functions $\bu_0, \bB_0,p_0$ are constant, if a density solution $\rho(\bx,t)$ solve the mass equation
\begin{equation}\label{eq:example_contact_rho}
\p_t\rho + \DIV(\rho\bu_0) = \epsilon \LAP\rho,
\end{equation}
then $\bu_0,\bB_0,p_0,\rho$ also solve the momentum, energy, and magnetic equations. Indeed, when the velocity, pressure, and magnetic solutions are constant, the momentum equation \eqref{eq:con_m}, and the magnetic equation \eqref{eq:con_B} follows trivially. Inserting $\bu_0, \bB_0,p_0, \rho$ into the energy equation \eqref{eq:con_E}, we have
\begin{equation}\label{eq:example_contact_E}
\p_t(\rho e)+\DIV(\bu_0\rho e)-\epsilon\LAP(\rho e) + \frac{1}{2}\bu_0^2\left(\p_t\rho+\DIV(\rho\bu_0)-\epsilon\LAP\rho\right) = 0.
\end{equation}
For ideal gas, the equation of state \eqref{eq:ideal_pressure} implies that $\rho e = p_0/(\gamma-1)$ is a constant function. Combining this fact with \eqref{eq:example_contact_rho}, we can see that \eqref{eq:example_contact_E} is fulfilled. Therefore, the monolithic parabolic flux is compatible with contact lines.

The same conclusion cannot be drawn about the resistive MHD flux. We make the same assumptions as above on velocity, pressure, magnetic field, and initial density solution. A density solution $\rho(\bx, t)$ to the MHD equations regularized by the resistive MHD flux satisfies
\begin{equation}\label{eq:example_contact_rho_resistive}
\p_t\rho + \DIV(\rho\bu_0) = 0.
\end{equation}
In a similar manner, the momentum equation \eqref{eq:con_m}, and the magnetic equation \eqref{eq:con_B} follows trivially. However, inserting $\bu_0, \bB_0,p_0, \rho$ into the energy equation \eqref{eq:con_E} gives
\begin{equation}\label{eq:example_contact_E_resistive}
\p_t(\rho e)+\DIV(\bu_0\rho e)+ \frac{1}{2}\bu_0^2\left(\p_t\rho+\DIV(\rho\bu_0)\right) + \DIV(\kappa p_0\rho^{-2}\nabla\rho)= 0.
\end{equation}
The equations \eqref{eq:ideal_pressure} and \eqref{eq:example_contact_rho_resistive} imply that \eqref{eq:example_contact_E_resistive} only holds if the thermal diffusivity is zero, i.e., $\kappa = 0$. For the compressible Euler equations, letting $\kappa = 0$ in \eqref{eq:resistive_mhd_flux} is known to lead to Gibbs phenomenon to the numerical solution, see \citep{Nazarov_Larcher_2017}. This argument suggests that the resistive MHD flux is not compatible with contact lines.

We numerically demonstrate this argument on a contact solution extracted from a Riemann solver \citep{Torrilhon_2002} of the Brio-Wu problem \citep{Brio_Wu_1988}. The spatial domain is one dimensional $\Omega = [0,1]$. The gas constant is $\gamma=2$. The initial solution contains a constant velocity field $\bu=(u_x,u_y)$, $u_x=0.5915470932$, $u_y=-1.5792628803$, constant pressure $p=0.5122334291$, and a constant magnetic field $\bB = (B_x,B_y)$, $B_x=0.75$, $B_y=-0.5349102426$. There is a discontinuity in the density
\[
\rho =
\begin{cases}
0.7156521382, & \bx\in[0,0.5], \\
0.2348529760, & \bx\in(0.5,1].
\end{cases}
\]
In one dimension, the divergence-free condition \eqref{eq:div0} reduces to $\p_x B_x = 0$, which means that $B_x$ needs to remain constant across $\Omega$ at all future time. Since the violation of this condition in the numerical approximations of the Brio-Wu solution is typically negligible, the divergence of $B_x$ is left untreated. At every nodal point, the viscosity coefficients $\epsilon, \mu, \nu$ are chosen to be $\frac{1}{2}h\max_{i=1,8}\vert\Lambda_i\vert$, $h$ is calculated by \eqref{eq:h_h}, and $\lambda$ is set to $0$. The monolithic flux with the mentioned choice of $\epsilon$ resembles the Lax-Friedrichs scheme or the upwind scheme. The contact line at time $\widehat t = 0.1$ is captured by the monolithic flux and the resistive MHD flux in Figure \ref{fig:briowu_contact} under multiple resolution levels. It can be seen that the monolithic flux can capture the contact line without undershoots and overshoots. This is not the case for the resistive MHD flux. Choosing either $\kappa=0$ or $\kappa=1$ leads to overshoots and undershoots in the numerical solutions. Increasing or decreasing $\kappa$ from the standard value $\kappa=1$ does not help with the situation.

\begin{figure}[!ht]
     \centering
     \begin{subfigure}{\textwidth}
         \centering
         \includegraphics[width=0.49\textwidth,viewport=100 580 480 775, clip=true]{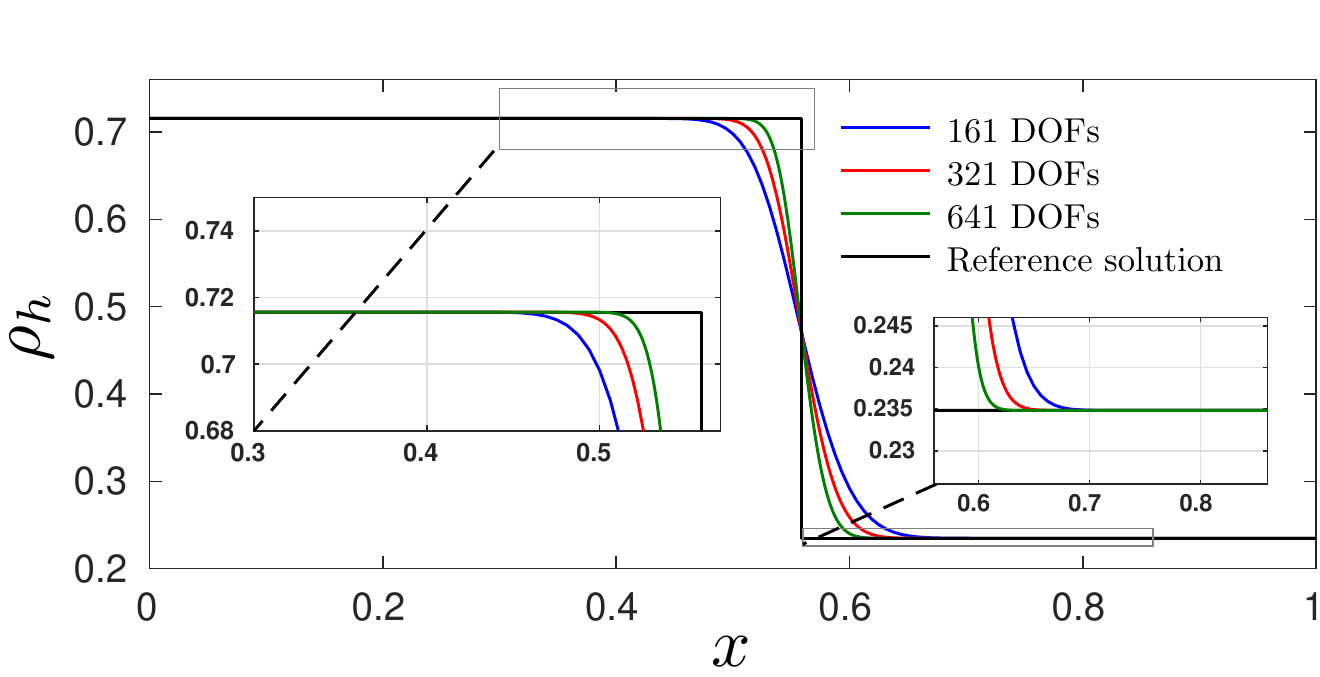}
         \caption{Monolithic flux}
     \end{subfigure}
     \hfill
     \begin{subfigure}{0.49\textwidth}
         \centering
         \includegraphics[width=\textwidth,viewport=100 580 480 775, clip=true]{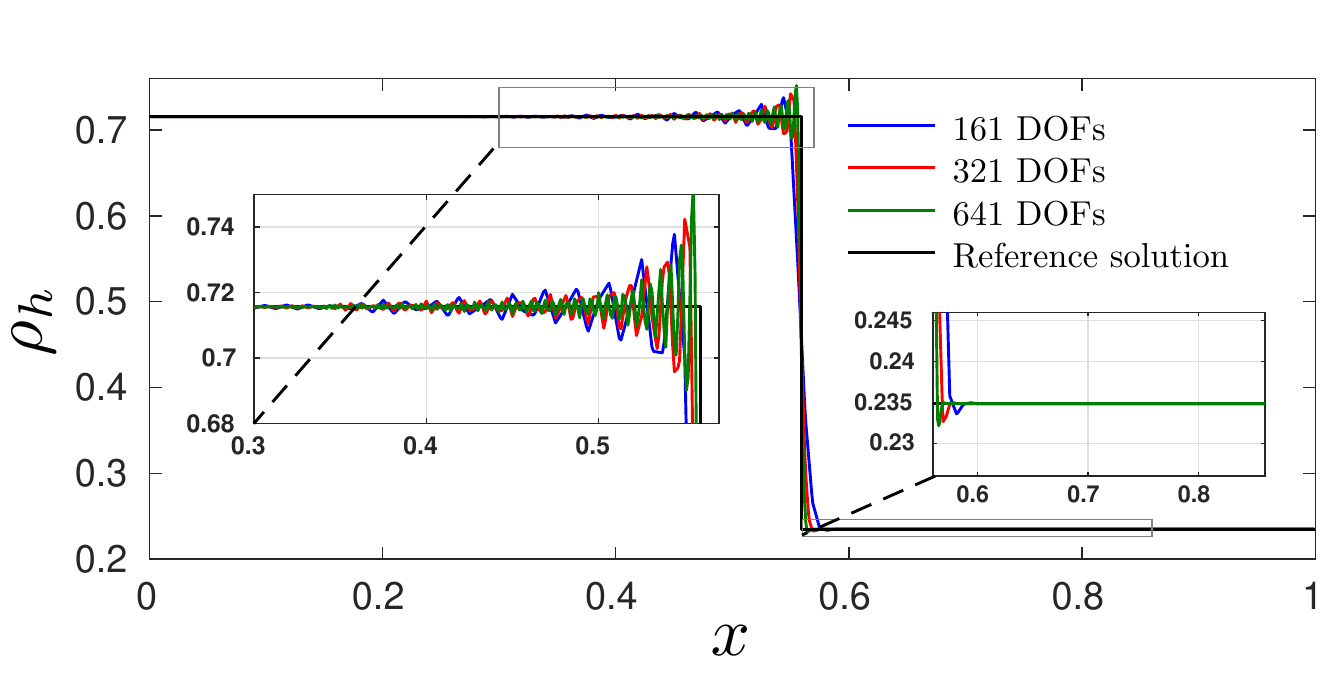}
         \caption{Resistive MHD flux, $\kappa=0$}
     \end{subfigure}
     \hfill
     \begin{subfigure}{0.49\textwidth}
         \centering
         \includegraphics[width=\textwidth,viewport=100 580 480 776, clip=true]{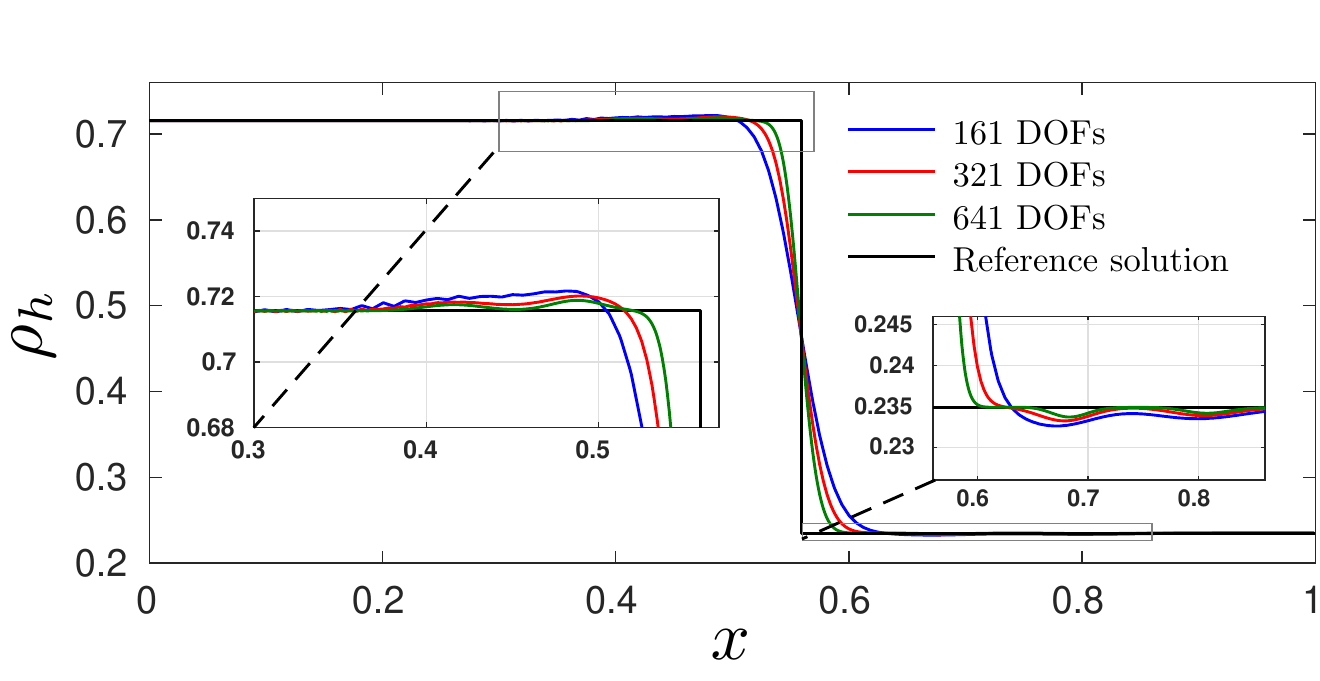}
         \caption{Resistive MHD flux, $\kappa=1$}
     \end{subfigure}
     \hfill
     \begin{subfigure}{0.49\textwidth}
         \centering
         \includegraphics[width=\textwidth,viewport=100 580 480 776, clip=true]{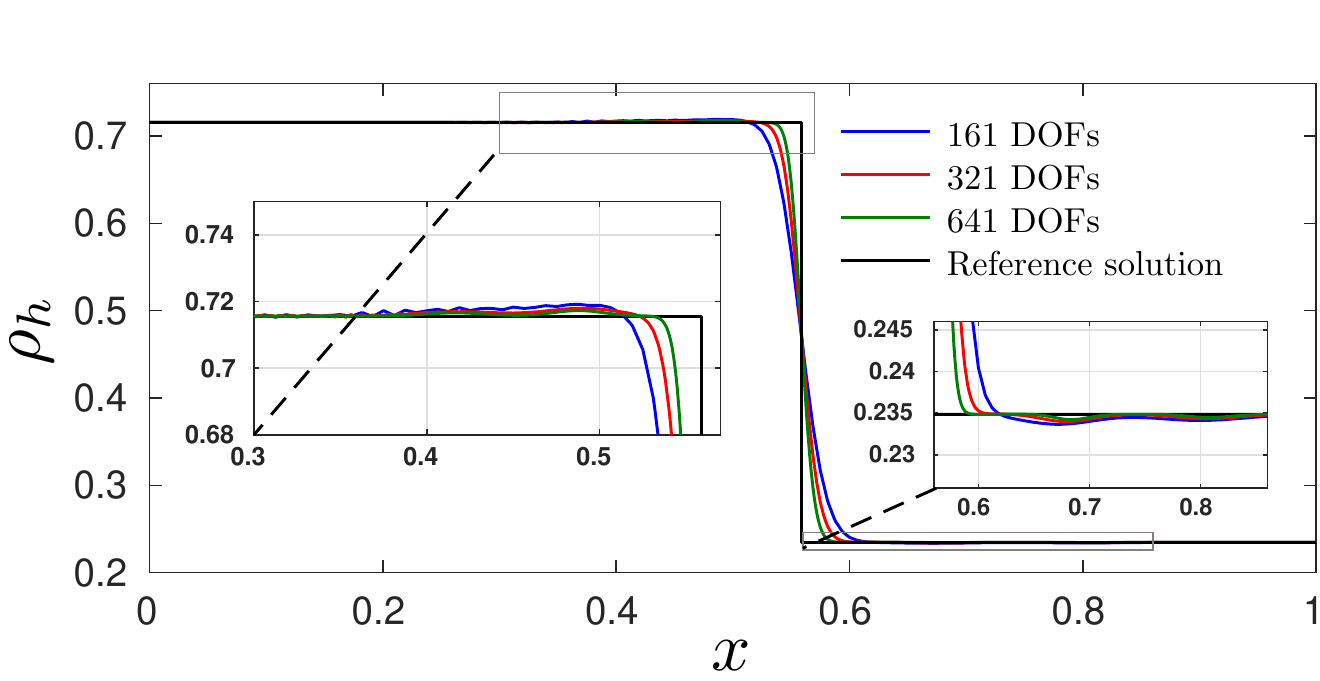}
         \caption{Resistive MHD flux, $\kappa=\frac{1}{2}$}
     \end{subfigure}
     \hfill\hfill
     \begin{subfigure}{0.49\textwidth}
         \centering
         \includegraphics[width=\textwidth,viewport=100 580 480 776, clip=true]{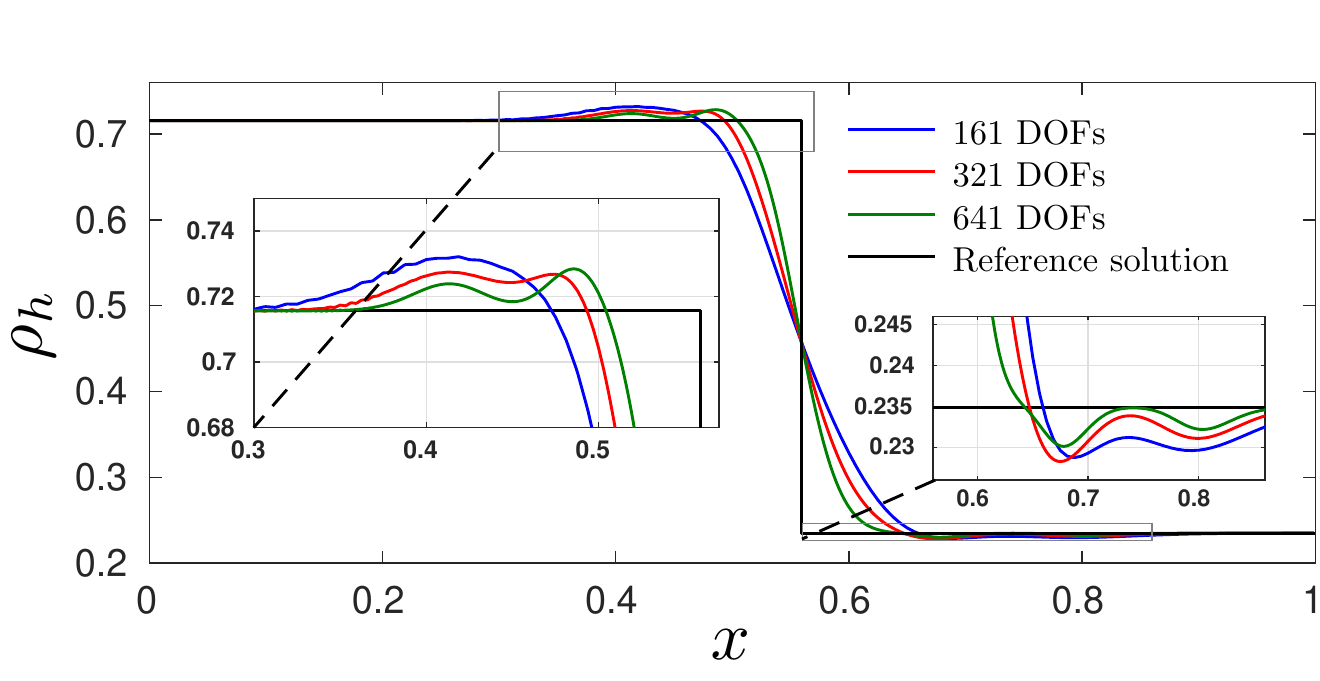}
         \caption{Resistive MHD flux, $\kappa=5$}
     \end{subfigure}
     \caption{Single contact solution captured by monolithic flux and resistive MHD flux with different choices of the thermal diffusivity coefficient $\kappa$. The solutions are shown under multiple resolution levels: 161, 321, 641 DOFs. First order viscosity. Final time $\widehat t = 0.1$.}
     \label{fig:briowu_contact}
\end{figure}

\subsubsection{Single waves from Brio-Wu problem \citep{Brio_Wu_1988}}\label{sec:single_waves}

The Brio-Wu problem is a one dimensional Riemann problem, $\Omega = [0, 1]$. The initial profile is given by
\[
(\rho, u, p, B_x, B_y) =
\begin{cases}
(1,0,1,0.75,1), & \bx\in[0,0.5], \\
(0.125,0,0.1,0.75,-1), & \bx\in(0.5,1].
\end{cases}
\]
The adiabatic constant is $\gamma=2$. The well-known Brio-Wu problem is an essential yet challenging test to examine if a numerical method can capture different MHD wave structures accurately: the shocks, the rarefactions, and the contact lines. The density solutions at the final time $\hat t = 0.1$ comparing monolithic flux and resistive MHD flux are shown in Figure~\ref{fig:briowu} with different mesh resolutions. The viscosity coefficients $\epsilon, \mu, \nu,\lambda$ are calculated at nodal points same to Section \ref{sec:example_contact_wave}. Similar to the previous example, no divergence cleaning procedure is used for the current one dimensional tests. Visualization of density solution in Figure~\ref{fig:briowu} shows that using the resistive MHD flux when $\kappa=1$ can be numerically sufficient to capture the whole compound structure of the Brio-Wu solution. We show the result by the resistive MHD flux with $\kappa=0$ in Figure~\ref{fig:briowu}(b), which clearly shows that the solution is polluted by spurious oscillations when the thermal diffusivity is zero. We also test an interesting setting of the monolithic flux when we drop the regularization to the mass equation, i.e., set $\epsilon\LAP\rho$ to zero in $\bsfF_{\calV}^m(\bsfU)$, and show the result in Figure~\ref{fig:briowu}(c). Under this setting, the compound structure of the Brio-Wu solution is still captured without the unphysical oscillations, as opposed to the intuition that mass regularization is the key to eliminating this phenomenon. In contrast to the distinguishable behaviors resulting from the two fluxes in Figure~\ref{fig:briowu_contact}, due to the combination of different waves, one can hardly decide whether the monolithic flux in Figure~\ref{fig:briowu}(a) or the resistive MHD flux with $\kappa=1$ in Figure~\ref{fig:briowu}(d) is better than the other.

\begin{figure}[h!]
     \centering
     \begin{subfigure}{0.49\textwidth}
         \centering
         \includegraphics[width=\textwidth,viewport=100 454 480 780, clip=true]{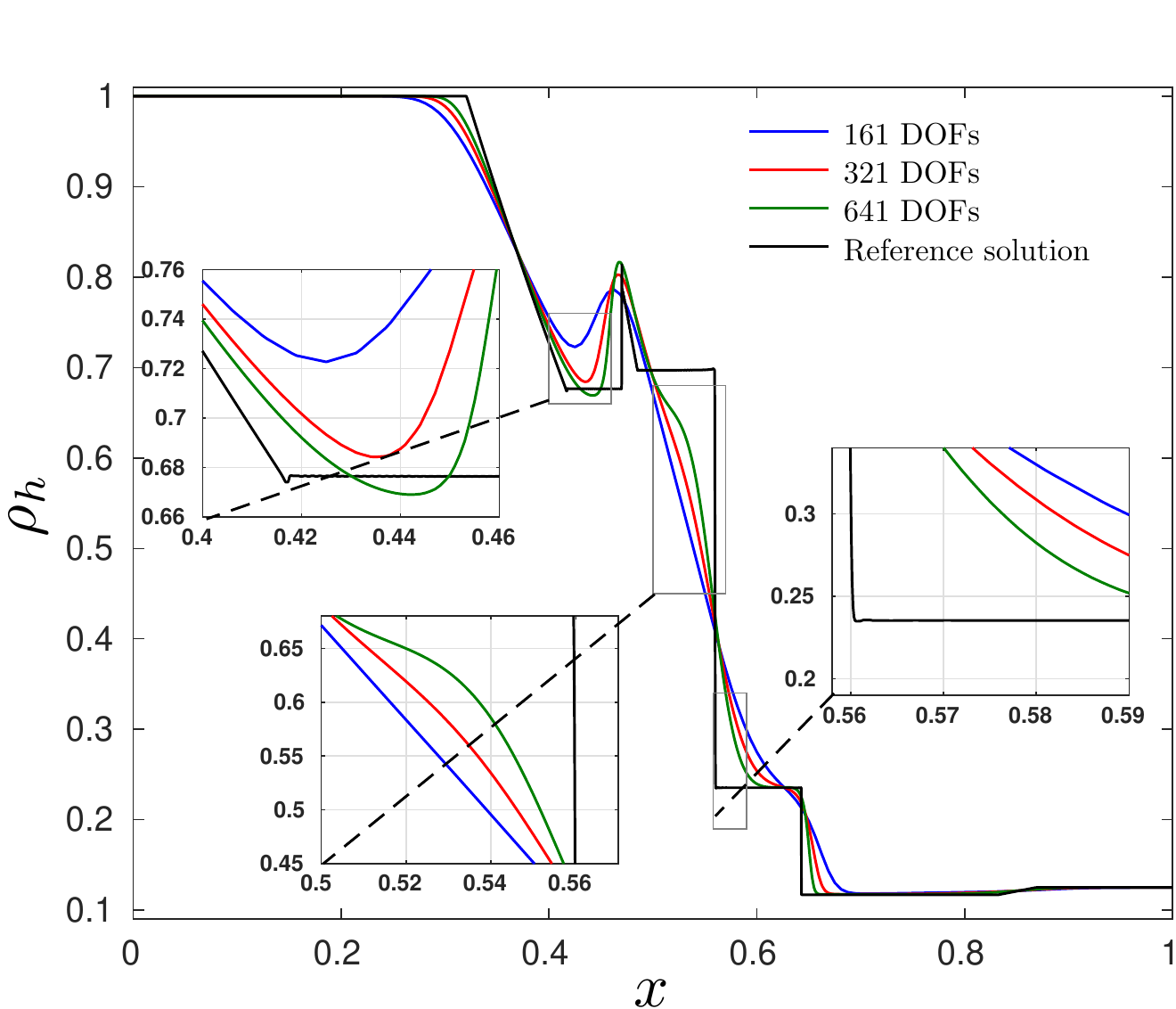}
         \caption{Monolithic flux}
     \end{subfigure}
     \hfill
     \begin{subfigure}{0.49\textwidth}
         \centering
         \includegraphics[width=\textwidth,viewport=100 454 480 780, clip=true]{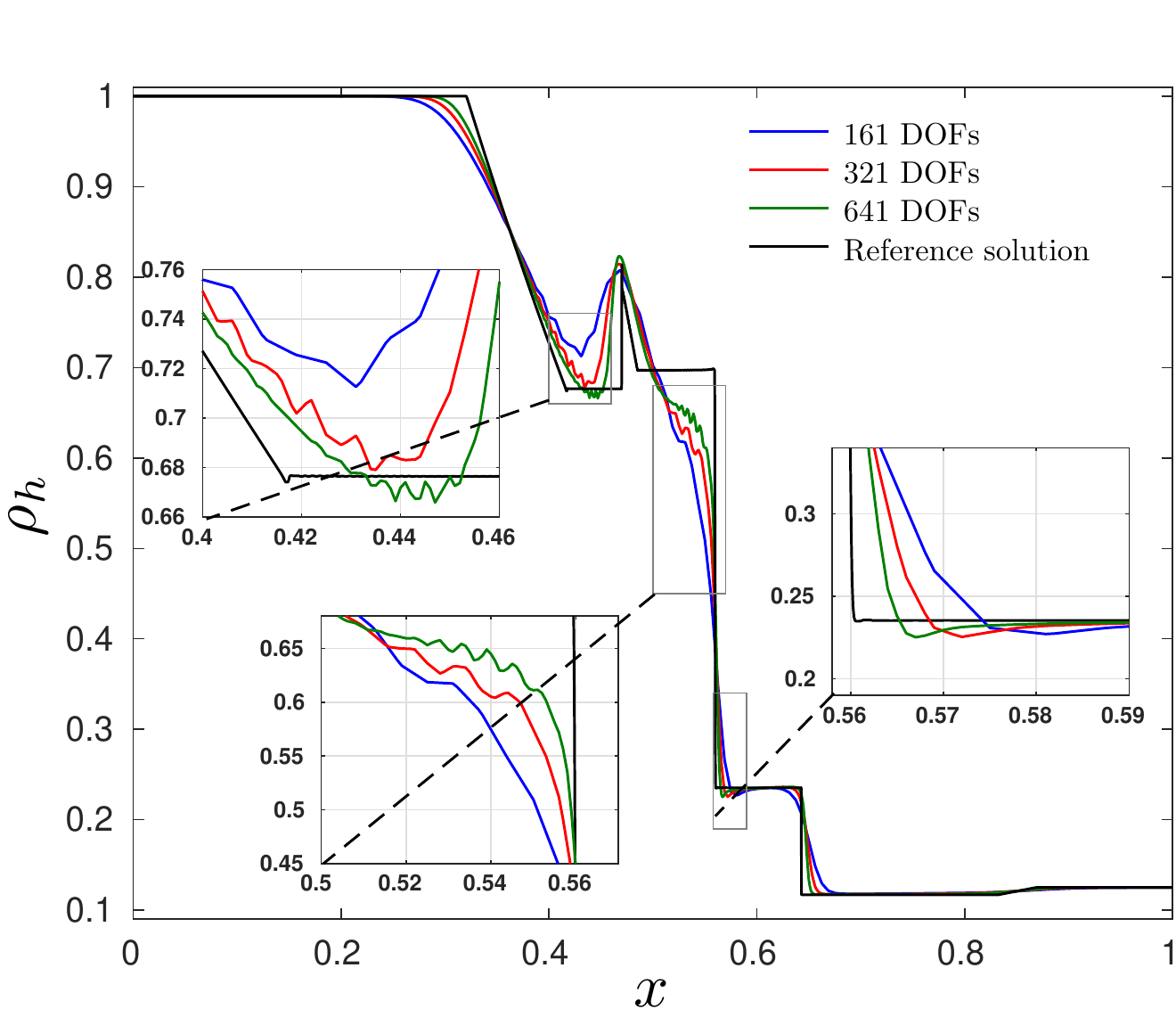}
         \caption{Resistive MHD flux, $\kappa=0$}
     \end{subfigure}
     \hfill
     \begin{subfigure}{0.49\textwidth}
         \centering
         \includegraphics[width=\textwidth,viewport=100 454 480 780, clip=true]{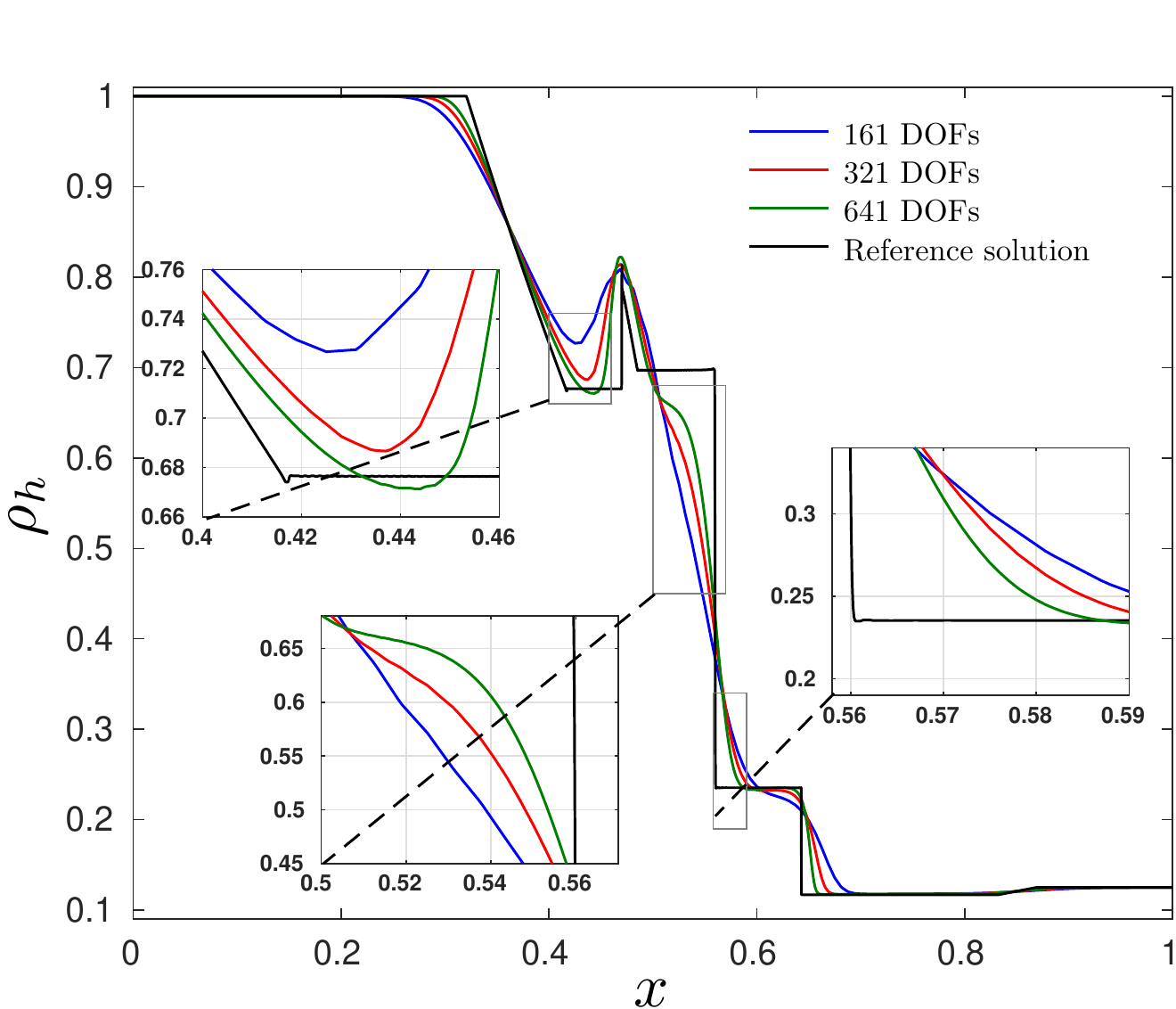}
         \caption{Monolithic flux, no mass regularization}
     \end{subfigure}
     \hfill
     \begin{subfigure}{0.49\textwidth}
         \centering
         \includegraphics[width=\textwidth,viewport=100 454 480 780, clip=true]{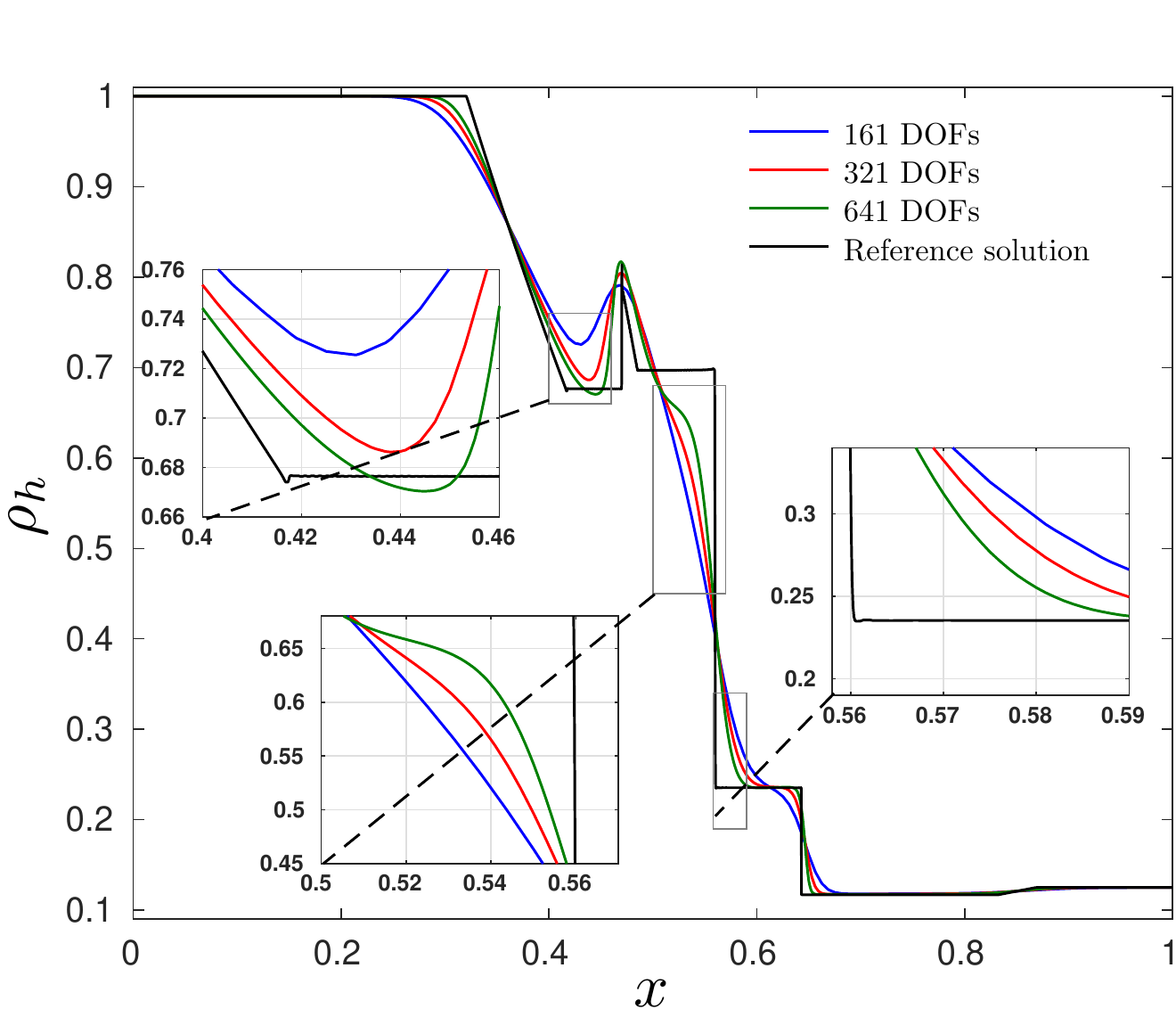}
         \caption{Resistive MHD flux, $\kappa=1$}
     \end{subfigure}
     \caption{The Brio-Wu solutions with 161, 321, 641 DOFs comparing the monolithic flux and the resistive MHD flux with first order viscosity. Final time $\widehat t = 0.1$.}
     \label{fig:briowu}
\end{figure}

Due to the above reason, we separate the waves of the Brio-Wu solution and look into the behavior of each of the single waves. For this purpose, we use an exact Riemann solver by \cite{Torrilhon_2002} to extract the single wave solutions. The initial profiles to generate them: the contact, fast rarefaction, intermediate shock-slow rarefaction, and slow shock are given in Table \ref{table:initial_single_waves}. The results are reported in Figure~\ref{fig:briowu_single_waves}.

\begin{table}[h!]
    \centering
    \caption{Initial solutions for the single waves of the Brio-Wu problem. The solutions are extracted from an exact Riemann solver \cite{Torrilhon_2002}. Entries with ``--'' indicates that the values on the left and the right states are the same.}
    \label{table:initial_single_waves}
    \begin{adjustbox}{max width=\textwidth}
    \begin{tabular}{|c|c|r|r|}\hline
{}  & & Left state & Right state  \\ \hline
\multirow{6}{*}{\rotatebox[origin=c]{90}{Contact}}
& $\rho$ &  0.7156521382 & 0.2348529760 \\
& $u_x$  &  0.5915470932 & -- \\
& $u_y$  & -1.5792628803 & -- \\
& $p$    &  0.5122334291 & -- \\
& $B_x$  &  0.7500000000 & -- \\
& $B_y$  & -0.5349102426 & -- \\ \hline
\multirow{6}{*}{\rotatebox[origin=c]{90}{\footnotesize{Intermediate shock}}}
& $\rho$ & 0.6799272943 & 0.2348529760 \\
& $u_x$  &  0.6288155014 & 0.5915470935 \\
& $u_y$  & -0.2295748706 & -1.5792628801 \\
& $p$    &  0.4623011255 & 0.5122334291 \\
& $B_x$  &  0.7500000000 & -- \\
& $B_y$  & 0.5900487481 & -0.5349102425 \\ \hline
    \end{tabular}
    \begin{tabular}{|c|c|r|r|}\hline
{}  & & Left state & Right state  \\ \hline
\multirow{6}{*}{\rotatebox[origin=c]{90}{Fast rarefaction}}
& $\rho$ &  1.0000000000 & 0.6799272943 \\
& $u_x$  &  0.0000000000 & 0.6288155014 \\
& $u_y$  &  0.0000000000 & -0.2295748706 \\
& $p$    &  1.0000000000 & 0.4623011255 \\
& $B_x$  &  0.7500000000 & -- \\
& $B_y$  &  1.0000000000 & 0.5900487481 \\ \hline
\multirow{6}{*}{\rotatebox[origin=c]{90}{Slow shock}}
& $\rho$ &  0.2348529760 & 0.1168051318 \\
& $u_x$  &  0.5915470930 & -0.2455906431 \\
& $u_y$  & -1.5792628803 & -0.1711653489 \\
& $p$    &  0.5122334291 & 0.0873180084 \\
& $B_x$  &  0.7500000000 & -- \\
& $B_y$  & -0.5349102426 & -0.9001418247 \\ \hline
    \end{tabular}
    \end{adjustbox}
\end{table}

\begin{figure}[h!]
     \centering
     \begin{subfigure}{0.325\textwidth}
         \centering
         \includegraphics[width=\textwidth,viewport=100 454 480 780, clip=true]{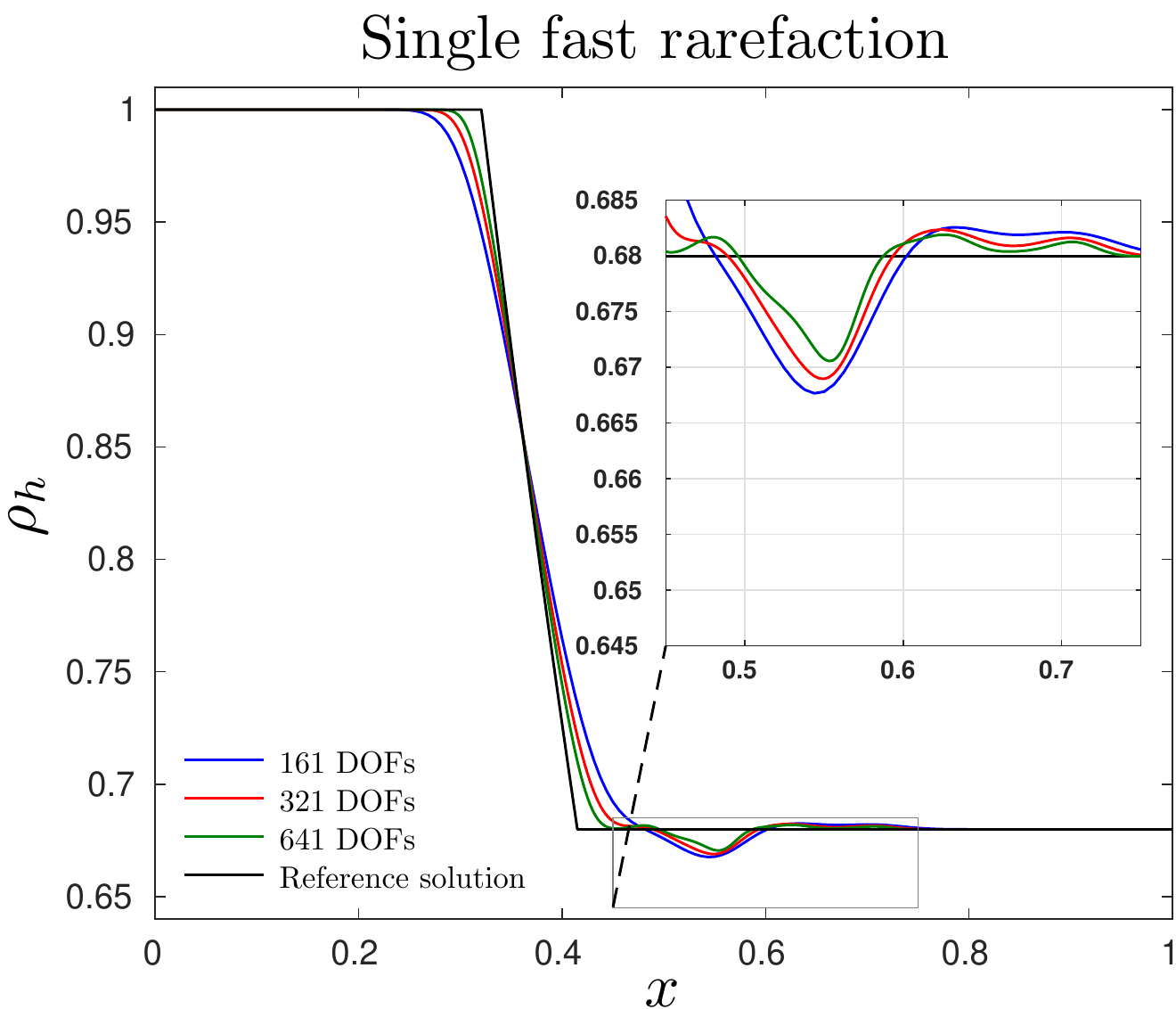}
         \caption{Monolithic}
     \end{subfigure}
     \hfill
     \begin{subfigure}{0.325\textwidth}
         \centering
         \includegraphics[width=\textwidth,viewport=100 454 480 780, clip=true]{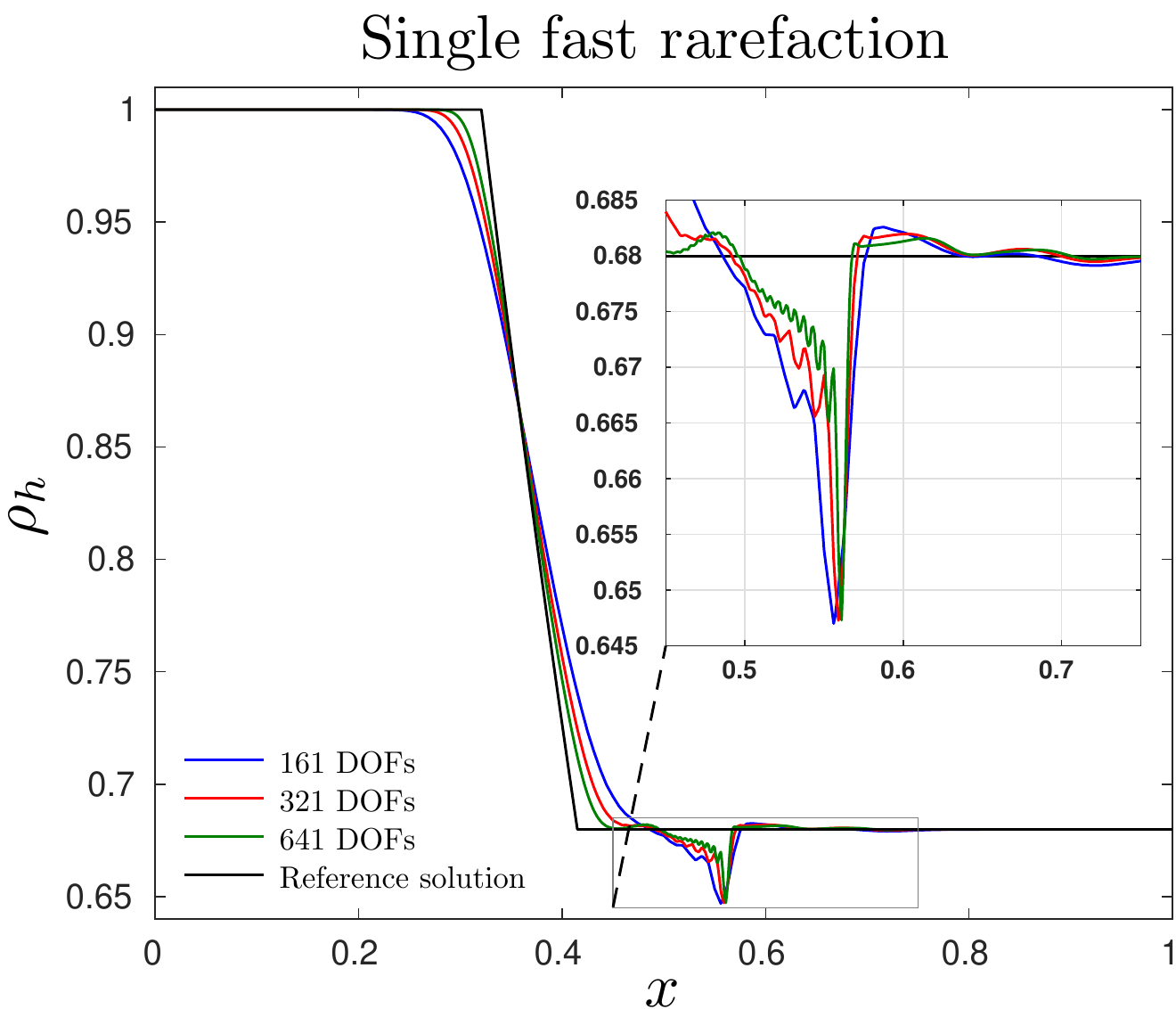}
         \caption{Resistive, $\kappa=0$}
     \end{subfigure}
     \hfill
     \begin{subfigure}{0.325\textwidth}
         \centering
         \includegraphics[width=\textwidth,viewport=100 454 480 780, clip=true]{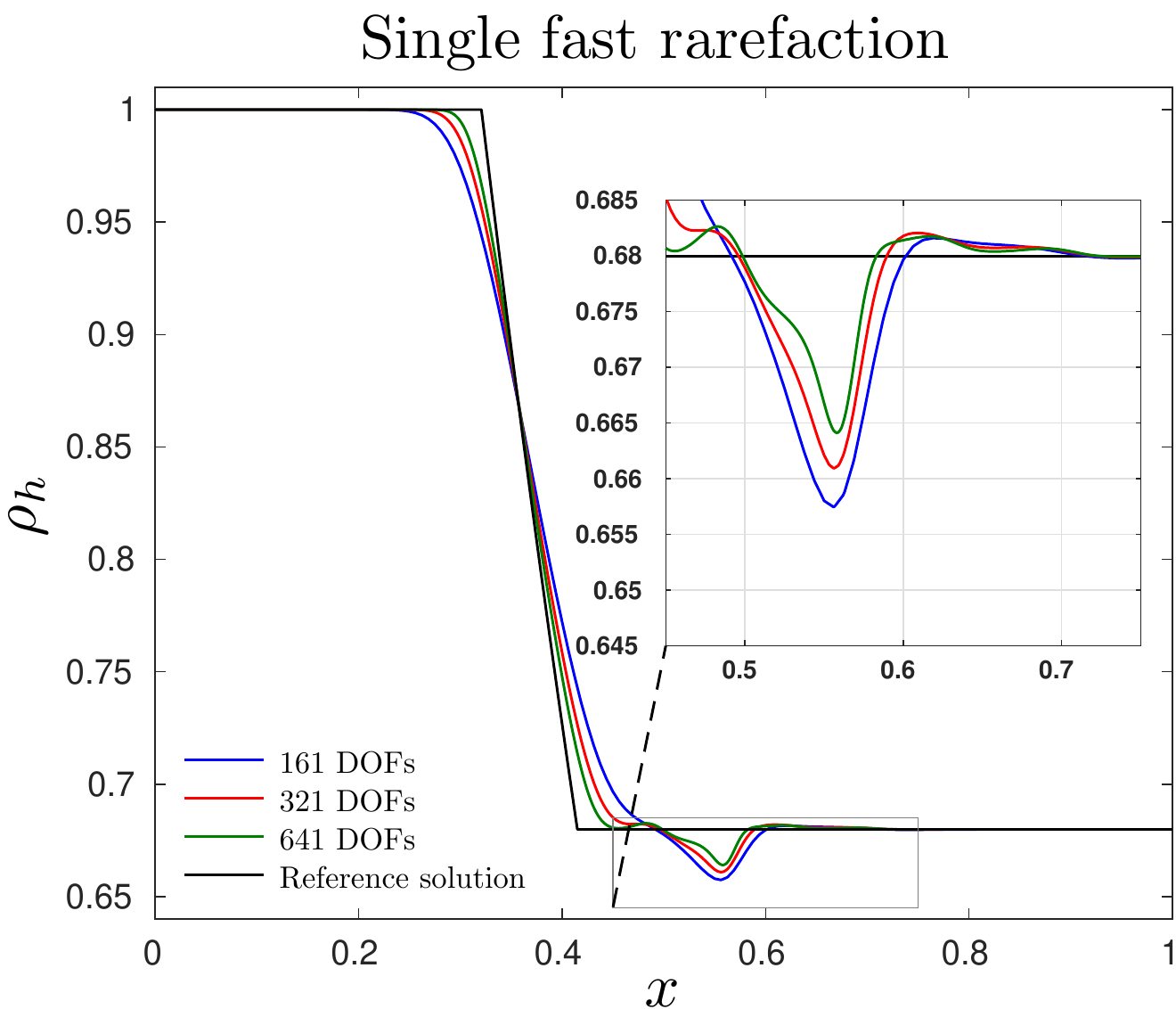}
         \caption{Resistive, $\kappa=1$}
     \end{subfigure}
     \hfill
     \begin{subfigure}{0.325\textwidth}
         \centering
         \includegraphics[width=\textwidth,viewport=100 454 480 780, clip=true]{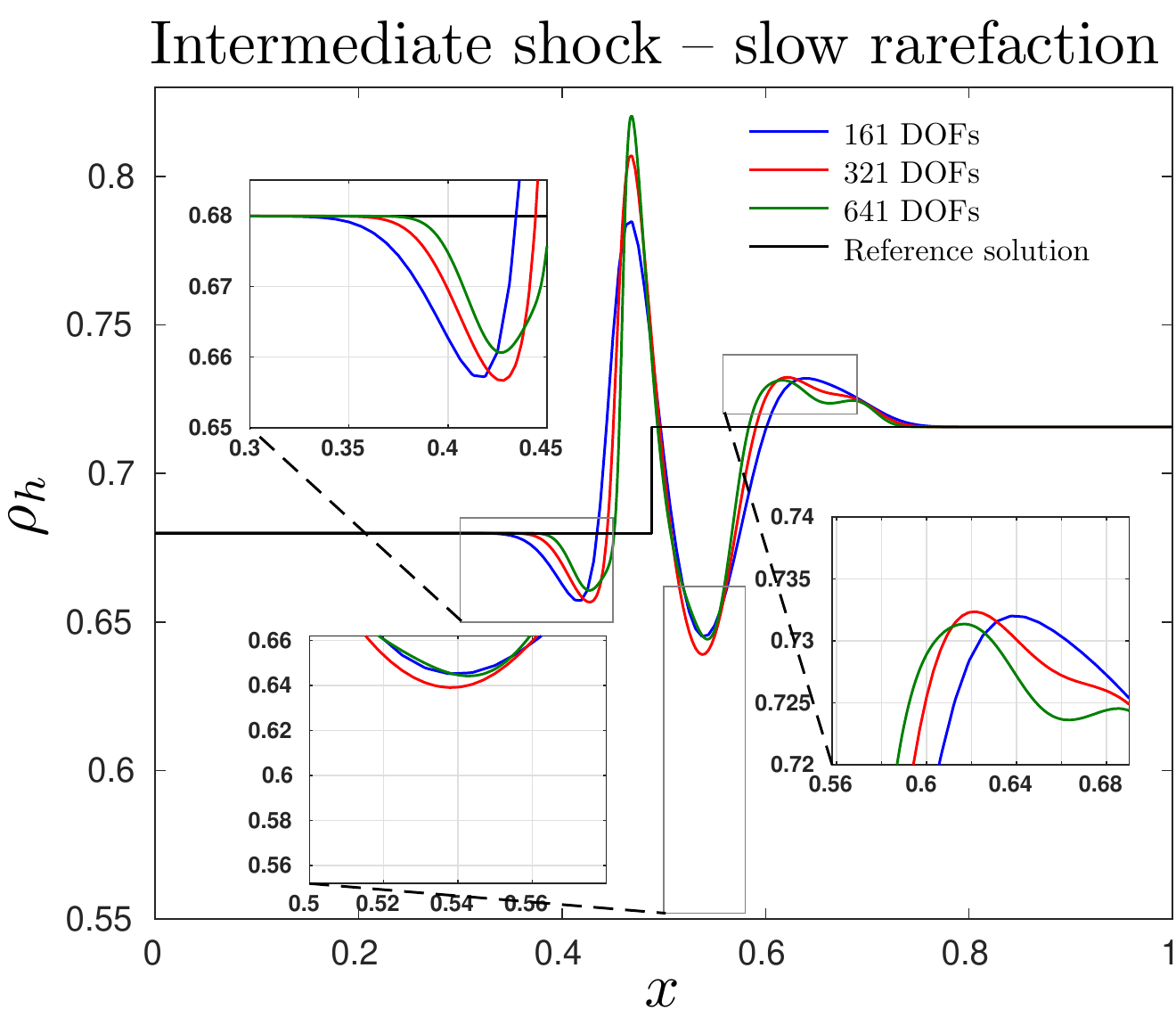}
         \caption{Monolithic}
     \end{subfigure}
     \hfill
     \begin{subfigure}{0.325\textwidth}
         \centering
         \includegraphics[width=\textwidth,viewport=100 454 480 780, clip=true]{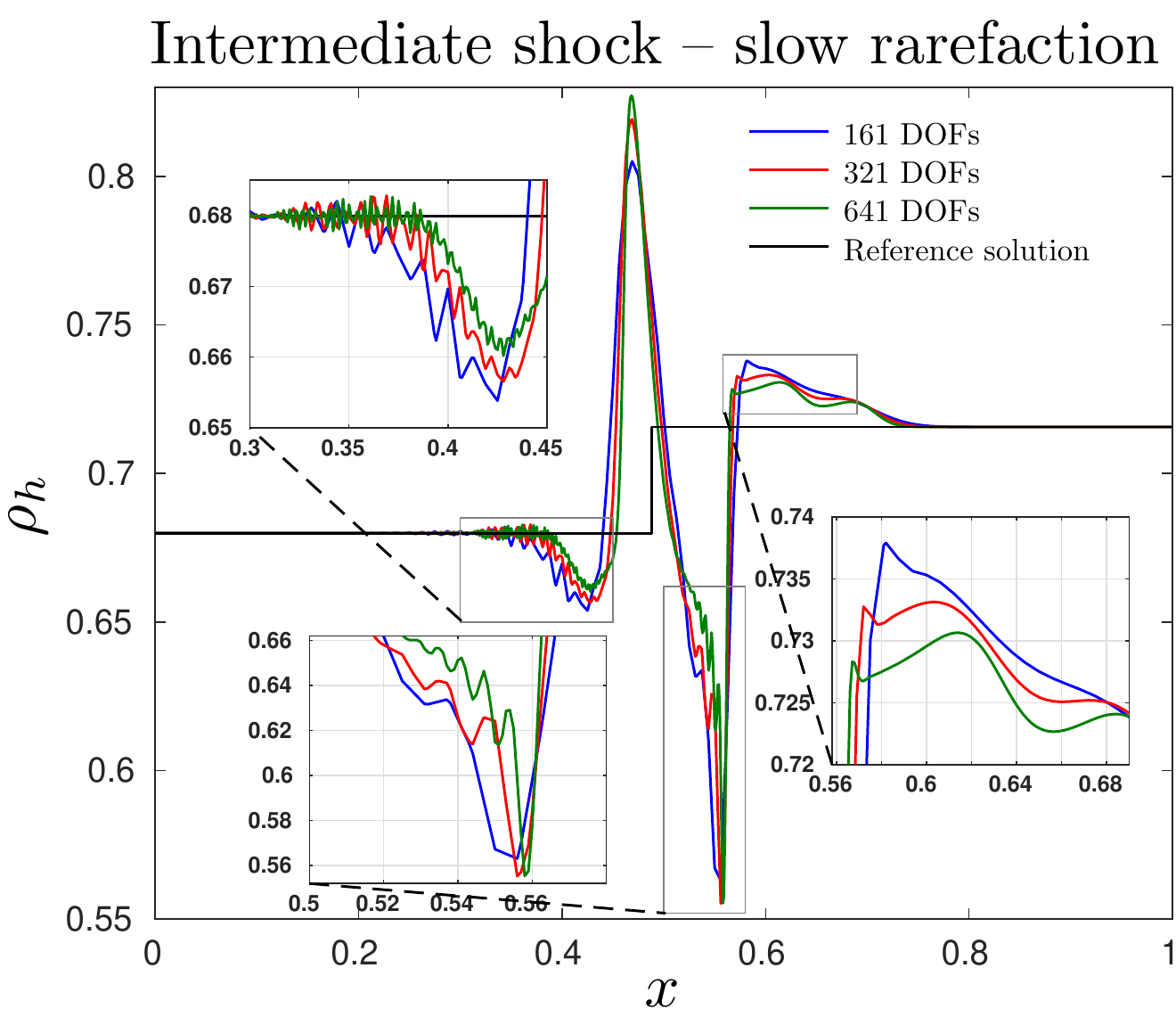}
         \caption{Resistive, $\kappa=0$}
     \end{subfigure}
     \hfill
     \begin{subfigure}{0.325\textwidth}
         \centering
         \includegraphics[width=\textwidth,viewport=100 454 480 780, clip=true]{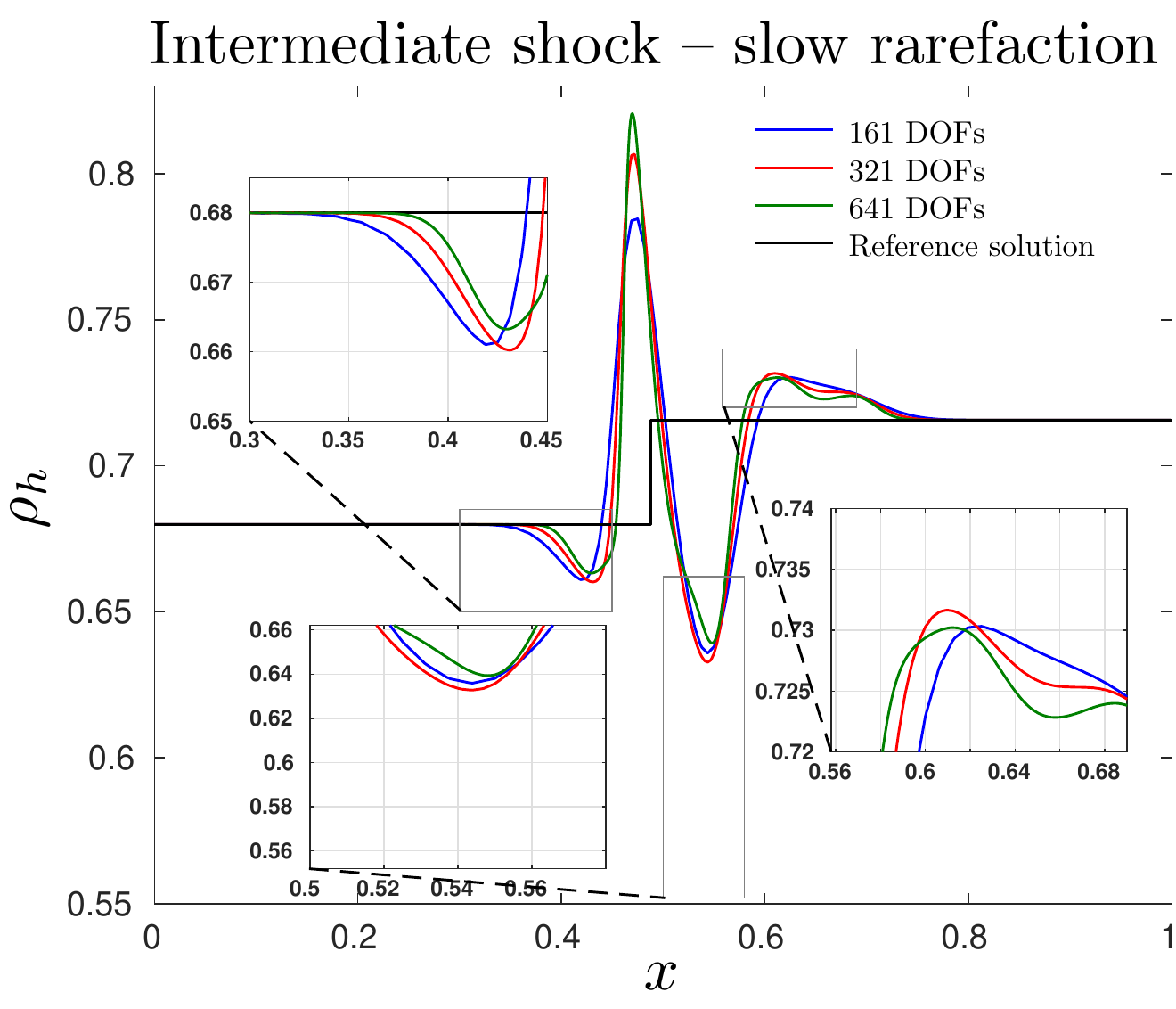}
         \caption{Resistive, $\kappa=1$}
     \end{subfigure}
     \hfill
     \begin{subfigure}{0.325\textwidth}
         \centering
         \includegraphics[width=\textwidth,viewport=100 454 480 780, clip=true]{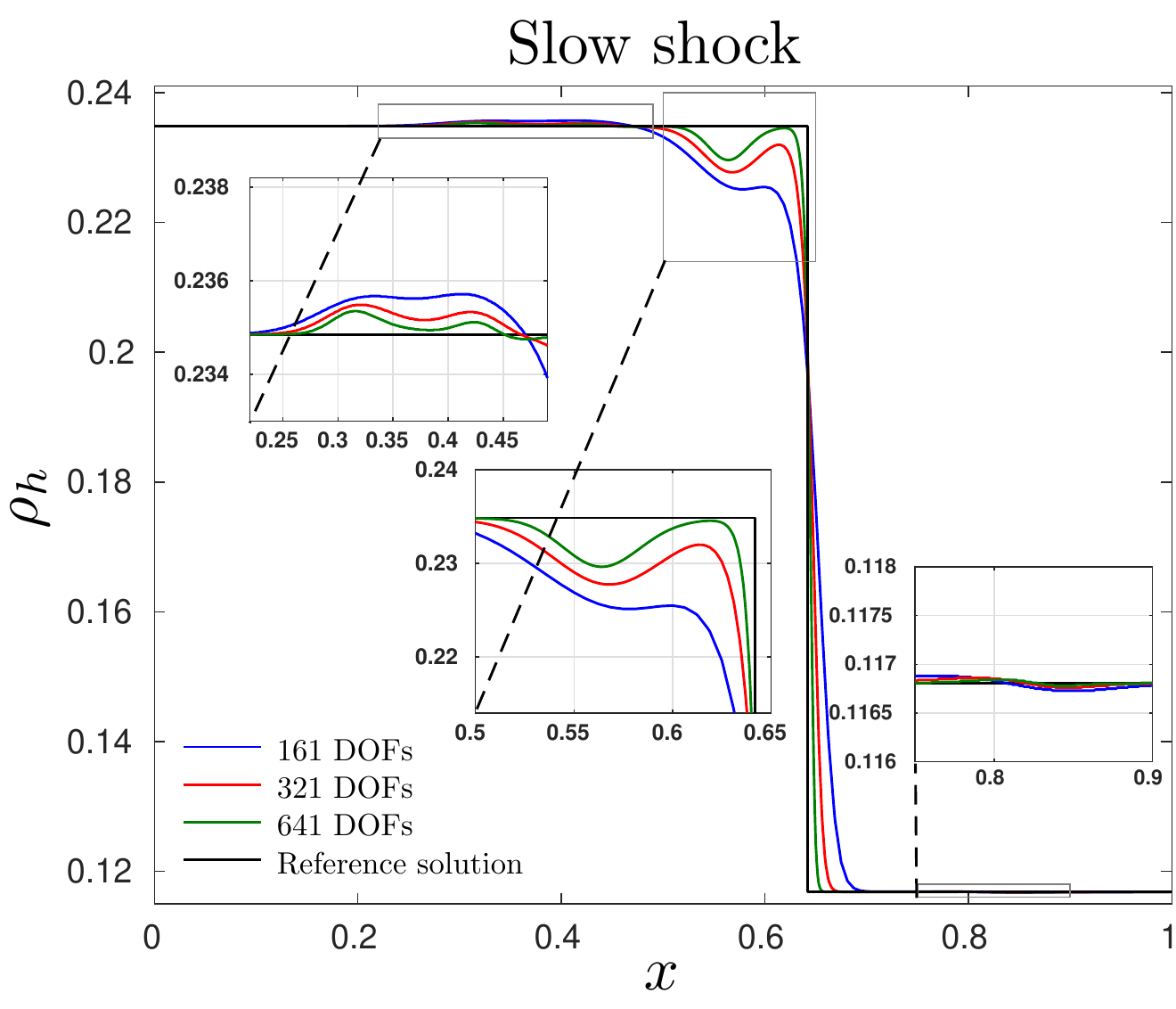}
         \caption{Monolithic}
     \end{subfigure}
     \hfill
     \begin{subfigure}{0.325\textwidth}
         \centering
         \includegraphics[width=\textwidth,viewport=100 454 480 780, clip=true]{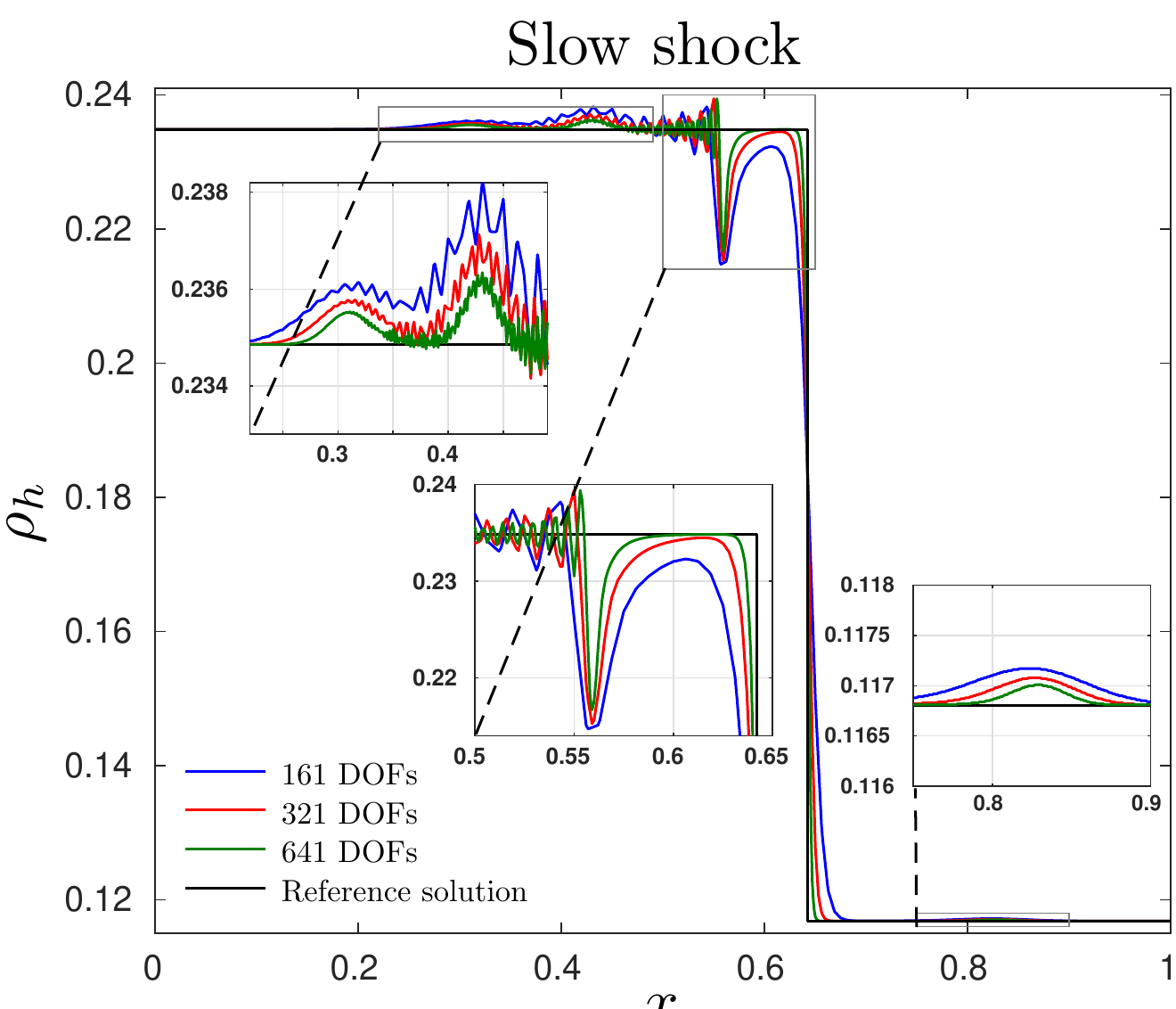}
         \caption{Resistive, $\kappa=0$}
     \end{subfigure}
     \hfill
     \begin{subfigure}{0.325\textwidth}
         \centering
         \includegraphics[width=\textwidth,viewport=100 454 480 780, clip=true]{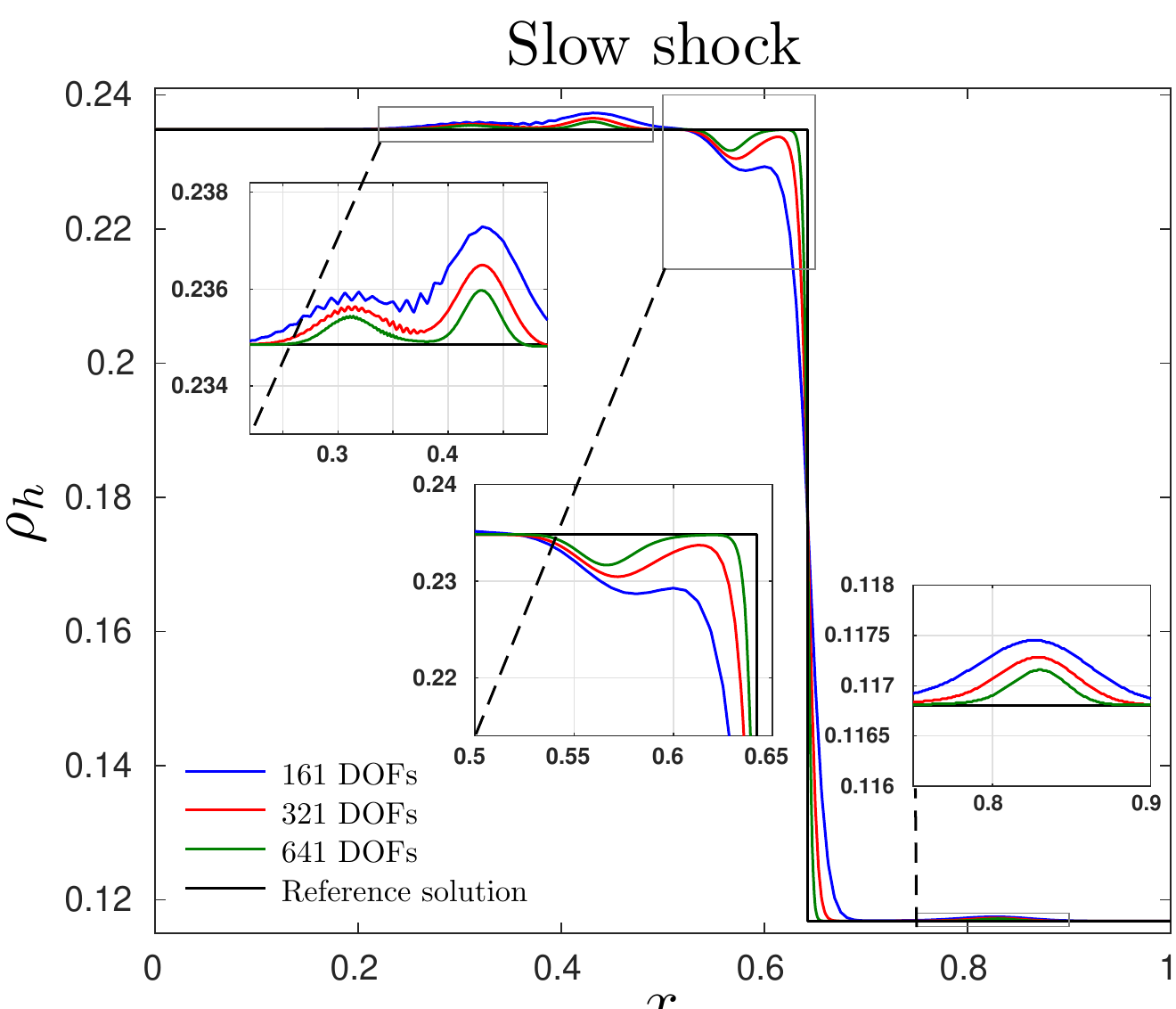}
         \caption{Resistive, $\kappa=1$}
     \end{subfigure}
     \hfill
     \caption{Simulation of single waves from Brio-Wu problem by the monolithic flux and the resistive MHD flux. The solutions are shown under multiple resolution levels: 161, 321, 641 DOFs. Final time $\widehat t = 0.1$.}
     \label{fig:briowu_single_waves}
\end{figure}
In Figure~\ref{fig:briowu_single_waves}(a)-(c), all the viscous fluxes produce undershoot around $x\in(0.5,0.6)$ for the fast rarefaction solutions. However, the undershoot produced by the monolithic flux is smaller in both $L^{\infty}$ and $L^{2}$ senses compared to the other two settings by the resistive MHD flux. The same can be said for the slow shock solutions in Figure~\ref{fig:briowu_single_waves}(g)-(i). The relevance of the undershoots and overshoots in the invariant domain set is however a difficult topic on its own, see \cite{Guermond_Popov_2016} for this discussion on the compressible Euler solutions. Figure~\ref{fig:briowu_single_waves}(d)-(f) show an intermediate shock followed by a slow rarefaction wave by the given initial solution. In this case, however, there is no clear benefits of using the monolithic flux over the resistive MHD flux with $\kappa=1$. Notice that the intermediate shock is not present in the reference solution in Figure~\ref{fig:briowu_single_waves}(d)-(f). The reason is that by default the exact Riemann solver \cite{Torrilhon_2002} does not capture this phenomenon. To this day, the existence of the intermediate shock is still a debatable topic, see e.g., \citep{Freistuhler_1995,Myong_1997,Takahashi_2014,Torrilhon_2002}.

\subsubsection{Discrete minimum principle}\label{sec:discrete_minimum_principle}
We investigate the minimum principle of the experimental CG solutions. For this study, we employ a commonly used entropy function for ideal gas,
\begin{equation}\label{eq:S}
S = \frac{\rho}{\gamma-1}s,
\end{equation}
where the thermodynamic entropy $s = \ln\frac{p}{\rho^\gamma}$ is obtained upon assuming $c_v = 1$ in \eqref{eq:s}. Because $s$ is linearly scaled by $c_v$, we can investigate the discrete minimum entropy principle using any choice of $c_v$.

\cpurple{\begin{remark}It has been shown by the authors of \cite{Guermond_Popov_Yang_2017} that positivity-preserving properties are impossible to achieve when the consistent mass matrix is used by CG methods. Positivity-preserving properties are important and are known to be connected to other invariant-domain preserving properties such as the minimum entropy principles. Therefore, in this Section~\ref{sec:discrete_minimum_principle}, to avoid any possible effects of the consistent mass matrix, we lump the mass matrix in the linear system yielded by \eqref{eq:ode}.\end{remark}}

A history plot of the entropy $s_h$ is presented in Figure~\ref{fig:briowu_entropy}. We note that $s_h \in \calQ_h$ is calculated pointwise from pressure and density, $s_{h,i}=\ln \frac{p_{h,i}}{\rho_{h,i}^\gamma}$ at every nodal point $i$. The viscosity coefficients $\epsilon, \mu, \nu$ are chosen to be first order, which is same to the previous test cases. We demonstrate further the dilemma of using whether $\kappa=0$ or $\kappa > 0$ for the resistive MHD flux in Section~\ref{sec:example_contact_wave}. The dilemma is that setting $\kappa > 0$ makes the flux inconsistent with contact wave as pointed out in Section~\ref{sec:example_contact_wave}, and violation of minimum entropy principle and entropy inequalities \citep{Guermond_Popov_2014}, while letting $\kappa = 0$ would lead to Gibbs phenomenon in numerical approximations \citep{Guermond_Popov_2014,Nazarov_Larcher_2017}. It is worth mentioning that in \cite{Bohm_2018}, the authors show that by adding the Powell terms \citep{Powell_et_al_1991} together with incorporating the generalized Lagrange multiplier (GLM) \citep{Dedner_et_al_2002} to the MHD system, the entropy inequality for the resulting resistive GLM-MHD system is recovered. Figure~\ref{fig:briowu_entropy}(a) compares the behavior of $s_h$ between using the monolithic flux $\bsfF_{\calV}^m(\bsfU)$ and the resistive MHD flux $\bsfF_{\calV}^r(\bsfU)$ when setting $\kappa=0$. Figure~\ref{fig:briowu_entropy}(b) shows a similar comparison but the parameter $\kappa$ is set to $1$. \cpurple{In Figure~\ref{fig:briowu_entropy}(a) and \ref{fig:briowu_entropy}(b), one can see the violation of the minimum principle since $\min_{\Omega}(s_h)$ is not monotonically increasing in time. In both cases of $\kappa=0$ and $\kappa=1$, there is a drastic fall of $\min_{\Omega}(s_h)$ in the start up phase. It gradually recovers after that. The portion of time when the minimum principle is violated seems to be shortened by refining the mesh. Only in the case of the monolithic flux, the discrete minimum entropy principle is fulfilled. We note that the minimum principle would not be satisfied if the consistent mass matrix was used. In the history plots of $\min_{\Omega}(s_h)$ by the monolithic flux in Figure~\ref{fig:briowu_entropy}(a) and \ref{fig:briowu_entropy}(b), the minimum principle violation is in the order of machine epsilon. Even though in this paper we have proved the minimum principle in the continuum case along with the shown numerical evidence, much work is needed to design a discretization that provably ensures this property at the fully discrete level. Constructing such a CG method is within our ongoing works.}

\begin{figure}[h!]
     \centering
     \begin{subfigure}{0.49\textwidth}
         \centering
         \includegraphics[width=\textwidth,viewport=100 485 458 765, clip=true]{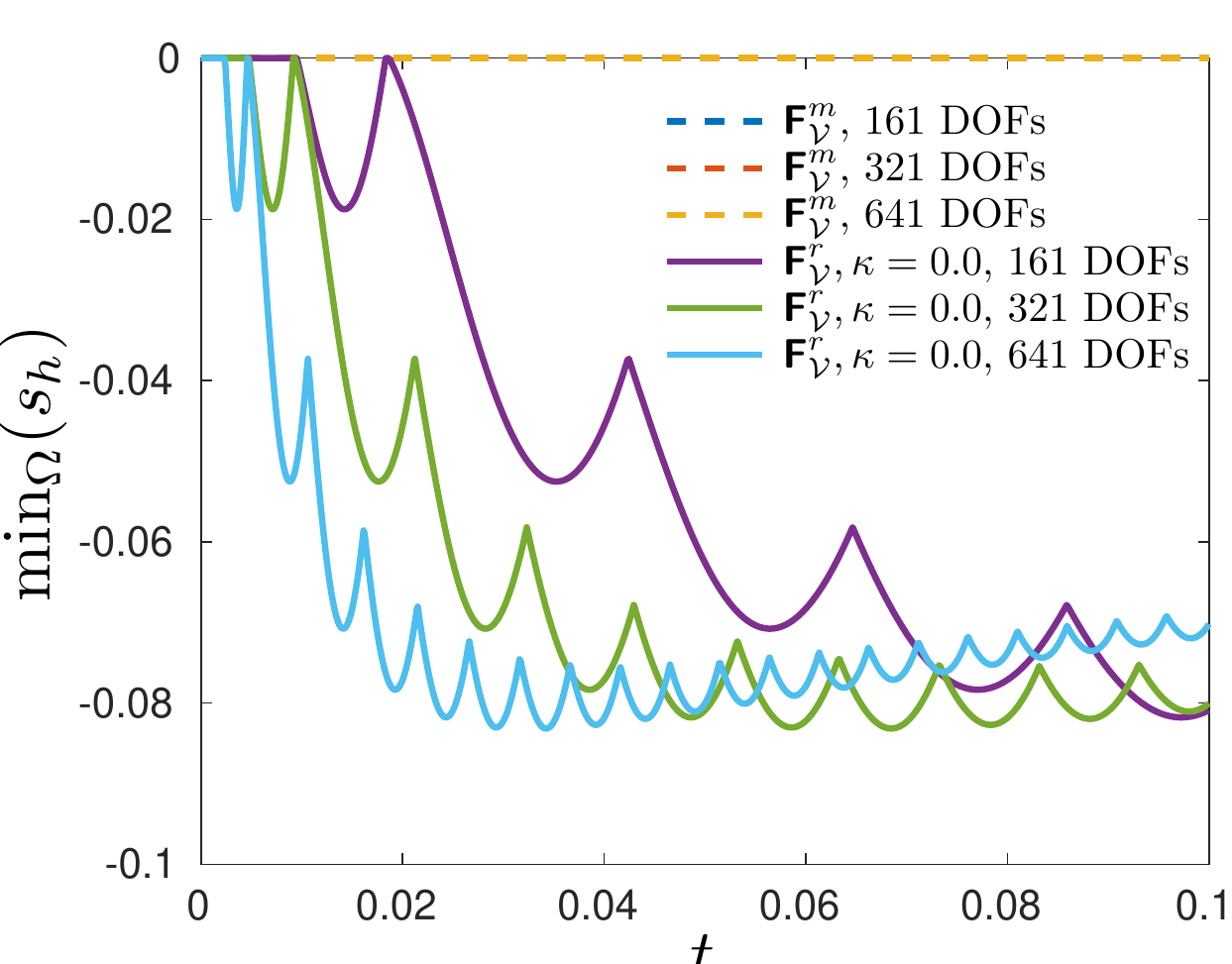}
         \caption{\cblue{Monolithic vs resistive MHD flux, $\kappa=0$}}
     \end{subfigure}
     \hfill
     \begin{subfigure}{0.49\textwidth}
         \centering
         \includegraphics[width=\textwidth,viewport=100 485 458 765, clip=true]{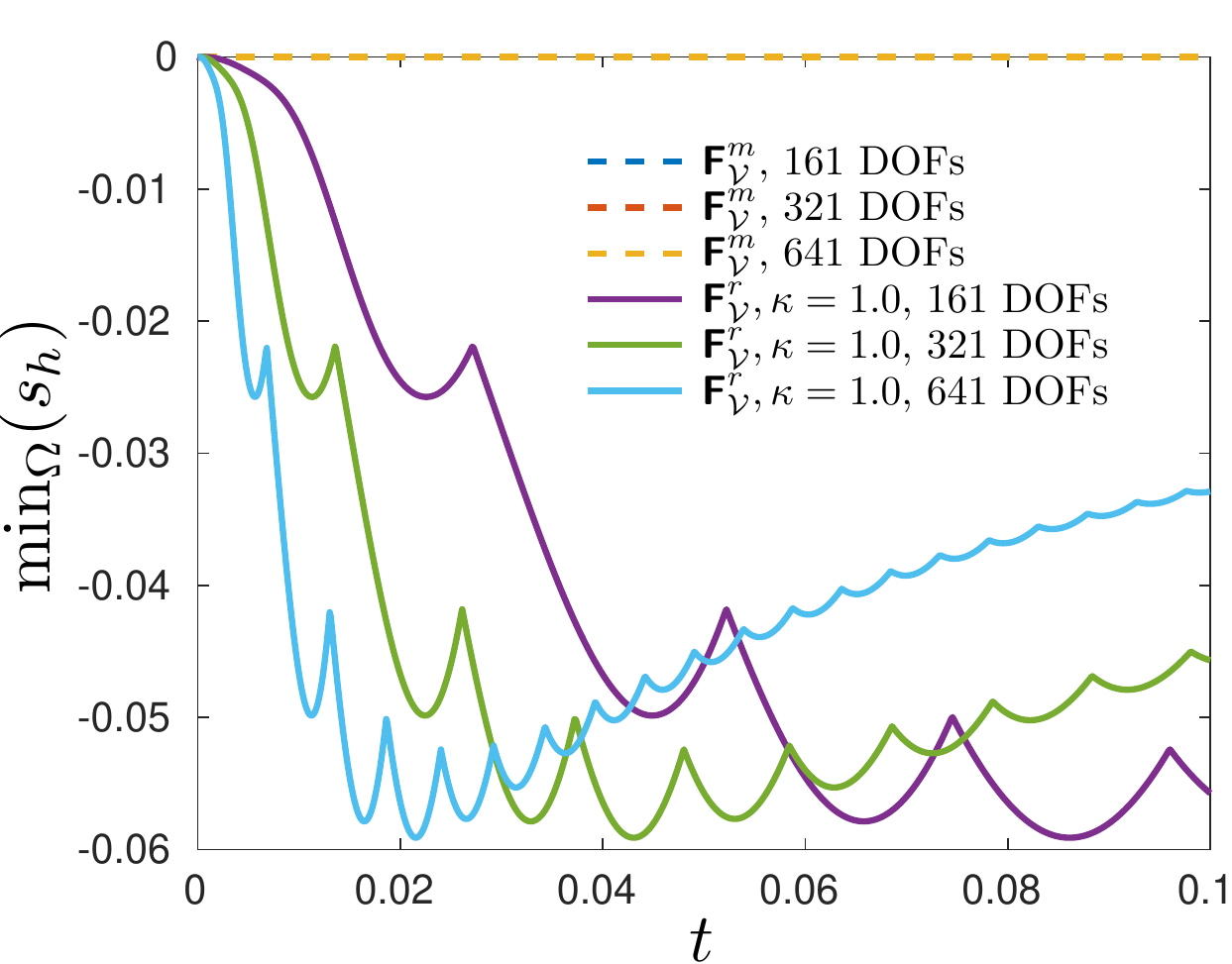}
         \caption{\cblue{Monolithic vs resistive MHD flux, $\kappa=1$}}
     \end{subfigure}
     \caption{\cblue{Investigation of the discrete minimum entropy principle on the Brio-Wu problem comparing the monolithic and the resistive MHD flux. The value $\min_\Omega(s_h)$ are captured at 1000 snapshots in time under multiple resolution levels: 161, 321, 641 DOFs. The range of $\min_\Omega(s_h)$ for the monolithic flux is in the order of machine epsilon: $\min_\Omega(s_h)\in$[-4.6E-15, 0] using 161 DOFs, $\min_\Omega(s_h)\in$[-5.2E-15,0] using 321 DOFs, and $\min_\Omega(s_h)\in$[-6.3E-15,0] using 641 DOFs.}}
     \label{fig:briowu_entropy}
\end{figure}

\subsection{Entropy viscosity method}
For high-order stabilization, the entropy viscosity method was proposed by \citep{Guermond_2008,Guermond_2011}. For systems of conservation laws, entropy is a conserved quantity in smooth solution regions, which is known as the ``entropy equation''. In presence of shocks and discontinuities, this equation becomes an inequality. The idea of the entropy viscosity method is to use entropy residual -- the violation of the entropy equation as an indicator of the shock locations. Effectively, it adds enough viscosity to stabilize the discontinuities, while being negligible in the smooth regions to preserve high-order accuracy away from the discontinuities. We employ the entropy function \eqref{eq:S} to calculate the entropy residual. In each time step, we seek the entropy residual $R_h$ in the finite element space $\calC^1(\mR^+,\calQ_h)$. A robust way to construct $R_h$ has shown to be by a nodal-based approach \cite{Nazarov_Larcher_2017},
\[
R_{h,i}(t) := \sum_{K\in\calT_h} \frac{1}{|K|} \int_{K} | \p_tS_h+\DIV(\bu_hS_h)|\varphi_i \ud\bx,
\]
where $|K|$ is the volume/area of the element $K$, and $S_h$ is computed as in Section \ref{sec:discrete_minimum_principle}. In order to construct the necessary amount of artificial viscosity to stabilize the solution, we need to compute a low-order viscosity coefficient function $\epsilon_h^L\in\calQ_h$ proportional to an approximate maximum wave speed, and a high-order viscosity function $\epsilon_h^H\in\calQ_h$ based on the entropy residual $R_h$.

The low-order viscosity is calculated as described in Section \ref{sec:example_contact_wave}, for every node $i$, we compute
\[
\epsilon_{h,i}^L = c_{\max}h_{i}\max_{i=1,8}\vert\Lambda_i\vert,
\]
where $\max_{i=1,8}\vert\Lambda_i\vert$ approximates the maximum wave speed, $h_i$ is the corresponding nodal value of $h(x)$ by \eqref{eq:h_h}, and the parameter $c_{\max}$ is set to be $0.5$.

The high-order viscosity is set to be proportional to the normalized entropy residual,
\[
\epsilon_{h,i}^H = c_E h_{i}^2 \frac{\vert R_h \vert}{\|\overline{R_h} - R_h\|_{\infty,\Omega}},
\]
where $\|\overline{R_h} - R_h\|_{\infty,\Omega}$ normalizes the unit of $\vert R_h \vert$, the maximum variance $\overline{R_h}:=\|R_h-\text{mean}(R_h)\|_{\infty,\Omega}$, and $c_E$ is set to be 1.

The final artificial viscosity coefficient at each nodal $i$ is assembled as
\[
\epsilon_{h,i} = \min(\epsilon_{h,i}^L,\epsilon_{h,i}^H).
\]

In the next sections, we investigate the accuracy and the shock-capturing capability of the entropy viscosity method when it is incorporated in the $\epsilon$ coefficient in the monolithic viscous flux \eqref{eq:monolithic_flux}.

\subsubsection{Accuracy test}
Consider a periodic smooth vortex problem on a rectangle domain $\Omega = [-10, 10] \times [-10, 10]$. The reference solution is a stationary flow with a vortex perturbation
\[
(\rho(t), \bu(t), p(t), \bB(t)) =(\rho_0, \bu_0+\delta\bu, p_0+\delta p, \bB_0+\delta\bB),
\]
where
\begin{align*}
\rho_0 &= 1, & & \\
\bu_0 &= (1,1), &\delta \bu &= \frac{\mu}{\pi\sqrt 2}e^{(1-r^2)/2}(-r_2,r_1),\\
p_0 &=0, &\delta p &= -\frac{\mu^2(1+r^2)e^{1-r^2}}{8\pi^2}, \\
\bB_0 &=(0.1,0.1), & \delta \bB & = \frac{\mu e^{(1-r^2)/2}}{2\pi}(-r_2,r_1),
\end{align*}
the vortex radius is $r=\sqrt{r_1^2+r_2^2}$, $(r_1,r_2)=(x,y)-\bu_0t$, and the vortex strength is $\mu=5.389489439$. The adiabatic constant is  $\gamma=\frac{5}{3}$. The errors measured at final time $\widehat t=0.05$.

\begin{table}[h!]
    \centering
    \caption{Smooth vortex problem. L$^1$, L$^2$ error of the finite element solution and convergence rates of velocity $\bu$ and magnetic field $\bB$ at final time $\widehat t=0.05$ using $\polP_1$ elements.}
    \label{table:convergence_vortex_P1}
    \begin{adjustbox}{max width=\textwidth}
    \begin{tabular}{c|c|c|c|c|c|c|c|c}
    \hline
    \multirow{2}{*}{\#DOFs} &  \multicolumn{4}{c|}{Entropy viscosity solution $\bu$} &  \multicolumn{4}{c}{Unregularized Galerkin solution $\bu$}  \\ \cline{2-9}
     {}   &       L$^1$   &    Rate   &       L$^2$   &    Rate &       L$^1$   &    Rate   &       L$^2$   &    Rate   \\ \hline
     7442 & 6.11E-04 &    --  &   3.37E-03 &     -- & 5.89E-04 &   --   &   3.24E-03 &   --   \\ 
    29282 & 1.51E-04 &   2.01 &   8.36E-04 &   2.01 & 1.48E-04 &   1.99 &   8.18E-04 &   1.99 \\ 
   116162 & 3.75E-05 &   2.01 &   2.07E-04 &   2.01 & 3.71E-05 &   2.00 &   2.05E-04 &   2.00 \\ 
   462722 & 9.33E-06 &   2.01 &   5.16E-05 &   2.01 & 9.28E-06 &   2.00 &   5.13E-05 &   2.00 \\ \hline
   \multirow{2}{*}{\#DOFs} &  \multicolumn{4}{c|}{Entropy viscosity solution $\bB$} &  \multicolumn{4}{c}{Unregularized Galerkin solution $\bB$}  \\ \cline{2-9}
     {}   &       L$^1$   &    Rate   &       L$^2$   &    Rate &       L$^1$   &    Rate   &       L$^2$   &    Rate   \\ \hline
     7442 & 2.47E-02 &   --   &   2.77E-02 &   --   & 2.35E-02 &   --   &   2.60E-02 &   --   \\ 
    29282 & 6.10E-03 &   2.04 &   6.80E-03 &   2.05 & 5.90E-03 &   2.01 &   6.56E-03 &   2.01 \\ 
   116162 & 1.50E-03 &   2.03 &   1.68E-03 &   2.03 & 1.48E-03 &   2.01 &   1.64E-03 &   2.01 \\ 
   462722 & 3.73E-04 &   2.02 &   4.16E-04 &   2.02 & 3.70E-04 &   2.01 &   4.11E-04 &   2.01 \\ \hline
    \end{tabular}
    \end{adjustbox}
\end{table}

\begin{table}[h!]
    \centering
    \caption{Smooth vortex problem. L$^1$, L$^2$ error of the finite element solution and convergence rates of velocity $\bu$ and magnetic field $\bB$ at final time $\widehat t=0.05$ using $\polP_2$ elements.}
    \label{table:convergence_vortex_P2}
    \begin{adjustbox}{max width=\textwidth}
    \begin{tabular}{c|c|c|c|c|c|c|c|c}
    \hline
    \multirow{2}{*}{\#DOFs} &  \multicolumn{4}{c|}{Entropy viscosity solution $\bu$} &  \multicolumn{4}{c}{Unregularized Galerkin solution $\bu$}  \\ \cline{2-9}
     {}   &       L$^1$   &    Rate   &       L$^2$   &    Rate &       L$^1$   &    Rate   &       L$^2$   &    Rate   \\ \hline
     7442 & 1.92E-04 &   --   &   1.14E-03 &   --    & 1.91E-04 &   --   &   1.13E-03 &   --   \\ 
    29282 & 3.41E-05 &   2.50 &   2.01E-04 &   2.50  & 3.41E-05 &   2.48 &   2.02E-04 &   2.49 \\ 
   116162 & 7.96E-06 &   2.10 &   4.69E-05 &   2.10  & 7.98E-06 &   2.10 &   4.70E-05 &   2.10 \\ 
   462722 & 1.99E-06 &   2.00 &   1.17E-05 &   2.00  & 2.00E-06 &   2.00 &   1.17E-05 &   2.00 \\ \hline
   \multirow{2}{*}{\#DOFs} &  \multicolumn{4}{c|}{Entropy viscosity solution $\bB$} &  \multicolumn{4}{c}{Unregularized Galerkin solution $\bB$}  \\ \cline{2-9}
     {}   &       L$^1$   &    Rate   &       L$^2$   &    Rate &       L$^1$   &    Rate   &       L$^2$   &    Rate   \\ \hline
     7442 & 7.95E-03 &   --   &   9.23E-03 &   --   & 7.83E-03 &   --   &   9.08E-03 &   -- \\ 
    29282 & 1.40E-03 &   2.53 &   1.54E-03 &   2.61 & 1.40E-03 &   2.51 &   1.54E-03 &   2.59 \\ 
   116162 & 3.18E-04 &   2.16 &   3.32E-04 &   2.23 & 3.18E-04 &   2.15 &   3.32E-04 &   2.22 \\ 
   462722 & 7.75E-05 &   2.04 &   8.01E-05 &   2.06 & 7.76E-05 &   2.04 &   8.03E-05 &   2.06 \\ \hline
    \end{tabular}
    \end{adjustbox}
\end{table}

\begin{table}[h!]
    \centering
    \caption{Smooth vortex problem. L$^1$, L$^2$ error of the finite element solution and convergence rates of velocity $\bu$ and magnetic field $\bB$ at final time $\widehat t=0.05$ using $\polP_3$ elements.}
    \label{table:convergence_vortex_P3}
    \begin{adjustbox}{max width=\textwidth}
    \begin{tabular}{c|c|c|c|c|c|c|c|c}
    \hline
    \multirow{2}{*}{\#DOFs} &  \multicolumn{4}{c|}{Entropy viscosity solution $\bu$} &  \multicolumn{4}{c}{Unregularized Galerkin solution $\bu$}  \\ \cline{2-9}
     {}   &       L$^1$   &    Rate   &       L$^2$   &    Rate &       L$^1$   &    Rate   &       L$^2$   &    Rate   \\ \hline
     7442 & 1.13E-04 &   -- &   5.70E-04 &   --  &  1.14E-04 &   -- &   5.87E-04 &   -- \\ 
    29282 & 8.09E-06 &   3.80 &   4.89E-05 &   3.54  &  8.11E-06 &   3.82 &   4.95E-05 &   3.57  \\ 
   116162 & 5.29E-07 &   3.93 &   3.95E-06 &   3.63  &  5.18E-07 &   3.97 &   3.89E-06 &   3.67 \\ 
   462722 & 3.93E-08 &   3.75 &   3.96E-07 &   3.32  &  3.42E-08 &   3.92 &   3.67E-07 &   3.41 \\ \hline
   \multirow{2}{*}{\#DOFs} &  \multicolumn{4}{c|}{Entropy viscosity solution $\bB$} &  \multicolumn{4}{c}{Unregularized Galerkin solution $\bB$}  \\ \cline{2-9}
     {}   &       L$^1$   &    Rate   &       L$^2$   &    Rate &       L$^1$   &    Rate   &       L$^2$   &    Rate   \\ \hline
     7442 & 4.43E-03 &   --   &   4.43E-03 &  --  & 4.44E-03 &   --   &   4.51E-03 &   --  \\ 
    29282 & 2.82E-04 &   4.02 &   2.94E-04 & 3.96 & 2.81E-04 &   4.03 &   2.96E-04 &   3.98  \\ 
   116162 & 1.84E-05 &   3.96 &   2.35E-05 & 3.67 & 1.79E-05 &   3.99 &   2.26E-05 &   3.73 \\ 
   462722 & 1.60E-06 &   3.54 &   2.64E-06 & 3.16 & 1.36E-06 &   3.73 &   2.15E-06 &   3.41 \\ \hline
    \end{tabular}
    \end{adjustbox}
\end{table}

\begin{figure}[h!]
     \centering
     \begin{subfigure}{0.49\textwidth}
         \centering
         \includegraphics[width=\textwidth,viewport=100 420 455 775, clip=true]{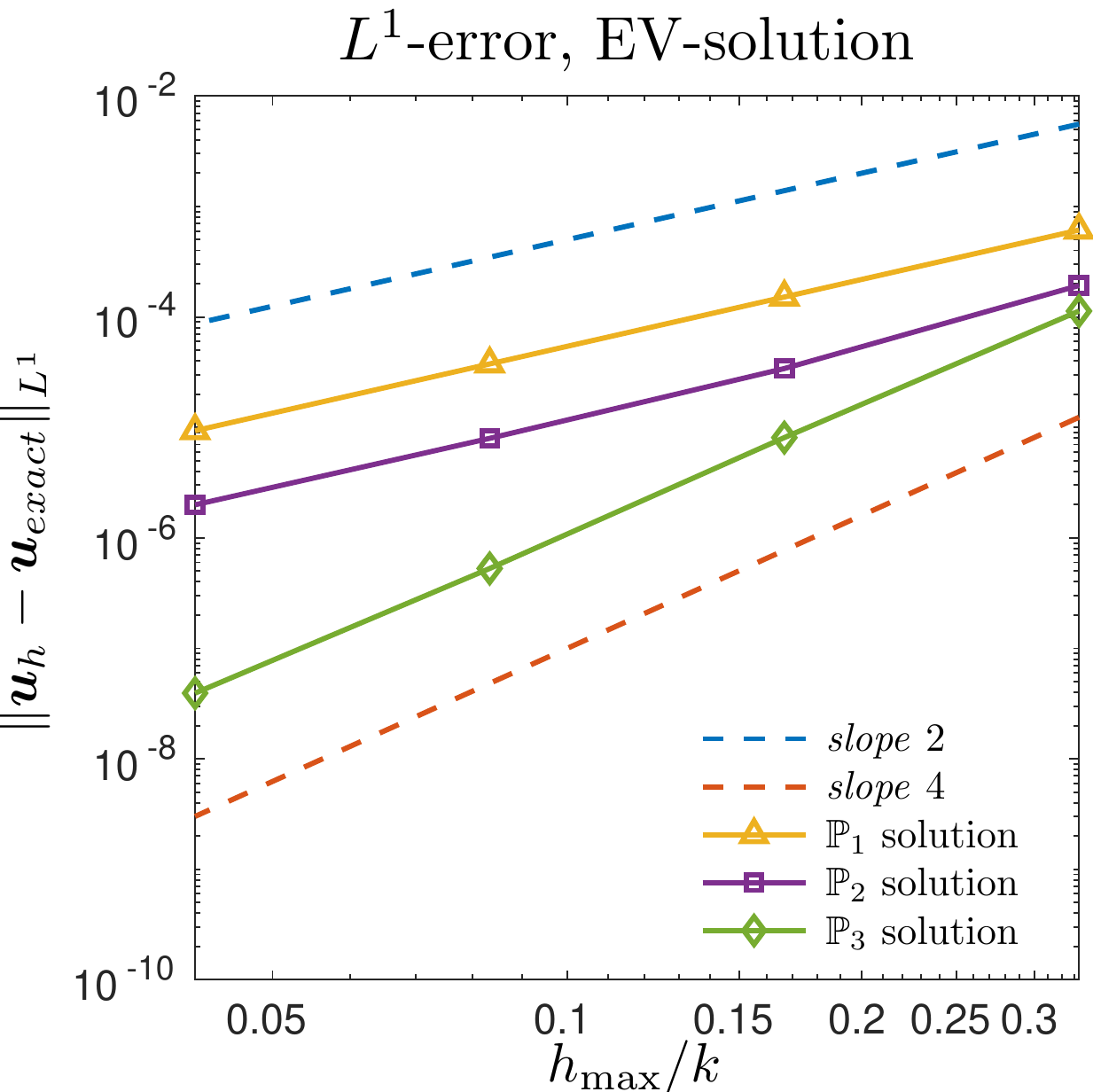}
     \end{subfigure}
     \hfill
     \begin{subfigure}{0.49\textwidth}
         \centering
         \includegraphics[width=\textwidth,viewport=100 420 455 775, clip=true]{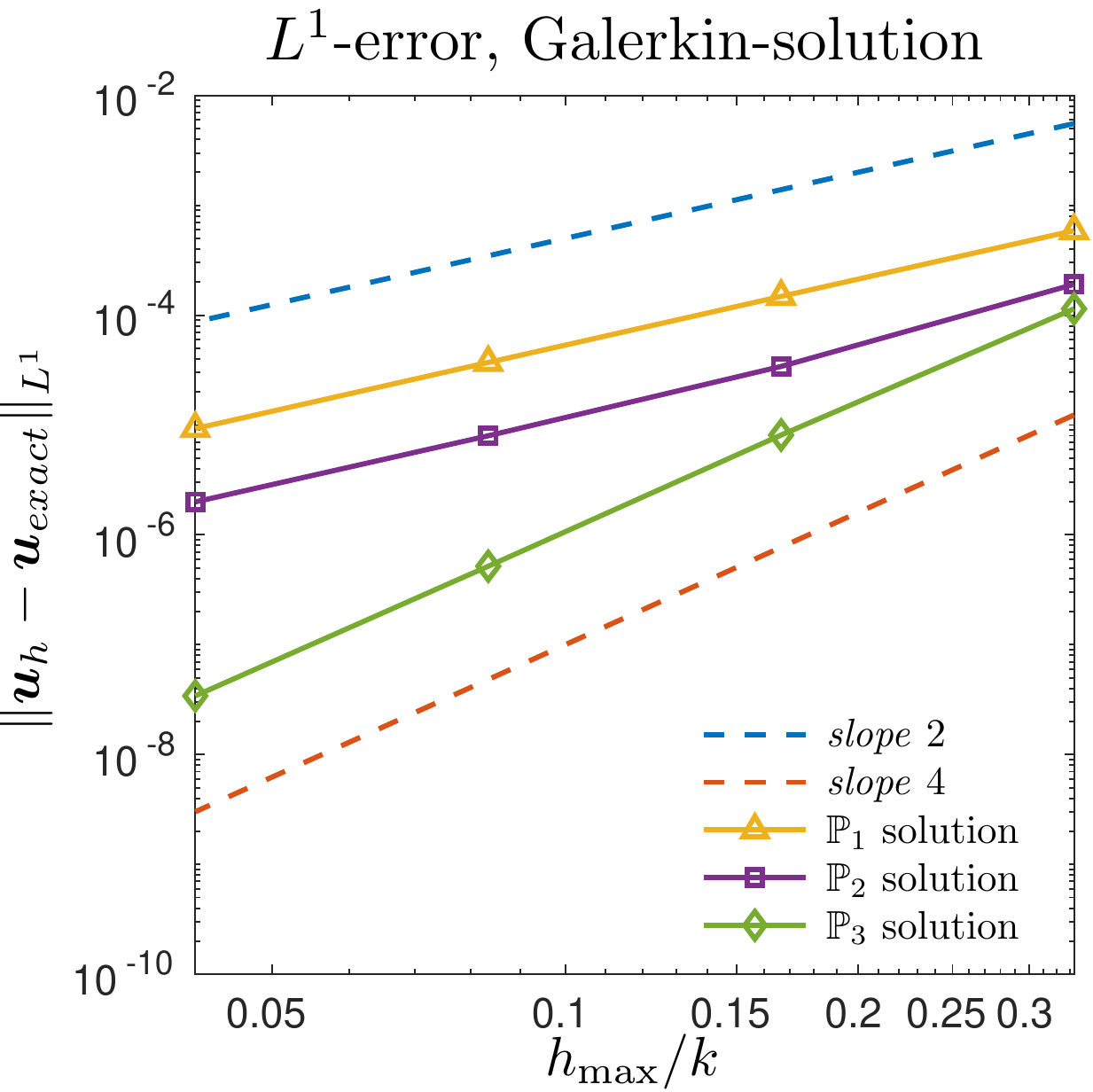}
     \end{subfigure}
     \caption{Accuracy test on the smooth vortex problem: convergence history for velocity $\bu_h$ against the exact solution. The entropy viscosity method using the monolithic flux does not destroy the accuracy of high-order solutions.}
     \label{fig:conv_history}
\end{figure}

\subsubsection{Brio-Wu problem with the EV method}

\cblue{We again consider the Brio-Wu problem described in Section~\ref{sec:single_waves}. It can be seen that due to being a first order method, the solutions presented in Section~\ref{sec:single_waves} are over dissipative for the given number of degrees of freedom (DOFs). In this section, we want to demonstrate the convergence of the EV method and its efficiency as a shock capturing technique. The problem settings are kept same but the viscosity coefficients are now calculated by the EV method. The result is presented in Figure~\ref{fig:briowu_EV}. For comparison, we show a solution under 641 nodes as it has been shown in Section~\ref{sec:single_waves}. One can see that under the same modest number of nodes, the EV solution is a much better approximation of the expected solution compared to the first order solution.}

\begin{figure}[!ht]
     \centering
     \includegraphics[width=0.59\textwidth,viewport=100 454 480 780, clip=true]{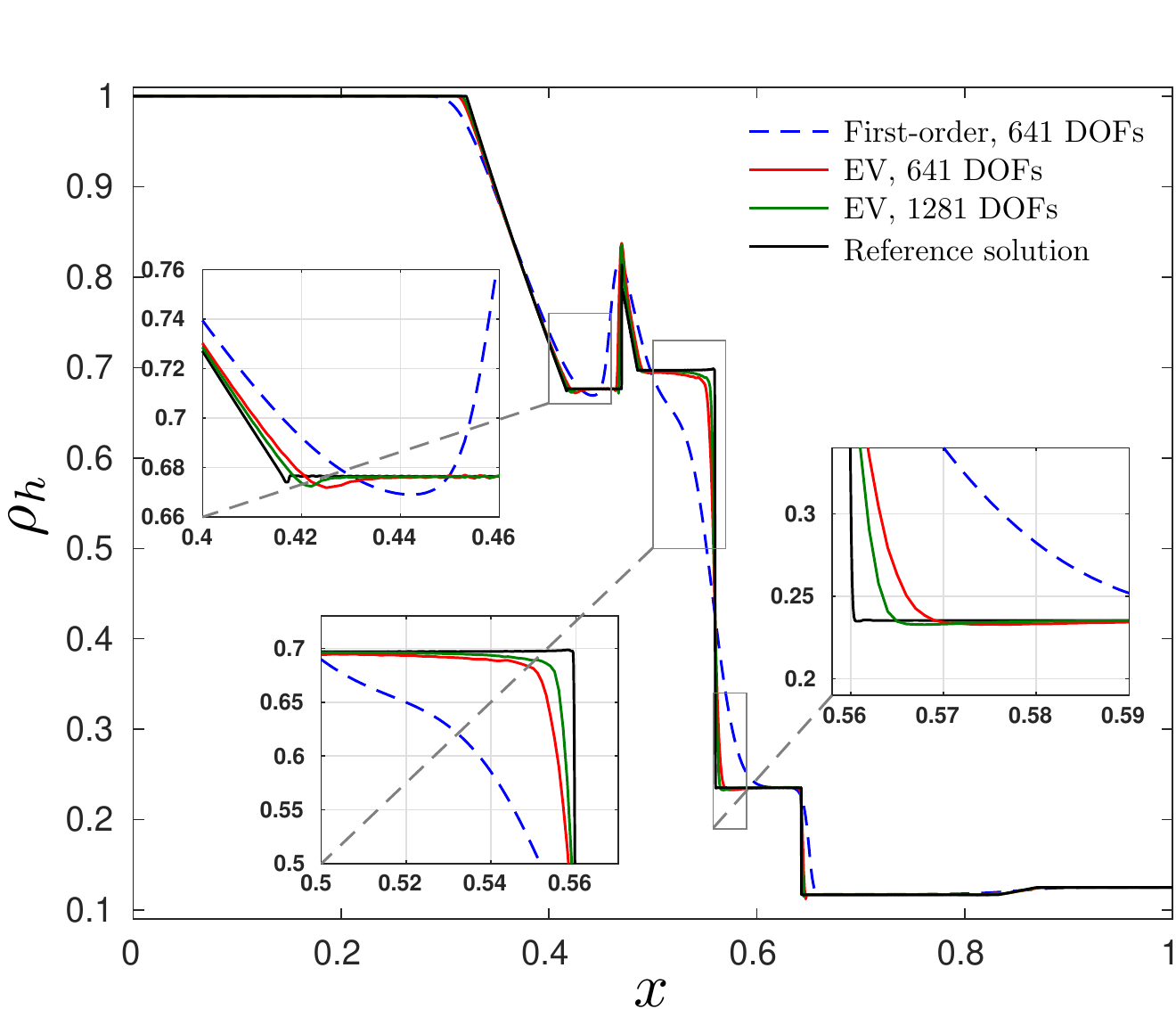}
     \caption{\cblue{EV solutions to the Brio-Wu problem with 641, 1281 DOFs. Final time $\widehat t = 0.1$. For comparison, the first order solution with 641 DOFs in Section~\ref{sec:single_waves} is presented by the blue dashed line. The red line presents the EV solution obtained under the same number of 641 DOFs. The green line presents the EV solution on a twice finer mesh.}}
     \label{fig:briowu_EV}
\end{figure}

\subsubsection{Orszag-Tang problem, \citep{Orszag_Tang_1979}}
The popular 2D Orszag-Tang benchmark is investigated in this section. We consider the unit square domain $\Omega = [0,1] \times [0,1]$ with periodic boundaries in both $x$- and $y$- directions. The initial solution is given as follows,
\[
(\rho_0, \bu_0, p_0, \bB_0)= \left(\frac{25}{36\pi},(-\sin(2\pi y),\sin(2\pi x)),\frac{5}{12\pi},\left(-\frac{\sin(2\pi y)}{\sqrt{4\pi}},\frac{\sin(4\pi x)}{\sqrt{4\pi}}\right)\right).
\]
The adiabatic constant is $\gamma=\frac{5}{3}$. We compute the solution with the entropy viscosity method using the monolithic parabolic flux. The density solution $\rho_h$ and the artificial viscosity $\mu_h$ using $\polP_3$ elements are respectively shown at time $t = 0.5$ and $t=1.0$ in Figure \ref{fig:orszag-tang_T05} and \ref{fig:orszag-tang_T1}. It can be seen that the artificial viscosity locally tracks the shocks and effectively adds enough viscosity to stabilize the solution. At $t=1.0$, the behavior of the Orszag-Tang solution is considered turbulence \citep{Tricco_2016}. It is worth noting that simulation of the Orszag-Tang problem after $t = 0.5$ is a challenging task that is not straightforwardly achieved by many numerical methods due to the solution being prone to divergence blowups, see \cite{Dao_2021,Guillet_2019}. Nonetheless, the solution at $t=1.0$ is still well captured by the described viscosity method.

\begin{figure}[h!]
     \centering
     \begin{subfigure}{0.49\textwidth}
         \centering
         \includegraphics[width=\textwidth]{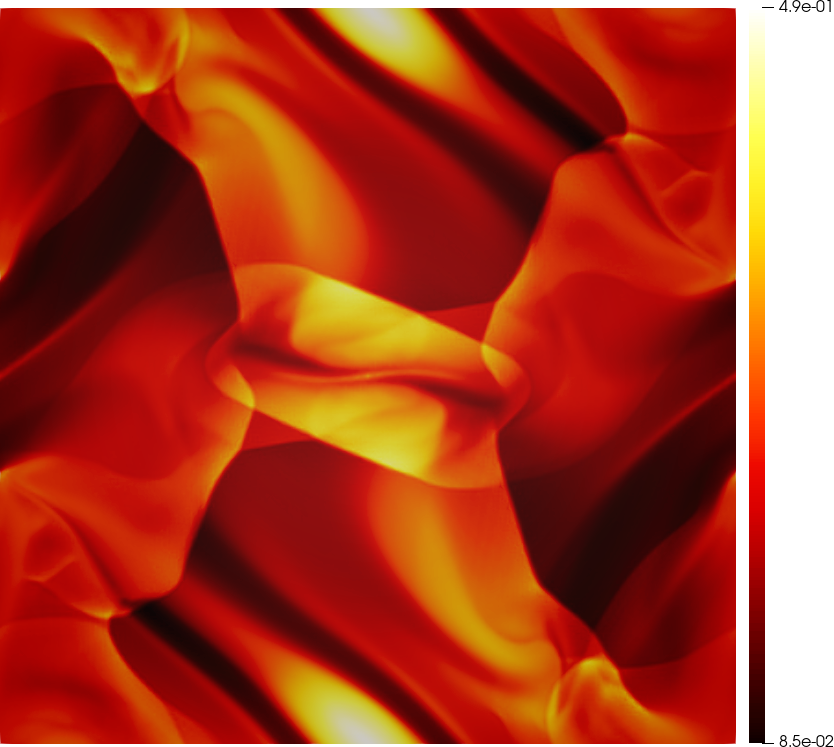}
         \caption{Density $\rho_h$}
     \end{subfigure}
     \hfill
     \begin{subfigure}{0.49\textwidth}
         \centering
         \includegraphics[width=\textwidth]{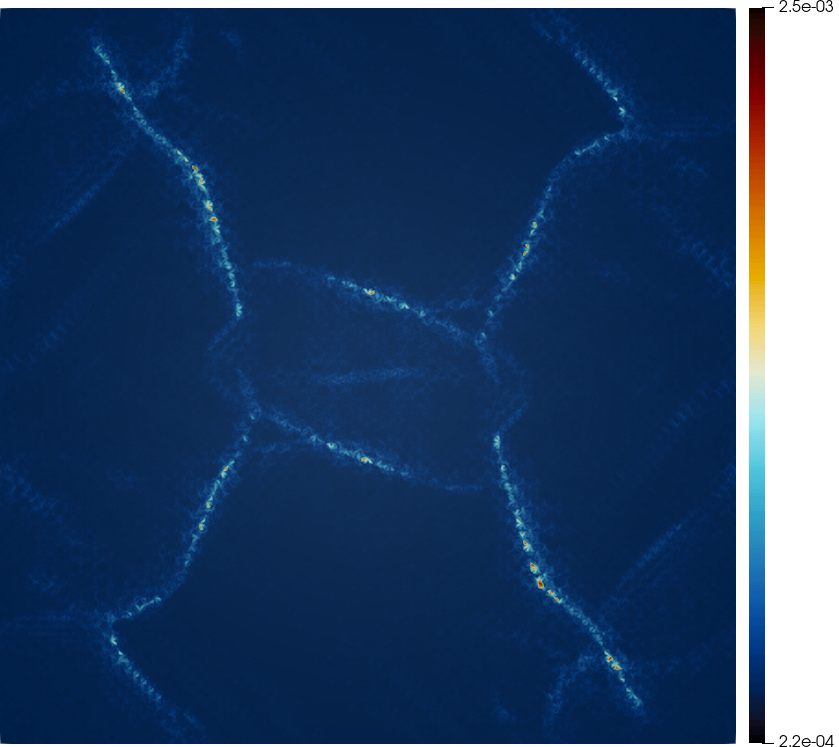}
         \caption{Entropy viscosity $\mu_h$}
     \end{subfigure}
     \hfill
     \begin{subfigure}{0.49\textwidth}
         \centering
         \includegraphics[width=\textwidth]{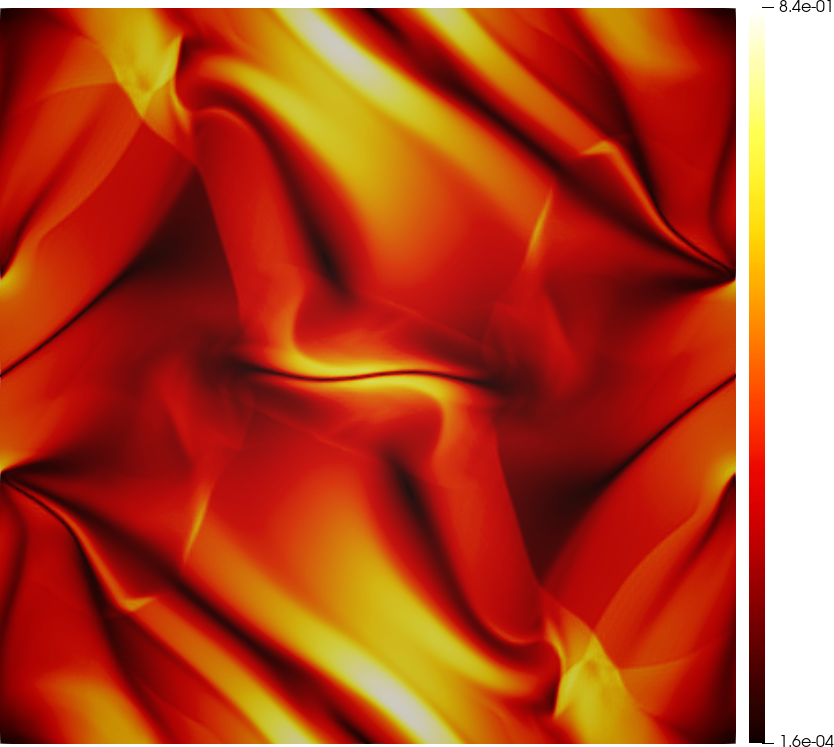}
         \caption{Magnetic pressure $\frac{1}{2}\bB_h^2$}
     \end{subfigure}
     \hfill
     \begin{subfigure}{0.49\textwidth}
         \centering
         \includegraphics[width=\textwidth]{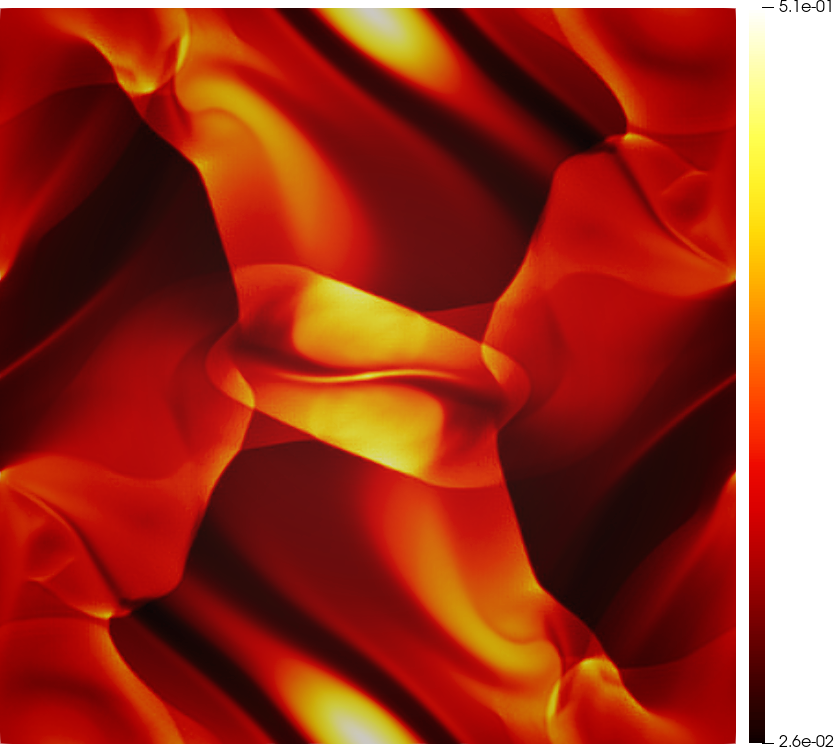}
         \caption{Thermodynamic pressure $p_h$}
     \end{subfigure}
     \caption{Entropy viscosity solution of the Orszag-Tang problem at time $\widehat t = 0.5$, 90000 $\polP_3$ nodes}
     \label{fig:orszag-tang_T05}
\end{figure}

\begin{figure}[h!]
     \centering
     \begin{subfigure}{0.49\textwidth}
         \centering
         \includegraphics[width=\textwidth]{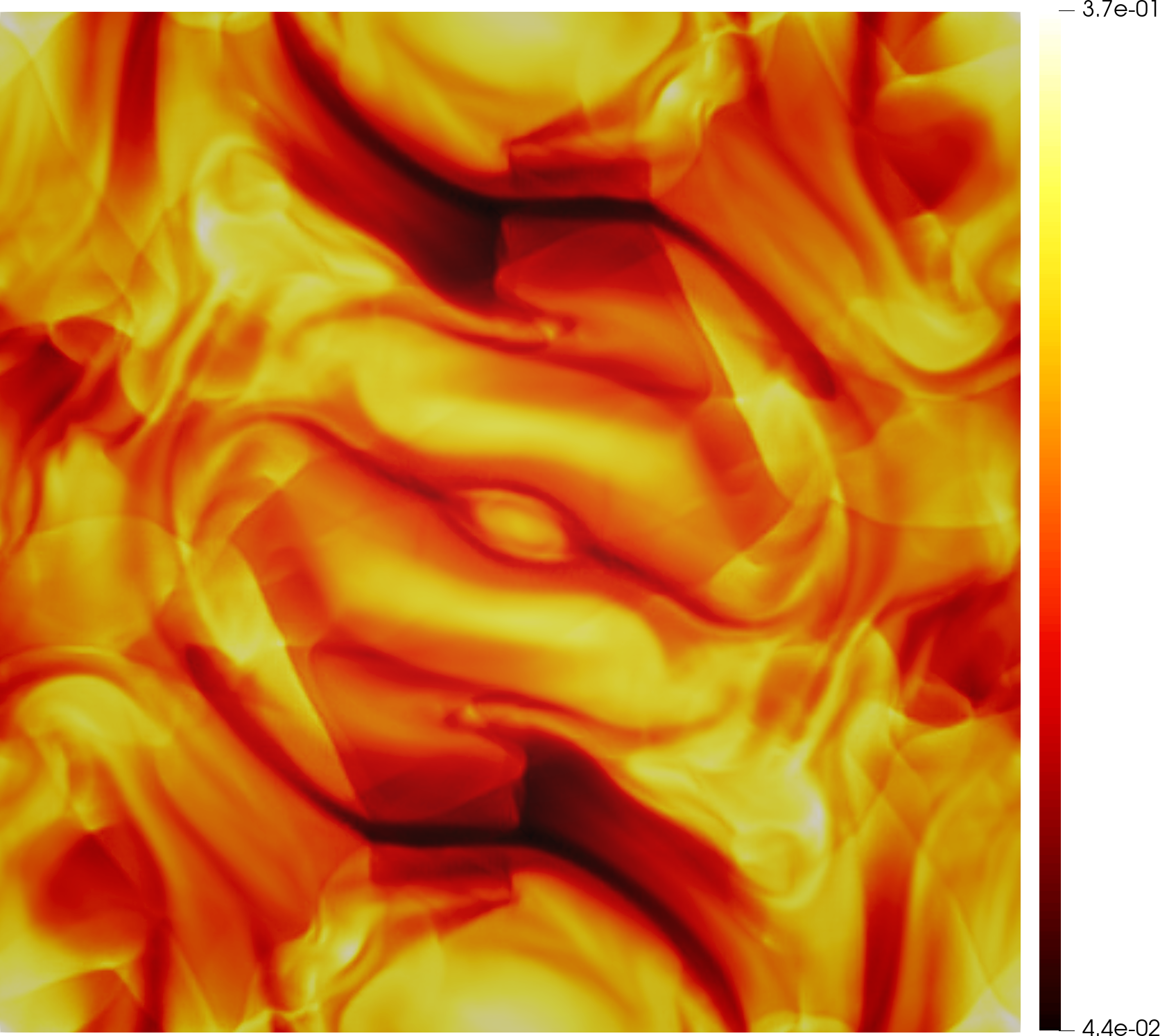}
         \caption{Density $\rho_h$}
     \end{subfigure}
     \hfill
     \begin{subfigure}{0.49\textwidth}
         \centering
         \includegraphics[width=\textwidth]{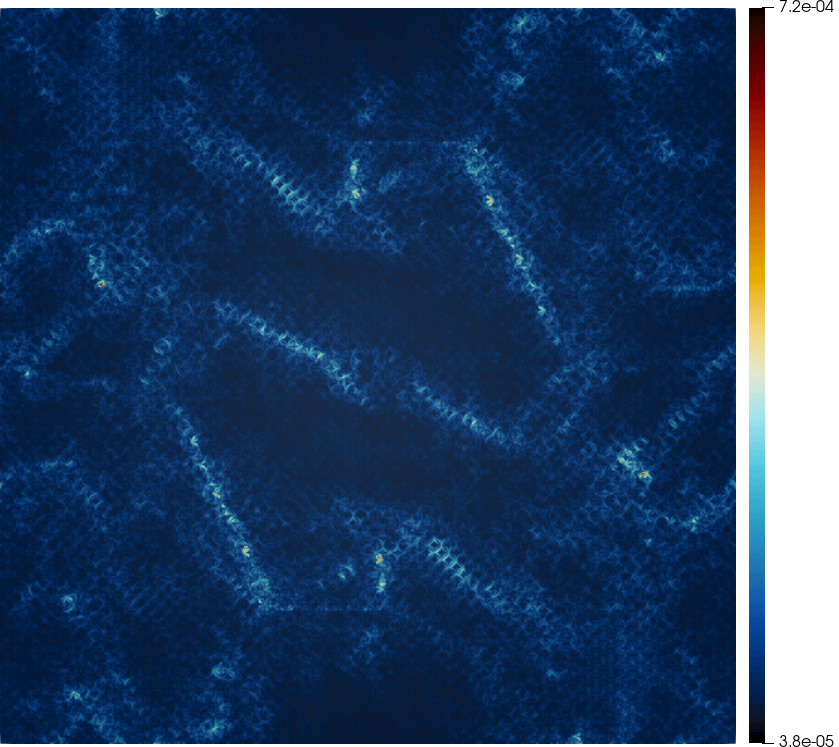}
         \caption{Entropy viscosity $\mu_h$}
     \end{subfigure}
     \hfill
     \begin{subfigure}{0.49\textwidth}
         \centering
         \includegraphics[width=\textwidth]{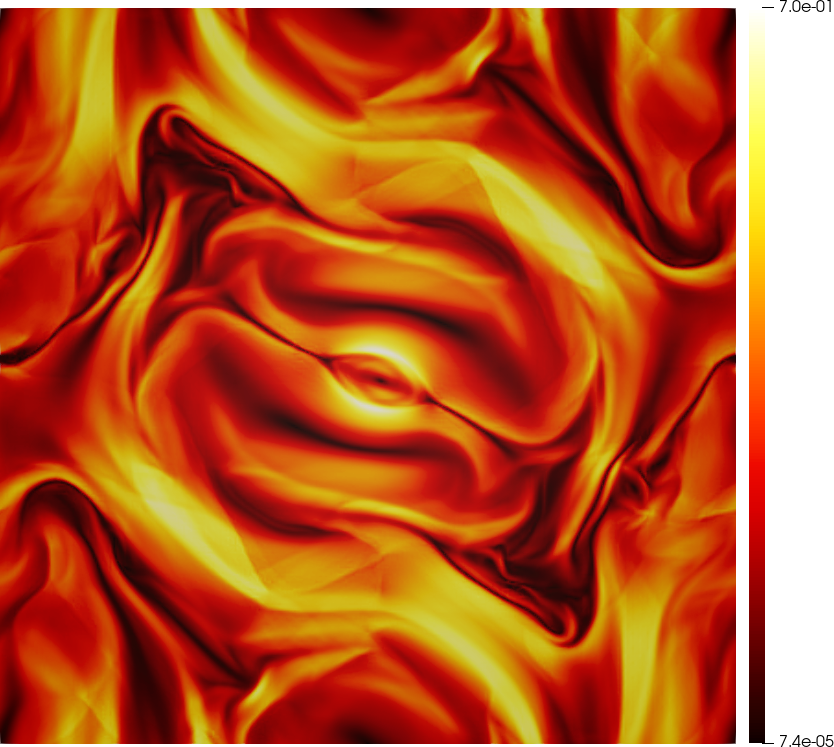}
         \caption{Magnetic pressure $\frac{1}{2}\bB_h^2$}
     \end{subfigure}
     \hfill
     \begin{subfigure}{0.49\textwidth}
         \centering
         \includegraphics[width=\textwidth]{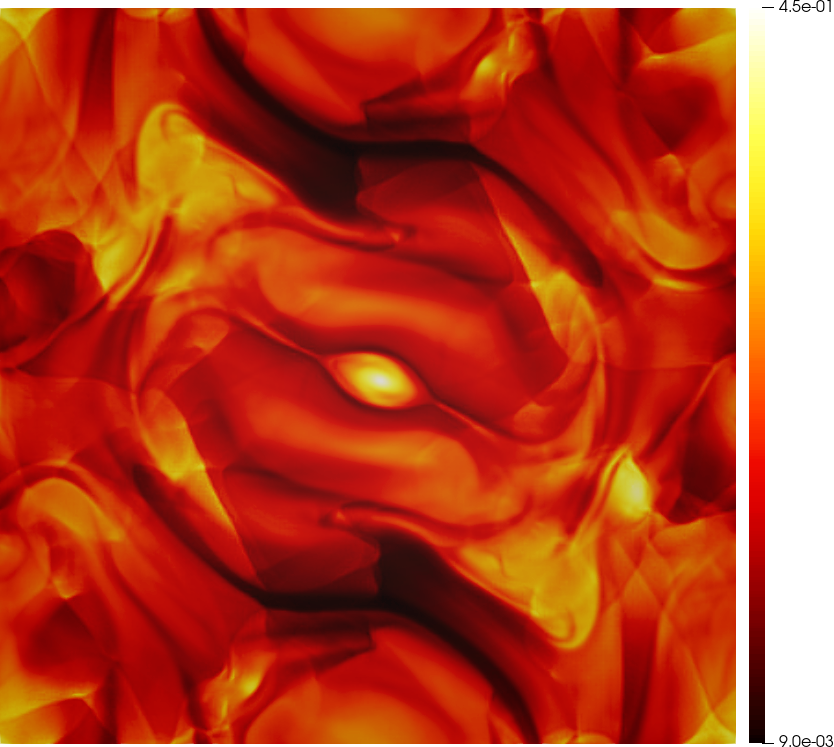}
         \caption{Thermodynamic pressure $p_h$}
     \end{subfigure}
     \caption{Entropy viscosity solution of the Orszag-Tang problem at time $\widehat t = 1.0$, 90000 $\polP_3$ nodes}
     \label{fig:orszag-tang_T1}
\end{figure}

\cblue{The next two benchmarks in Section~\ref{sec:rotor} and \ref{sec:blast} are more challenging for numerical MHD solvers. They contain strong discontinuities while having a low plasma-beta number $\beta:=\frac{2p}{\bB^2}$, small density, and small gas pressure. This condition may easily lead to breakdowns of the solvers due to the negativity of the pressure and density of the numerical solutions.}

\subsubsection{\cblue{MHD Rotor problem,} \citep{Balsara_et_al_1999}}\label{sec:rotor}

\cblue{The computational domain is the unit square $\Omega=[0,1]\times[0,1]$. The initial pressure and magnetic fields are uniform, $p_0=1, \bB_0 = \left(\frac{5}{\sqrt{4\pi}},0\right)$. For $\bx = (x,y) \in \Omega$, we define a radius $r$ as $r:=\sqrt{(x-0.5)^2+(y-0.5)^2}$. The density and velocity are defined as
\[
(\rho_0,\bu_0) = \begin{cases}
  \left(10, \left(\frac{2}{r_0}\left(\frac12-y\right),\frac{2}{r_0}\left(x-\frac12\right)\right)^\top\right) & \text{ if } r < r_0, \\
  \left(1+9f, \left(f\frac{2}{r}\left(\frac12-y\right),f\frac{2}{r}\left(x-\frac12\right)\right)^\top\right) & \text{ if } r_0 \leq r < r_1, \\
  (1, (0,0)^\top) & \text{ otherwise},
\end{cases}
\]
where $r_0=0.1,r_1=0.115$, and $f=\frac{r_1-r}{r_1-r_0}$. We use periodic boundaries in both $x-$ and $y-$ directions. A $\polP_3$ solution by the EV method at the final time $\widehat t = 0.15$ with 90000 $\polP_3$ nodes is shown in Figure~\ref{fig:rotor}. The computed solution agrees with the existing results, e.g., \cite{Balsara_et_al_1999,Tricco_2016}.}

\begin{figure}[h!]
     \centering
     \begin{subfigure}{0.49\textwidth}
         \centering
         \includegraphics[width=\textwidth]{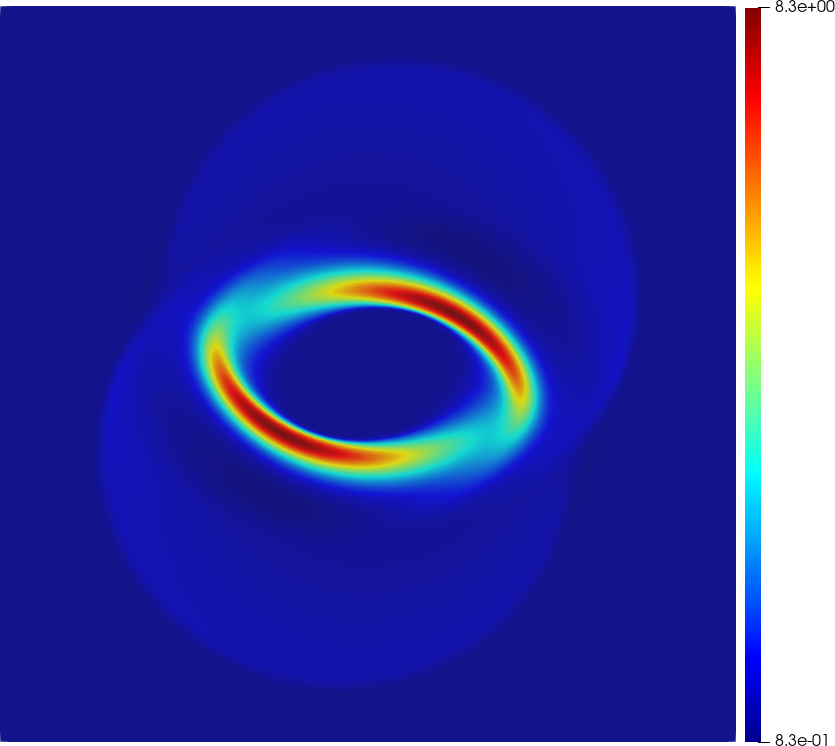}
         \caption{Density $\rho_h$}
     \end{subfigure}
     \hfill
     \begin{subfigure}{0.49\textwidth}
         \centering
         \includegraphics[width=\textwidth]{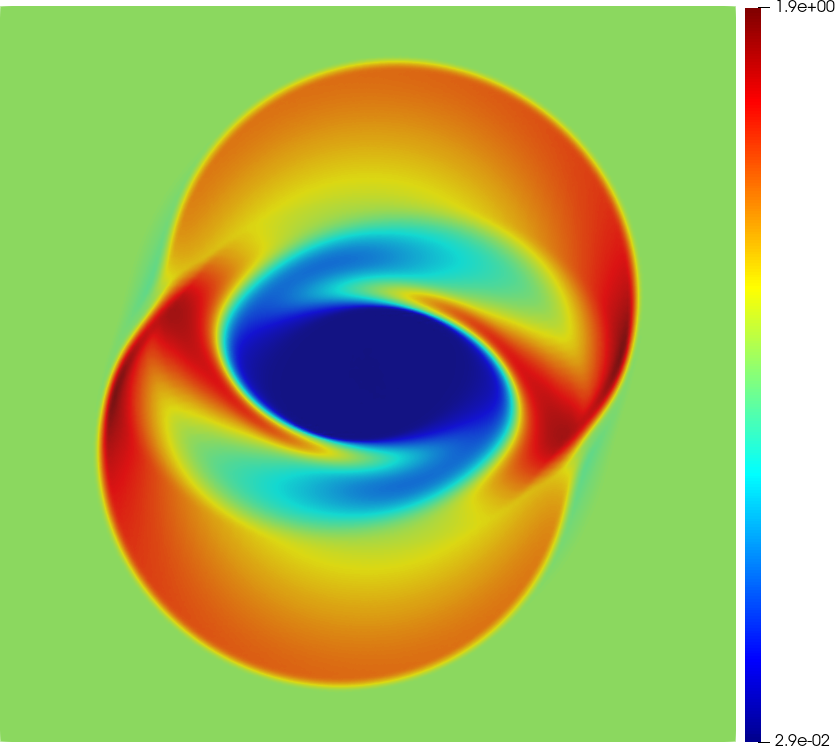}
         \caption{Thermodynamic pressure $p_h$}
     \end{subfigure}
     \hfill
     \begin{subfigure}{0.49\textwidth}
         \centering
         \includegraphics[width=\textwidth]{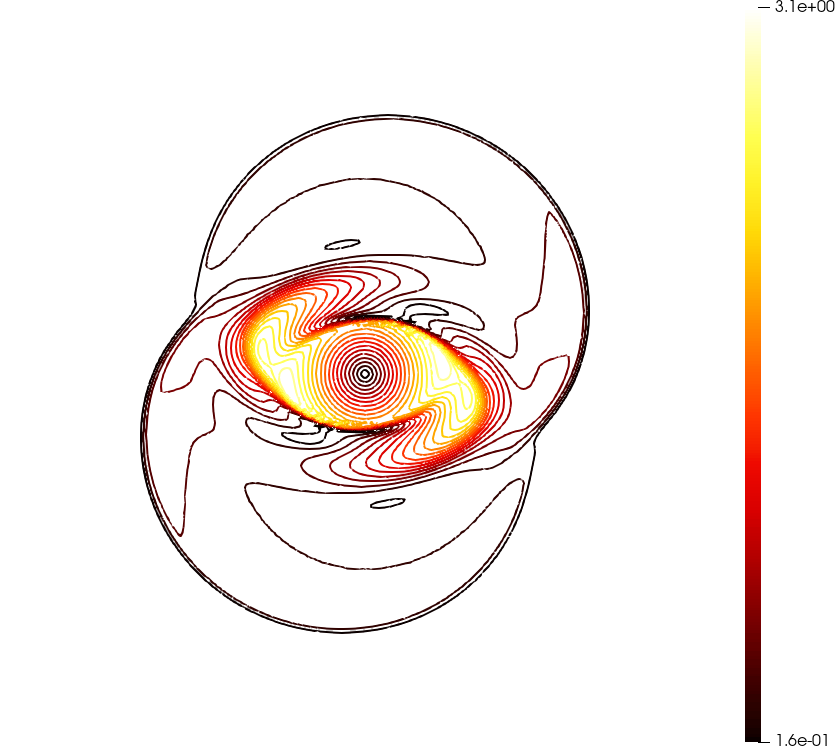}
         \caption{Contour plot of Mach number}
     \end{subfigure}
     \hfill
     \begin{subfigure}{0.49\textwidth}
         \centering
         \includegraphics[width=\textwidth]{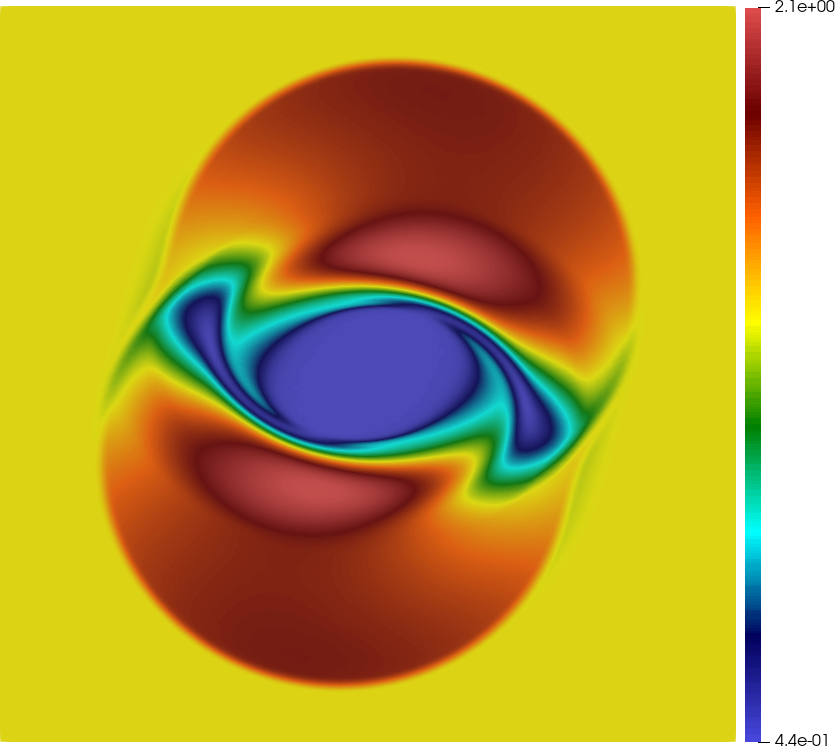}
         \caption{Magnetic pressure $\frac{1}{2}\bB_h^2$}
     \end{subfigure}
     \caption{Entropy viscosity solution to the Rotor problem at time $\widehat t = 0.15$, 90000 $\polP_3$ nodes}
     \label{fig:rotor}
\end{figure}

\subsubsection{\cpurple{MHD Blast problem, \citep{Balsara_et_al_1999}}}\label{sec:blast}

\cpurple{We consider the same setup as \cite{Wu_2018b}. See \cite{Wu_2018b} for a more detailed description of this problem.

  The computational domain is $\Omega=[-0.5,0.5]$. Initially, the density is uniform $\rho_0=1$; the magnetic field is uniform $\bB_0 = \left(\frac{100}{\sqrt{4\pi}},0\right)$; and the velocity $\bu_0$ is uniformly zero. For $\bx=(x,y)\in\Omega$, we define a radius $r:=\sqrt{x^2+y^2}$. There is a sharp jump in the initial pressure,
  \[
  p_0 = \begin{cases}
    1000 & \text{ if } r < 0.1, \\
    0.1 & \text{ otherwise}.
    \end{cases}
  \]
  We use periodic boundaries in both $x-$ and $y-$ directions. We show a numerical solution to this problem in Figure~\ref{fig:blast}. The computed solution agrees with the existing results, e.g., \cite{Balsara_et_al_1999,Wu_2018b}.}

\begin{figure}[h!]
     \centering
     \begin{subfigure}{0.49\textwidth}
         \centering
         \includegraphics[width=\textwidth]{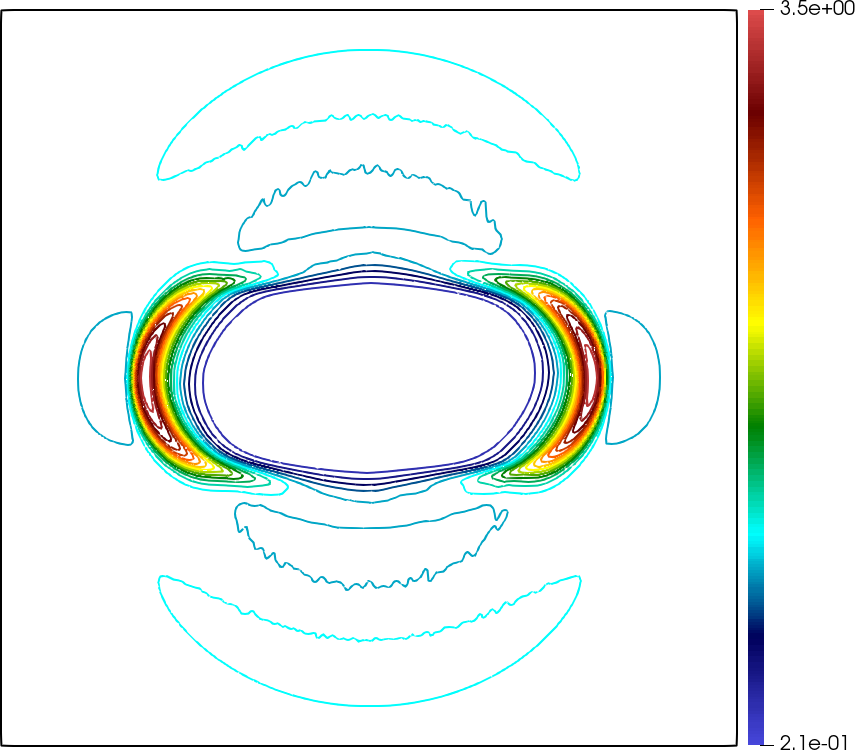}
         \caption{Density $\rho_h$}
     \end{subfigure}
     \hfill
     \begin{subfigure}{0.49\textwidth}
         \centering
         \includegraphics[width=\textwidth]{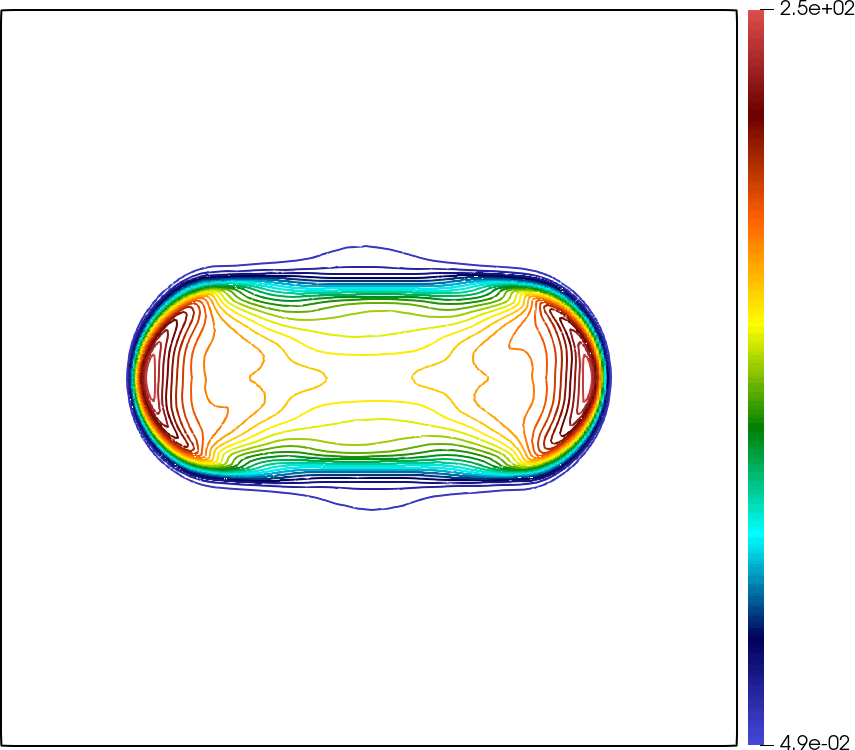}
         \caption{Thermodynamic pressure $p_h$}
     \end{subfigure}
     \hfill
     \begin{subfigure}{0.49\textwidth}
         \centering
         \includegraphics[width=\textwidth]{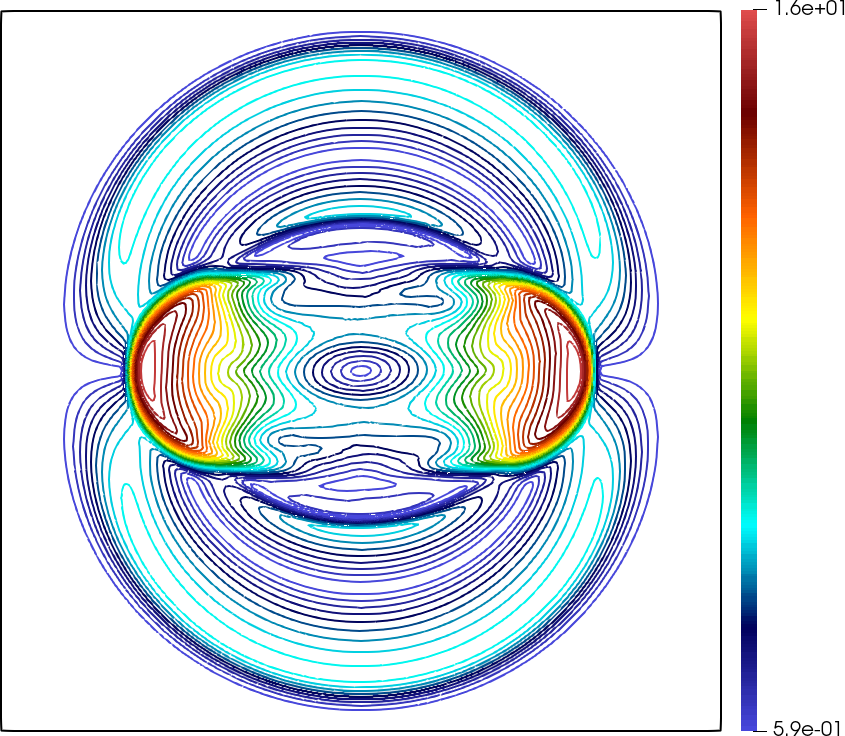}
         \caption{Velocity $|\bu_h|$}
     \end{subfigure}
     \hfill
     \begin{subfigure}{0.49\textwidth}
         \centering
         \includegraphics[width=\textwidth]{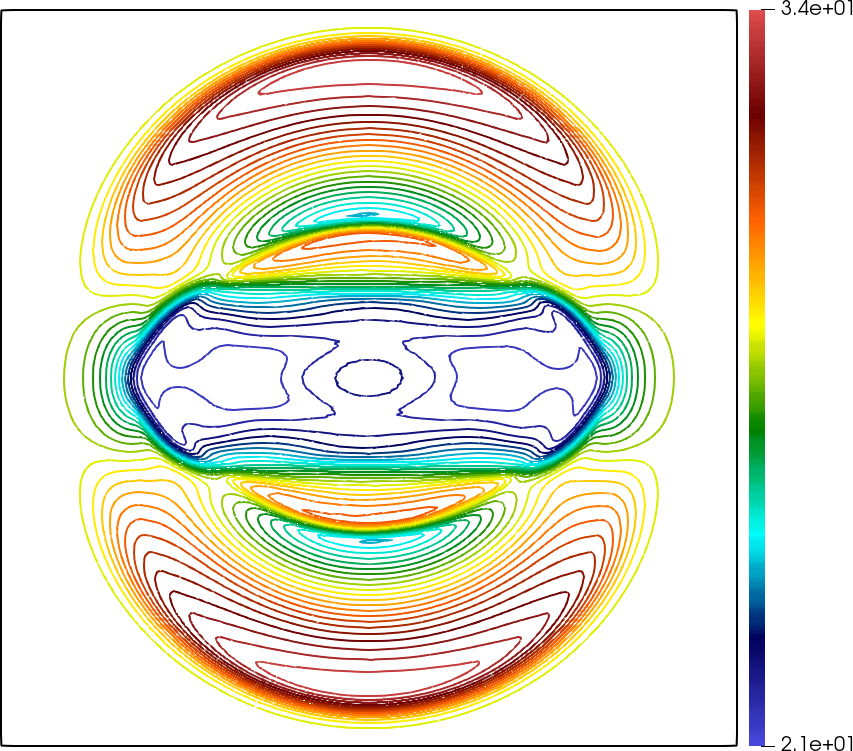}
         \caption{Magnetic pressure $\frac{1}{2}\bB_h^2$}
     \end{subfigure}
     \caption{Contour plots of the entropy viscosity solution to the MHD Blast wave problem at time $\widehat t = 0.01$, 90000 $\polP_3$ nodes}
     \label{fig:blast}
\end{figure}

\section{Conclusion}\label{Sec:conclusion}
The purpose of this paper has been to investigate the continuous and numerical properties of the monolithic parabolic regularization to the ideal MHD equations. Despite having no physical motivation, this regularization has been shown to be compatible with all the generalized entropies, fulfill minimum entropy principles, and positivity of pressure and density. As demonstrated in the CG context, the monolithic parabolic regularization also holds attractive numerical properties which can be related to the continuous analysis.

A known shortcoming with this regularization is that it is not Galilean and rotational invariant. This means that the regularization changes with a shift or rotation of the reference frame. However, since the viscosity coefficient to scale the regularization is artificial and vanishing, it is unclear how this violation would affect the numerical solutions. During the investigation in this paper, we have also made attempts to demonstrate the downsides of the monolithic flux but so far we have not witnessed any significant pitfalls.

In future works, we may investigate Galilean invariant viscous regularizations to the ideal MHD equations, or add control of nonzero divergence in the regularization. Allowing small violations of the divergence is numerically important, and has been done by several other methods such as \cite{Winters_2017,Wu_2018b}.


\newpage
\begin{appendices}
\section{Strict convexity of $-s$}\label{appendix:convexity_implication}
We prove that the strict convexity of $-s$ with respect to $\rho^{-1}$ and $e$ implies \eqref{eq:convex_entropy_assumption}. The properties \eqref{eq:convex_entropy_assumption} are borrowed from the Euler equations, see \citep{Guermond_Popov_2014} but for completeness, we want to give a short proof to demonstrate that many entropy properties still hold in the presence of electromagnetism.

The second inequality $s_{ee} < 0$ is rather obvious. We have $\p_{\rho^{-1}}s = -\rho^2s_\rho$. Since we know that $\p_{\rho^{-1}}(\p_{\rho^{-1}}s) < 0$, this leads to $\p_\rho(\rho^2s_\rho) < 0$. The Hessian matrix of $-s$ with respect to $\rho^{-1}$ and $e$ is also positive definite. Hence, the determinant of this matrix is positive, which says $\rho^2\p_\rho(\rho^2s_\rho)s_{ee}-\rho^4s_{\rho e}^2 > 0$. This inequality is equivalent to the third property in \eqref{eq:convex_entropy_assumption}.
\section{Proof of Theorem \ref{theorem:convex_entropy}}\label{appendix:convex_entropy}
We include a proof of Theorem \ref{theorem:convex_entropy} since this result, to the best of our knowledge, has not existed in the literature. Although the two inequalities in \eqref{eq:convex_entropy_inequalities} have been proved for the compressible Euler equation \citep{Harten_1998}, it is certainly nontrivial to show that they also hold for the MHD equations. In this appendix, we follow the same procedure of \cite{Harten_1998} to show \eqref{eq:convex_entropy_inequalities} for ideal MHD.
\begin{proof} Recall that our solution vector is $\bU := (\rho, \bbm, E, \bB)^\top$. An entropy $S =-\rho f(s)$ is said to be strictly convex if the Hessian matrix $S_{\bU\bU}$ is positive definite. The minus sign in $S$ differs from \cite{Harten_1998} because the entropy $s$ in the thermodynamic identity \eqref{eq:thermodynamic_identity} is defined with opposite sign to $s$ in \citep{Harten_1998}.

We express $\bu$ and $e$ under the components of $\bU$
\[
\bu = \frac{\bbm}{\rho}, \quad e = \frac{E}{\rho}-\frac{\bbm^2}{2\rho^2}-\frac{\bB^2}{2\rho}.
\]
We have
\[
\rho_{\bU} = (1, \bzero^\top, 0, \bzero^\top)^\top, \quad
\rho_{\bU\bU} = \polO,
\]
where $\bzero$ is the zero column vector and $\polO$ is the zero square matrix of appropriate sizes. Because $(\rho_{\bU\bU})_{ij} = 0, \;\forall i,j$, by differentiation rules,
\begin{align*}
  \left(S_{\bU\bU}\right)_{ij} & = \frac{\p^2 (-\rho f(s))}{\p \bU_i\p \bU_j}\\
  & =  \frac{\p}{\p \bU_i}\left(-f(s)\p_{\bU_j}\rho-\rho f'(s) \p_{\bU_j}s\right) \\
  & = -f'(s)\left(\p_{\bU_i}\rho\p_{\bU_j}s+\p_{\bU_j}\rho\p_{\bU_i}s+\rho\frac{\p^{2}s}{\p_{\bU_i}\p_{\bU_j}}\right)-\rho f''(s)\p_{\bU_i}s\p_{\bU_j}s.
\end{align*}
Therefore, in matrix form, $S_{\bU\bU}$ can be written as
\[
S_{\bU\bU} = -f'(s)\bH - f''(s)\rho s_\bU s_\bU^\top,
\]
where $\bH = \rho_{\bU}s_\bU^\top+s_{\bU}\rho_\bU^\top+\rho s_{\bU\bU}$. Using the chain rule \eqref{eq:s_chain_rule}, we can write $s_{\bU}$ as
\[
s_{\bU}=s_\rho\rho_\bU+s_ee_\bU.
\]
Consequently, we have
\[
\bH = (\rho s_{\rho\rho}+2s_\rho)\rho_\bU\rho_\bU^\top+\rho s_{e\rho}(\rho_\bU e_\bU^\top+e_\bU\rho_\bU^\top)+\rho s_{ee}e_\bU e_\bU^\top-s_e\bC,
\]
where $\bC = -(\rho e_{\bU\bU}+\rho_{\bU}e_{\bU}^\top+e_{\bU}\rho_{\bU}^\top)$. We now compute $e_\bU$, $e_{\bU\bU}$, and $\bC$ to determine $\bH$ and $s_\bU$. Straightforward calculation gives
\[
e_{\bU} = \begin{pmatrix}
  -\frac{E}{\rho^2}+\frac{\bbm^2}{\rho^3}+\frac{\bB^2}{2\rho^2}\\
  -\frac{\bbm}{\rho^2}\\
  \frac{1}{\rho}\\
  -\frac{\bB}{\rho}
\end{pmatrix}
= \frac{1}{\rho}
\begin{pmatrix}
  -\frac{E}{\rho}+\frac{\bbm^2}{\rho^2}+\frac{\bB^2}{2\rho}\\
  -\frac{\bbm}{\rho}\\
  1\\
  -\bB
\end{pmatrix}
= \frac{1}{\rho}
\begin{pmatrix}
  \frac{1}{2}\bu^2-e\\
  -\bu\\
  1\\
  -\bB
\end{pmatrix},
\]
\[
e_{\bU\bU} = \begin{pmatrix}
\frac{2e}{\rho^3}-\frac{3\bbm^2}{\rho^4}-\frac{\bB^2}{\rho^3} & \frac{2\bbm^\top}{\rho^3} & -\frac{1}{\rho^2} & \frac{\bB^\top}{\rho^2} \\
\frac{2\bbm}{\rho^3} & -\frac{1}{\rho^2}\polI & \bzero & \polO \\
-\frac{1}{\rho^2} & \bzero^\top & 0 & \bzero^\top \\
\frac{\bB}{\rho^2} & \polO & \bzero & -\frac{1}{\rho}\polI
\end{pmatrix},
\]
and
\[
\bC = \frac{1}{\rho} \begin{pmatrix}
  \bu^2 & -\bu^\top & 0 & \bzero^\top \\
  -\bu & \polI & \bzero & \polO \\
  0 & \bzero^\top & 0 & \bzero^\top \\
  \bzero & \polO & \bzero & \rho\polI
\end{pmatrix}.
\]
The key to this proof is to introduce the following invertible matrix
\[
\bP := \begin{pmatrix}
  1 & \bu^\top & \frac{1}{2} \bu^2 + e & \bzero^\top \\
  \bzero & \rho\polI & \rho\bu & \polO \\
  0 & \bzero^\top & \rho & \bzero^\top \\
  \bzero & \polO & \bB & \polI
\end{pmatrix}.
\]
Applying a left multiplication with $\bP$ and a right multiplication with $\bP^\top$ simplifies the structure of $S_{\bU\bU}$. We have
\[
\bP\rho_{\bU} = \begin{pmatrix}
  1 \\
  \bzero \\
  0 \\
  \bzero
\end{pmatrix},
\quad
\bP e_{\bU} = \begin{pmatrix}
  0 \\
  \bzero \\
  1 \\
  \bzero
\end{pmatrix},
\quad
\bP\bC\bP^\top = \begin{pmatrix}
  0 & \bzero^\top & 0 & \bzero^\top \\
  \bzero & \rho\polI & \bzero & \polO \\
  0 & \bzero^\top & 0 & \bzero^\top \\
  \bzero & \polO & \bzero & \polI
\end{pmatrix},
\]
\[
\bP\bH\bP^\top = \begin{pmatrix}
  2s_\rho+\rho s_{\rho\rho} & \bzero^\top & \rho s_{e\rho} & \bzero^\top \\
  \bzero & -\rho s_e\polI & \bzero & \polO \\
  \rho s_{e\rho} & \bzero^T & \rho s_{ee} & \bzero^\top \\
  \bzero & \polO & \bzero & -s_e\polI
\end{pmatrix}.
\]
Combining all the above derivations gives
\begin{adjustwidth}{0cm}{-0.6cm}
\[
\bP S_{\bU\bU}\bP^\top = -f'(s) \begin{pmatrix}
  2s_\rho+\rho s_{\rho\rho} & \bzero^\top & \rho s_{e\rho} & \bzero^\top \\
  \bzero & -\rho s_e\polI & \bzero & \polO \\
  \rho s_{e\rho} & \bzero^\top & \rho s_{ee} & \bzero^\top \\
  \bzero & \polO & \bzero & -s_e\polI
\end{pmatrix}
-
f''(s)\rho
\begin{pmatrix}
  s_\rho \\
  \bzero \\
  s_{e} \\
  \bzero
\end{pmatrix}
\begin{pmatrix}
  s_\rho &
  \bzero^\top &
  s_{e} &
  \bzero^\top
\end{pmatrix}.
\]
\end{adjustwidth}
This reveals that two eigenvalues of $\bP S_{\bU\bU}\bP^\top$ are $f'(s)\rho s_e$ and $f'(s)s_e$. Since $\rho >0, s_e > 0$, and the invertability of $\bP$, we conclude that $f'(s) > 0$ is a necessary condition for strict positivity of $S_{\bU\bU}$. Disregarding the two detached dimensions corresponding to the two known eigenvalues, the remaining dimensions require that the following matrix $\bM$ of size $2\times 2$ must be positive definite
\[
\bM := -f'(s) \begin{pmatrix}
  2s_\rho+\rho s_{\rho\rho} & \rho s_{e\rho} \\
  \rho s_{e\rho} & \rho s_{ee} \\
\end{pmatrix}
-
f''(s)\rho
\begin{pmatrix}
  s_\rho \\
  s_{e}
\end{pmatrix}
\begin{pmatrix}
  s_\rho &
  s_{e}
\end{pmatrix}.
\]
We then can proceed as in \cite[from (3.13)]{Harten_1998} to see that the positive definiteness of $\bM$ is equivalent to $\frac{f'(s)}{c_p} - f''(s) > 0$. This is because the same thermodynamic identity \eqref{eq:thermodynamic_identity} holds for the MHD equations as it holds for the Euler equations.
\end{proof}
\end{appendices}

\section*{Acknowledgement}
Some computations were performed on UPPMAX provided by the Swedish National Infrastructure for Computing (SNIC) under project number SNIC 2021/22-233.

\bibliographystyle{abbrvnat}
\bibliography{ref}

\end{document}